\documentclass[12pt,a4paper]{amsbook}
\usepackage[left=2.9cm,right=2.9cm,top=3cm,bottom=4.5cm]{geometry}
\usepackage[all]{xy}
\usepackage{txfonts}
\usepackage{hyperref}

\title[Chern Simons Theory]{Lie groups and Chern-Simons Theory}

\author{Benjamin Himpel}
\address{Reutlingen University\\
	Department of Informatics\\
	Alteburgstr. 150\\
	72762 Reutlingen}

\date{\today}

\makeatletter
{\end{list}}%
\makeatother

\usepackage{multirow}
\usepackage{mathrsfs}
\usepackage{graphicx}
\usepackage{pinlabel}
\usepackage{amssymb}

\usepackage{pb-diagram,pb-xy}
\usepackage{xy}

\usepackage[active]{srcltx} 

 \newtheorem{thm}{Theorem}[section]
 \newtheorem{cor}[thm]{Corollary}
 \newtheorem{lem}[thm]{Lemma}
 \newtheorem{prop}[thm]{Proposition}
 \newtheorem{exe}[thm]{Exercise}
\theoremstyle{definition}
 \newtheorem{defn}[thm]{Definition}
 \newtheorem{exa}[thm]{Example}
\newtheorem*{rem}{Remark}
 \numberwithin{equation}{section}

\makeatletter
  \let\c@lem=\c@thm
  \let\c@prop=\c@thm
  \let\c@defn=\c@thm
\makeatother

\newcommand{\kommentar}[1]{}
\newcommand{\later}[1]{}

\newcommand{\eval}[2]{\left.#1\right|_{#2}}

\def\co{\colon\thinspace}

\newcommand{\PD}{\operatorname{\rm PD}}

\newcommand{\Hom}{\operatorname{Hom}}
\newcommand{\End}{\operatorname{End}}
\newcommand{\Aut}{\operatorname{Aut}}
\newcommand{\im}{\operatorname{im}}
\newcommand{\hol}{\operatorname{hol}}
\newcommand{\ad}{\operatorname{ad}}
\newcommand{\Ad}{\operatorname{Ad}}
\newcommand{\re}{\operatorname{Re}}
\newcommand{\Maps}{\operatorname{Maps}}
\newcommand{\spec}{\operatorname{spec}}
\newcommand{\SF}{\operatorname{SF}}
\newcommand{\sgn}{\operatorname{sgn}}
\newcommand{\Sign}{\operatorname{Sign}}
\newcommand{\supp}{\operatorname{supp}}

\newcommand{\Id}{\operatorname{Id}}
\newcommand{\gr}{\text{gr}}
\newcommand{\proj}{\operatorname{proj}}

\newcommand{\coker}{\operatorname{coker}}
\newcommand{\Hol}{\operatorname{Hol}}
\newcommand{\Diff}{\operatorname{Diff}}

\newcommand{\ch}{\operatorname{ch}}
\newcommand{\cs}{\operatorname{cs}}
\newcommand{\CS}{\operatorname{CS}}
\newcommand{\PT}{\operatorname{PT}}
\newcommand{\hs}{\nabla\operatorname{grad}}
\newcommand{\grad}{\operatorname{grad}}
\newcommand{\vertical}{\operatorname{vert}}
\newcommand{\ind}{\operatorname{ind}}
\newcommand{\Hess}{\operatorname{Hess}}
\newcommand{\sign}{\text{sign}}

\newcommand{\ev}{\text{ev}}
\newcommand{\odd}{\text{odd}}

\newcommand{\wtilde}{\widetilde}
\newcommand{\tr}{\operatorname{{tr}}}

\newcommand {\tensor}{\otimes}

\newcommand{\cA}{{\mathcal A}}
\newcommand{\cB}{{\mathcal B}}
\newcommand{\cD}{{\mathcal D}}
\newcommand{\cF}{{\mathcal F}}
\newcommand{\cG}{{\mathcal G}}

\newcommand{\cL}{{\mathcal L}}
\newcommand{\cM}{{\mathcal M}}
\newcommand{\cN}{{\mathcal N}}
\newcommand{\cP}{{\mathcal P}}

\newcommand{\cZ}{{\mathcal Z}}
\newcommand{\cI}{{\mathcal I}}

\newcommand{\sO}{{\mathscr O}}

\newcommand{\Z}{{\bf Z}}
\newcommand{\R}{{\bf R}}
\newcommand{\C}{{\bf C}}
\newcommand{\N}{{\bf N}}
\newcommand{\Q}{{\bf Q}}

\newcommand{\frakg}{\mathfrak{g}}
\newcommand{\frakm}{\mathfrak{m}}
\newcommand{\frakso}{\mathfrak{so}}
\newcommand{\fraksu}{\mathfrak{su}}
\newcommand{\fraku}{\mathfrak{u}}
\newcommand{\frako}{\mathfrak{o}}
\newcommand{\frakh}{\mathfrak{h}}
\newcommand{\frakgl}{\mathfrak{gl}}
\newcommand{\frakk}{\mathfrak{k}}

\renewcommand{\phi}{\varphi}

\newcommand{\ep}{\varepsilon}

\newcommand{\la}{\langle}
\newcommand{\ra}{\rangle}

\newcommand{\msgp}[1]{\text{\it{#1}}}
\newcommand{\SL}{\msgp{SL}}
\newcommand{\SU}{\msgp{SU}}
\newcommand{\GL}{\msgp{GL}}
\newcommand{\SO}{\msgp{SO}}

\begin{document}

\begin{abstract}
Witten introduced classical Chern-Simons theory to topology in 1989, when he defined an invariant for knots in 3-manifolds by an integral over a certain infinite-dimensional space, which up to today have not been entirely understood. However, they motivated lots of interesting questions and results in knot theory and low-dimensional topology, as well as the development of entirely new fields.

These are lecture notes for a course in Lie groups and Chern-Simons Theory aimed at graduate students.
\end{abstract}

\maketitle

\tableofcontents


\chapter{Introduction}

These are notes for a course in an active field of geometric topology at the border to mathematical physics. They are aimed at graduate students, who know the basics of differentiable manifolds and differential forms. It is based on parts of:
\begin{itemize}
 \item Frank W. Warner: {\em Foundations of Differentiable Manifolds and Lie Groups} \cite{warner83}.
 \item Dan Freed: {\em Classical Chern-Simons Theory, Part 1} \cite{freed95}.
 \item Simons Donaldson and Peter Kronheimer: {\em The Geometry of Four-Manifolds} \cite{donaldson-kronheimer90}.
 \item Michael Bohn: On Rho invariants of fiber bundles \cite{bohn2009}
 \item Unpublished notes by Paul Kirk.
\end{itemize}

After an introduction to Lie groups and Lie algebras, principal bundles, connections and gauge transformations, we will carefully construct the Chern-Simons action and study the moduli space of its classical solutions. This yields Taubes' beautiful and influential description of Casson's invariant for homology 3-spheres via Chern-Simons theory \cite{taubes90}, for which we need such concepts as differential operators and spectral flow. This naturally leads to subjects like the eta invariant and the rho invariant on the one hand as well as the quantization of the Chern-Simons action and Witten's invariants on the other.

\section{An overview}

Let $G$ be a semi-simple Lie group and $\frakg$ its Lie algebra.  A connection $A$ on the trivial $G$-bundle over a closed oriented 3--manifold $M$ is a $\frakg$-valued 1--form, i.e. $A \in \Omega^1(M;\frakg)$. Chern-Simons theory is a quantum field theory in three dimensions, whose action is proportional to the Chern-Simons invariant given in \cite{chern-simons74}
$$cs(A) = \frac{1}{8\pi^2} \int_M \tr(A \wedge dA + \frac{1}{3} A\wedge [A
\wedge A]).$$

Witten introduced Chern-Simons theory to knot theory in 1989
\cite{witten89}, when he described for each integer level $k\in\Z$ an
invariant of a link $L=(L_j)$ in a 3--manifold $M$ (and a list of finite-dimensional representations $\rho_j$
of $G$ associated to $L_j$) as the
(non-rigorous) Feynman path integral
$$
Z_k(M,L) = \int_{\cA/\cG} e^{2\pi k i \; cs(A)} \prod_j\tr_{\rho_j}(\hol_A(L_j)) dA,
$$
where $\cA$ is the space of $G$--connections, $\cG$ is the space of
gauge transformations.
He interpreted these invariants using the axioms of topological quantum field theory (TQFT) as well as
via an asymptotic expansion ---the semiclassical approximation--- by using the method of stationary phase. There have
been several advances in understanding both approaches separately,
notably the rigorous construction of the TQFT version by
Reshetikhin and Turaev \cite{reshetikhin-turaev91}, the first computer
calculations and the refinement of the semiclassical approximation by Freed and Gompf \cite{freed-gompf91} and various recent work by Andersen,
Hansen, Ueno and Takata \cite{andersen2006_Asymptoticfaithfulness,andersen-ueno2007,
 andersen-ueno2007b, andersen-hansen2006, hansen2001, hansen2005,
 hansen-takata2002, hansen-takata2004} in the realm of the asymptotic
expansion conjecture for the Resehetikhin-Turaev invariants \cite[Section 7.2]{ohtsuki2002}. However, there
have been
relatively few developments relating the TQFT version to the semiclassical approximation apart from the direct
comparison for lens spaces for $G=\SU(2)$ and some torus bundles over $S^1$ by Jeffrey
\cite{jeffrey92},  where she explicitly analyzed the asymptotic
behaviour of the Reshetikhin-Turaev invariants and exhibited agreement
in the leading term with the semiclassical approximation. One promising
approach is to give a rigorous treatment of the Feynman path integral version using stochastical analysis via Fresnel integrals or Hida
distributions, if possible, this would provide a strong link between the TQFT version
and the semiclassical approximation. See \cite[Section
10.5.5]{albeverio-hoegh-krohn-mazzucchi2008} by Albeverio,
H{\o}egh-Krohn, and Mazzucchi as well as \cite{hahn2005} by Hahn
for an overview.

In \cite{witten89}, Witten proposed a physical interpretation of the
Jones polynomial using Chern-Simons theory,
and \cite{kashaev97} Kashaev made a remarkable prediction that the volume of a
hyperbolic knot $K$ in $S^3$ is
given by the limit of the colored Jones polynomial $J_N(K,q)$ of $K$.
This is known as the volume conjecture, and the precise statement
is that $\log|J_N(K,e^{2 \pi i/N})|/N$ limits to $\frac{1}{2 \pi}{\rm
Vol}(S^3 \setminus K)$ as $N \to \infty.$ (This conjecture can be extended to all knots by simply replacing hyperbolic
volume with simplicial volume on the right hand side.)

Using $\SL(2, {\bf C})$ Chern-Simons theory, Gukov conjectured in \cite{gukov2005} that
the asymptotic expansion of the colored Jones polynomial
$J_N(K,q)$ as $N \to \infty$ and $q \to 1$ should be equal to the partition
function of the $\SL(2,{\bf C})$ Chern-Simons knot
theory on $S^3 \setminus K$, and using this he derived a parameterized
version of the volume conjecture which postulates a relationship between the
family of limits of the colored Jones polynomial and the volume function ${\rm Vol(K,u)}$ on the character
variety of the knot complement. The paper \cite{gukov-murakami2006} provides further
evidence for the generalized conjecture by by studying some of the sub-leading terms in the asymptotic expansion of the colored Jones polynomial.

There is a strong connection between Witten's invariants and the so-called finite-type
invariants. The first definition of a finite-type 3-manifold invariant
was given by Ohtsuki \cite{ohtsuki96}, though by now many other useful and equivalent definitions
have been given (see e.g. \cite{garoufalidis-levine97, le-murakami-ohtsuki98}).
The first nontrivial finite type 3--manifold invariant can be identified with the $\SU(2)$ Casson invariant \cite{akbulut-mccarthy90,
 walker90} , and indeed all the terms in
the stationary phase expansion of the Chern-Simons path integral give rise to
finite type 3--manifold invariants. (See the papers \cite{rozansky97, kuperberg-thurston99} for a
mathematically rigorous definition of invariants arising from
a stationary phase approximation, as well as  proofs the invariants are of
finite-type.)

A different direction of research was initiated by Taubes in 1990 \cite{taubes90}. He laid the groundwork for new topological
invariants motivated by Chern-Simons theory by showing that the $\SU(2)$
Casson invariant  has
a gauge theoretical interpretation as the Euler characteristic of $\cA/\cG$ in the spirit of the Poincar\'e-Hopf theorem, where he views the Chern-Simons invariant as a $S^1$-valued Morse function on $\cA/\cG$. Taubes
realized that the Hessian of the Chern-Simons invariant and the odd
signature operator coupled to the same path of $\SU(2)$ connections
have the same spectral flow. Floer extended this idea around the
same time to instanton Floer homology \cite{floer88}, which has the $\SU(2)$
Casson invariant as its Euler characteristic, by viewing the
critical points of the Chern-Simons function as a $\Z/8$ graded
Morse complex.

The Seiberg-Witten revolution in gauge theory shifted attention from
instanton Floer homology to the more difficult
monopole Floer homology, and progress in either of these versions of Floer
homology has been vastly outpaced by the remarkable achievements
in the closely related but much more accessible Heegaard-Floer homology
introduced by Ozsv\'{a}th and Szab\'{o}.
Heegaard-Floer includes an entire package of
homology theories which have become powerful tools in low-dimensional
topology, and
there has been some success in  defining the corresponding pieces in
instanton Floer and monopole Floer.
Interest in these Floer homologies has been rekindled by the work \cite{kronheimer-mrowka2007_MonopolesAnd3Manifolds,kronheimer-mrowka2011_KnotHomologyFromInstantons,kronheimer-mrowka2010_KnotsSuturesAndExcision} of
Kronheimer and Mrowka;
in \cite{kronheimer-mrowka2007_MonopolesAnd3Manifolds} they give a thorough and rigorous treatment of monopole Floer theory,
in \cite{kronheimer-mrowka2011_KnotHomologyFromInstantons} they define knot Floer homology using instantons, and in \cite{kronheimer-mrowka2010_KnotsSuturesAndExcision} they
use  sutured Floer homology to give a new and more direct proof of property
$P$ for knots,
as well as the result, originally due to Ghiggini and Ni in the
Heegaard-Floer setting, that shows that monopole
Floer knot homology detects fibered knots.

In summary, Chern-Simons gauge theory is an active area of research, which
has produced a lot of interesting questions and results, as well as initiated the development
of entirely new fields.

\chapter{Background material}

\section{Lie groups}

Lie groups are groups which are also differentiable manifolds, in which the group operations are smooth. Well-known examples are the general linear group, the unitary group, the orthogonal group and the special linear group.

\begin{defn}
 A {\em Lie group} $G$ is a differentiable manifold with a group structure, such that the map $G \times G \to G$ given by $(g,h) \mapsto g h^{-1}$ is $C^\infty$.
\end{defn}

\begin{exe}
Let $G$ be a Lie group.
\begin{enumerate}
 \item Show that the maps $g \mapsto g^{-1}$ and $(g,h) \mapsto gh$ for $g,h\in G$ are $C^\infty$.
 \item Show that the identity component of $G$ is a Lie group.
\end{enumerate}
\end{exe}

\begin{exe}
Convince yourself, that the following spaces are Lie groups.
\begin{enumerate}
 \item $\R^n$ under vector addition.
 \item $\C \setminus \{ 0 \}$ under multiplication.
 \item $S^1 = U(1)$ under multiplication.
     \item $G \times H$ for Lie groups $G$ and $H$ with the product manifold structure and the group structure $(g_1,h_1)(g_2,h_2) \coloneqq (g_1g_2,h_1h_2)$.
 \item $T^n = \{ z \in \C^n \mid |z_1| = \ldots = |z_n| = 1\}$.
 \item The non-singular real matrices $Gl(n,\R)$ under matrix multiplication.
\end{enumerate}
\end{exe}

\section{Lie algebras}

\begin{defn}
 A {\em Lie algebra} $\frakg$ over $\R$ is a real vector space $\frakg$ together with a bilinear operator $[\cdot,\cdot]\co \frakg \times \frakg \to \frakg$ (the {\em bracket}) such that for all $x,y,z\in \frakg$,
\begin{enumerate}
 \item $[x,y] = -[y,x]$.\hfill{\em (anti-commutativity)}
 \item $[[x,y],z] +[[y,z],x] + [[z,x],y] = 0$.\hfill{\em (Jacobi identity)}
\end{enumerate}
\end{defn}

We will see that there is a Lie algebra associated to each Lie group, and that every connected, simply connected (i.e. the fundamental group $\pi_1(G)$ is trivial) Lie group is determined up to isomorphism by its Lie algebra.
The study of such Lie groups then reduces to a study of their Lie algebras.

\begin{exe}\label{liealgebras}
Convince yourself that the following vector spaces are Lie algebras.
 \begin{enumerate}
  \item \label{liebracketexple}The vector space of all smooth vector fields under the Lie bracket on vector fields.
  \item Any algebra (e.g. all $n\times n$ matrices $\mathfrak{gl}(n,\R)$) with
$$
[A,B] \coloneqq AB-BA.
$$
  \item $\R^3$ with the cross product as the bracket.
 \end{enumerate}
\end{exe}

\begin{defn}\label{phirelateddef}
 Let $\phi\co M \to N$ be $C^\infty$. Smooth vector fields $X$ on $M$ and $Y$ on $N$ are {\em $\phi$-related} if $d\phi \circ X = Y \circ \phi$, which is short for $d\phi_p(X_p) = Y_{\phi(p)}$ for all $p \in M$.
\end{defn}

\begin{exe}\label{phirelatedbracket}
 Let $\phi\co M \to N$ be $C^\infty$. Let $X_1$ and $X_2$ be smooth vector fields on $M$, and let $Y_1$ and $Y_2$ be smooth vector fields on $N$. If $X_i$ is $\phi$-related to $Y_i$, $i=1,2$, then $[X_1,X_2]$ is $\phi$-related to $[Y_1,Y_2]$.
\end{exe}

\begin{defn}
 Let $G$ be a Lie group and $g\in G$. {\em Left translation} $l_g$ and {\em right translation $r_g$ by $g$} are diffeomorphisms of $G$ given by
$$
l_g h = g h \quad r_g = h g, \quad \text{for all }h \in G.
$$
A vector field $X$ on $G$ is called {\em left-invariant} if $X$ is $l_g$-related to itself for all $g \in G$, that is
\begin{align*}
dl_g \circ X &= X \circ l_g,\\
\tag*{\text{which is short for}} \quad dl_g(X_h) & = X_{gh} \quad \text{for all } h \in G.
\end{align*}
\end{defn}

\begin{prop}\label{leftinvariantvf}
 Let $G$ be a Lie group and $\frakg$ its set of left-invariant vector fields.
\begin{enumerate}
 \item\label{iso} $\frakg$ is a real vector space, and the map $\alpha \co \frakg \to T_eG$ defined by $\alpha(X) = X(e)$ is an isomorphism of $\frakg$ with the tangent space $T_eG$ of $G$ at the identity. Consequently $\dim(\frakg) = \dim(G)$.
  \item Left invariant vector fields are smooth.
 \item The Lie bracket of two left-invariant vector fields is a left-invariant vector field.
 \item $\frakg$ is a Lie algebra under the Lie bracket operation of vector fields.
\end{enumerate}
\end{prop}

\begin{proof}
It is not difficult to see that $\frakg$ is a real vector space and that $\alpha$ is linear. Since $\alpha(X) = \alpha(Y)$ yields
$$
X_g = dl_g(X_e) = dl_g(\alpha(X)) = dl_g(\alpha(Y)) = dl_g(Y_e) = Y_g \quad \text{for each $g\in G$,}
$$
$\alpha$ is injective. In order to see the surjectivity of $\alpha$ let $v \in T_e G$ and define
$$
X_g = dl_g(v) \quad \text{for all } g \in G.
$$
Then $\alpha(X) = v$ and the left-invariance of $X$ is shown by
$$
X_{gh} = dl_{gh}(v) = dl_g dl_h (v) = dl_g(X_h) \quad \text{for all }g,h \in G.
$$
This proves part (1).

For (2) let $X \in \frakg$ and $f \in C^\infty(G)$. We need to show that $Xf \in C^\infty(G)$.
Let $\phi\co G\times G \to G$ be the smooth map given by multiplication $\phi(g,h) = gh$. Let $i^1_e$ and $i^2_g$ be the smooth maps $G \to G\times G$ given by $i^1_e (h) = (h,e)$ and $i^2_g (h) = (g,h)$. Let $Y$ be any smooth vector field on $G$ with $Y_e = X_e$. Then $[(0,Y)(f\circ \phi)] \circ i_e^1$ is smooth and
\begin{align*}
[(0,Y)(f\circ \phi)] \circ i_e^1 (g) & = (0,Y)_{(g,e)}(f\circ \phi)\\
&= 0_g(f\circ\phi\circ i^1_e) + Y_e (f\circ \phi \circ i^2_g)\\
&= X_e(f\circ \phi \circ i^2_g) = X_e(f\circ l_g)\\
&= dl_g(X_e)f =X_gf = Xf(g),
\end{align*}
which proves part (2).

Since by (2), left-invariant vector fields are smooth, their Lie brackets are defined and by Exercise \ref{phirelatedbracket} they are again left-invariant, which shows (3). (4) follows from (3) and part (\ref{liebracketexple}) of Exercise \ref{liealgebras}
\end{proof}

\begin{defn}
 We define the {\em Lie algebra of the Lie group $G$} to be the Lie algebra $\frakg$ of left-invariant vector fields on $G$.
\end{defn}

Often it will be convenient to think of the $T_e G$ as the Lie algebra of $G$ with the Lie algebra structure induced by the isomorphism $\alpha$ from Proposition \ref{leftinvariantvf}(\ref{iso}).

\begin{exa}
 The real line $\R$ is a Lie group under addition. The left-invariant vector fields are simply the constant vector fields $\{\lambda(\frac{d}{dr}) \mid \lambda \in \R\}$. The bracket of any two such vector fields is 0.
\end{exa}

\begin{exa}\label{GL} The Lie group $\GL(n,\R) = \det^{-1}(\R\setminus \{0\})$ is a (differentiable) submanifold of $\mathfrak{gl}(n,\R)$. $\GL(n,\R)$ has global coordinates $A_{ij}$ which assigns to each matrix $A$ the $ij$-th entry. Since $\det(A^{-1}) = \det(A)^{-1}$ and $\det(AB) = \det(A)\det(B)$ for any $A,B\in \GL(n,\R)$, the matrix $AB^{-1}$ is invertible with the $ij$-th entry $(AB^{-1})_{ij}$ being a rational function in the entries of $A$ and $B$ with non-zero denominator. Therefore $(A,B) \mapsto AB^{-1}$ is $C^\infty$.

Consider the isomorphism $\alpha\co\frakg \to T_e G$ from Proposition \ref{leftinvariantvf}(\ref{iso}). Since $T_e \GL(n,\R) = T_e \frakgl(n,\R)$ and the map $\beta \co T_e\mathfrak{gl}(n,\R) \to \frakgl(n,\R)$ given by $\beta(v)_{ij} \coloneqq v_{ij}$ is a (canonical) isomorphism, $\beta\circ\alpha$ induces a Lie algebra structure on $\frakgl(n,\R)$. We leave it as an exercise to show the bracket agrees with the usual $[A,B] = AB-BA$.
\end{exa}

\begin{exe}
 Let $\frakg$ be the Lie algebra of $\GL(n,\R)$. Prove that $\beta\circ\alpha$ from Example \ref{GL} induces the correct Lie algebra structure, i.e. show that for $X,Y \in \frakg$ we have $\beta\circ\alpha([X,Y]) = [\beta\circ\alpha(X),\beta\circ\alpha(Y)]$.
\end{exe}

Taking Example \ref{GL} as a starting point, we can create other examples: The non-singular matrices $\GL(n,\C)$ of all $n\times n$ complex matrices $\frakgl(n,\C)$ is a Lie group with Lie Algebra $\frakgl(n,\C)$. If $V$ is an $n$-dimensional real vector space, a basis of $V$ determines a diffeomorphism from $\End(V)$ to $\frakgl(n,\R)$ sending $\Aut(V)$ onto $\GL(n,\R)$. In this way $\End(V)$ is a Lie group with Lie algebra $\Aut(V)$. We get an analog example for complex case vector spaces.

The special linear group $\SL(n,\C)$, the unitary group $U(n)$, the special unitary group $\SU(n)$ are the most important examples for us.

\begin{exe} Assume that the bracket corresponds to the usual bracket on matrices as in example \ref{GL}, and show the following by differentiating the relations defining the Lie group:
\begin{enumerate}
 \item The Lie algebra $\mathfrak{sl}(n,\C)$ of $\SL(n,\C)$ consists of all $n\times n$--matrices with vanishing trace.
 \item The Lie algebra $\mathfrak{u}(n)$ of $U(n)$ consists of all skew-symmetric $n\times n$--matrices.
 \item The Lie algebra $\mathfrak{su}(n)$ of $\SU(n)$ consists of all skew-symmetric, traceless $n\times n$--matrices.
\end{enumerate}
\end{exe}

\section{Homomorphisms}

\begin{defn}
 A map $\phi\co G \to H$ is a {\em (Lie group) homomorphism} if $\phi$ is both $C^\infty$ and a homomorphism of groups. We call $\phi$ an {\em isomorphism} if $\phi$ is also a diffeomorphism. \kommentar{If $H = Aut(V)$ for some vector space $V$, or if $H = \GL(n,\C)$ or $\GL(n\R)$, then a homomorphism $\phi\co G \to H$ is called a {\em representation of the Lie group $G$}.}\\
If $\frakg$ and $\frakh$ are Lie algebras, a map $\psi\co \frakg \to \frakh$ is a {\em (Lie algebra) homomorphism} if it is linear and preserves brackets ($\psi[X,Y] = [\psi(X),\psi(Y)]$ for all $X,Y \in \frakg$). If $\psi$ is also a bijection, then $\psi$ is an {\em isomorphism}.\\
\end{defn}

Let $\phi\co G \to H$ be a homomorphism. Then $\phi$ maps the identity of $G$ to the identity of $H$, and the differential $d\phi$ of $\phi$ is a linear transformation of $\frakg = T_eG$ into $\frakh = T_eH$. Notice that by the natural identification between Lie algebras and left-invariant vector fields $d\phi(X)$ is the unique left-invariant vector field satisfying
\begin{equation}\label{identification}
 (d\phi(X))_e = d\phi_e(X_e).
\end{equation}

\begin{thm}\label{LieAlgebraHomo}
 Let $G$ and $H$  Lie groups with Lie algebras $\frakg$ and $\frakh$ respectively, and let $\phi\co G \to H$ be a homomorphism. Then
\begin{enumerate}
 \item $X$ and $d\phi(X)$ are $\phi$-related for each $X \in \frakg$.
 \item $d\phi\co \frakg \to \frakh$ is a Lie algebra homomorphism.
\end{enumerate}
\end{thm}

\begin{proof}
The left-invariant vector fields $d\phi(X)$ and $X$ are $\phi$-related, because
$$
 (d\phi (X))_{\phi(g)}  = dl_{\phi(g)} d\phi_e (X_e) = d(l_{\phi(g)} \circ \phi) X_e = d(\phi\circ l_g) X_e = d\phi (X_g).
$$
To show part (2), let $X,Y \in \frakg$. By Exercise \ref{phirelatedbracket} $[X,Y]$ is $\phi$--related to the left-invariant vector field $[d\phi(X), d\phi(Y)]$, that is
$$
(d\phi([X,Y]))_e \stackrel{\eqref{identification}}{=} d\phi_e([X,Y]_e) = [d\phi(X),d\phi(Y)]_{\phi(e)} = [d\phi(X),d\phi(Y)]_{e}
$$
Therefore using the identification of left-invariant vector fields with tangent vectors at the identity, $d\phi([X,Y]) = [d\phi(X),d\phi(Y)]$.\footnote{We had to be a bit careful not to mix up notation or use short cuts: By the way we had defined $\phi$-related in Definition \ref{phirelateddef}, we had to take the detour via the tangent space at $e$.}
\end{proof}

\section{Lie subgroups}

\begin{defn}
 $(H,\phi)$ is a {\em Lie subgroup} of the Lie group $G$ if
\begin{enumerate}
 \item $H$ is a Lie group;
 \item $(H,\phi)$ is a submanifold of $G$;
 \item $\phi\co H\to G$ is a Lie group homomorphism.
\end{enumerate}
$(H,\phi)$ is called a {\em closed subgroup} of $G$ if $\phi(H)$ is also a closed subset of $G$.

Let $\frakg$ be a Lie algebra. A subspace $\frakh \subset \frakg$ is a {\em subalgebra} if $[X,Y] \in \frakh$ for all $X,Y \in \frakh$.\footnote{A Lie algebra $\frakg$ or a subalgebra of $\frakg$ is not necessarily an algebra.}
\end{defn}

Let $(H,\phi)$ be a Lie subgroup of $G$, and let $\frakh$ and $\frakg$ be their respective Lie algebras. Then by Theorem \ref{LieAlgebraHomo} $d\phi$ yields an isomorphism between $\frakh$ and the subalgebra $d\phi(\frakh)$ of $\frakg$.

We will show in this section one of the fundamental theorems in Lie group theory, which asserts that there is a 1:1 correspondence between connected Lie subgroups of a Lie group and subalgebras of its Lie algebra. We first need the two following propositions.

\begin{prop}\label{generating}
 Let $G$ be a connected Lie group, and let $U$ be a neighborhood of $e$. Then
$$
G = \cup^\infty_{n=1} U^n
$$
where $U^n$ consists of all $n$--fold products of elemets of $U$.
\end{prop}

\begin{proof}
 Consider $V \coloneqq U \cup U^{-1}$ (where $U^{-1} = \{ g^{-1} \in U \mid g\in U\}$) and let $$H \coloneqq \bigcup_{n=1}^\infty V^n \subset \bigcup_{n=1}^{\infty} U^n.$$ $H$ is a subgroup of $G$ and open in $G$, since $h\in H$ implies $hV \in H$. Therefore each coset of $H$ is also open in $G$. Since $$H = G \setminus \bigcup_{g\in G,
gH \neq H} gH,$$
$H$ is also closed. Since $H$ is non-empty and $G$ is connected we get $G = H$.
\end{proof}

Recall that a {\em $k$--dimensional distribution} $\cD$ on a manifold $M$ is a choice of a $k$-dimensional linear subspace $\cD_p \subset T_p M$. A distribution is {\em smooth}, if it is locally spanned by smooth vector fields $X_1,\ldots,X_k$, or equivalently if $\cD$ is a smooth subbundle of $T M$. A smooth distribution is {\em involutive} if $[X,Y] \in \cD$ for any smooth vector fields $X,Y \in \cD$. It is {\em completely integrable}, if for each point $p\in M$ there is an integral manifold $N$ of $\cD$ passing through $p$, where $N$ is integral if $$N\stackrel{i}{\hookrightarrow}M \quad \text{and} \quad di_p(T_pN) = \cD_p \quad \text{for each } p\in N.$$ A smooth coordinate chart $(U,\phi)$ is {\em flat} for $\cD$, if $\phi(U) = U' \times U'' \subset \R^k\times \R^{n-k}$, and at points of $U$, $\cD$ is spanned by the first $k$ coordinate vector fields $\frac{\partial}{\partial x^1},\ldots,\frac{\partial}{\partial x^k}$. This implies that each slice of the form $x^{k+1}=c^{k+1},\ldots,x^{n}=c^n$ for constants $c^i$ is an integral manifold of $\cD$. Beware that the terminology is inconsistent among various books. However, we will show the famous Frobenius' Theorem which states that terminology is of no consequence.  Before that we need to recall a few definitions and prove a few lemmas.

\begin{defn}
 Let $M$ be a smooth manifold and let $X$ be a vector field on $M$.
\begin{enumerate}
 \item An {\em integral curve} of $X$ is a smooth curve $\gamma \co J \to M$ with $J \subset \R$ an open interval such that
$$
\gamma'(t) = X_{\gamma(t)} \quad \text{ for all } t\in J\kommentar{, \quad \text{ where }\gamma'(t) \coloneqq  d\gamma_t \left(\frac{d}{dt}\right)}.
$$
If $0 \in J$, then $\gamma(0)$ is the {\em starting point}.
\item Let $\theta^{(p)}\co \cD^{(p)} \to M$ be the {\em maximal integral} curve with starting point $p$.
\item The {\em flow of $X$} is the map $\theta\co \cD \to M$ given by $\theta_t(p) = \gamma_p(t)$ for $(t,p)$ in the domain $\cD \coloneqq  \{(\cD^{(p)},p) \mid p\in M\}$.
\item $X$ is called {\em complete} if the domain of its flow is $\R \times M$.
\end{enumerate}
\end{defn}

It follows from the theory of ordinary differential equations, that for any starting point integral curves exist and are uniquely determined by vector fields. For details see \cite[Chapter 17]{lee2003}. With some more work one can show the existence and uniqueness of the flow of $X$ \cite[Theorem 17.8]{lee2003}.
\begin{thm}[Fundamental Theorem on Flows]\label{fundthmflows}
 Let $X$ be a smooth vector field on a smooth manifold $M$. The flow $\theta\co \cD \to M$ of $X$ exists, is unique and has the following properties:
\begin{enumerate}
 \item For each $p\in M$, the curve $\theta^{(p)}\co \cD^{(p)} \to M$ is the unique maximal integral curve starting at $p$.
 \item If $s\in \cD^{(p)}$, then $\cD^{(\theta(s,p))}$ is the interval $\cD^{p}-s = \{ t-s \mid t\in \cD^{(p)} \}$.
 \item For each $t\in \R$, the set $M_t = \{ p\in  M \mid (t,p) \in \cD\}$ is open in $M$, and $\theta_t \co M_t \to M_{-t}$ is a diffeomorphism with inverse $\theta_{-t}$.
 \item For each $(t,p) \in \cD$, $(d\theta_t)_p X_p = X_{\theta_t(p)}$.
\end{enumerate}
\end{thm}

In order to get acquainted with these concepts, here are a few exercises.
\begin{exe}
Consider the vector field $X$ on $\R^2$ given by
$$
       X(x,y) = \frac{\partial}{\partial x} + (1+y^2)\frac{\partial}{\partial x}.
$$
Compute the flow of $X$ and its domain. Sketch its flow lines.
\end{exe}

\begin{exe}
Show that
$$
    X( x,y,z) = (-y,x,0)
$$
determines a smooth vector field on $S^2\subset \R^3$. What are the integral curves of $X$ and what is the flow of $X$?
\end{exe}

\begin{exe}
Let $M$ be a smooth manifold, $g$ a Riemannian metric on $M$ and $f\co M \to \R$ a differentiable function. Consider the gradient vector field $\operatorname{grad} f \in C^\infty(TM)$. Let $\gamma \co (-\ep,\ep) \to M$ an integral curve of $\operatorname{grad} f$. Show that $f\circ \gamma$ is monotonically increasing.
\end{exe}

\begin{lem}\label{relatedflow}
 Suppose $f\co M \to N$ is a smooth map, $X$ and $Y$ are vector fields on $M$ and $N$ respectively and $\theta$ and $\eta$ flows of $X$ and $Y$ respectively. Then $X$ and $Y$ are $f$--related if and only if $\eta_t \circ f (p)= f \circ \theta_t (p)$ for all $(t,p)$ in the domain $\cD$ of $\theta$:
\begin{equation}\label{relatedflowdiagram}
\begin{diagram}
 \node{M} \arrow{e,t}{f} \arrow{s,l}{\theta_t}\node{N} \arrow{s,r}{\eta_t}\\
\node{M} \arrow{e,b}{f} \node{N}
\end{diagram}
\end{equation}
\end{lem}

\begin{proof}
The commutativity of Diagram \eqref{relatedflowdiagram} is equivalent to
$$
\eta^{f(p)}(t) = f \circ \theta^{(p)}(t) \quad \text{for all } (t,p) \in \cD.
$$
If $X$ and $Y$ are $f$--related, then the curvce $\gamma\co \cD^{(p)}\to N$ given by
$$
\gamma = f \circ \theta^{(p)}
$$
satisfies
$$ \gamma'(t) = (f \circ \theta^{(p)})'(t) = df (\theta^{(p)}{}'(t))=df (X_{\theta^{(p)}(t)}) = Y_{f\circ \theta^{(p)}(t)} = Y_{\gamma(t)} \quad \text{for all }t\in \cD^{(p)}
$$
Therefore $\gamma$ is an integral curve with starting point $f \circ \theta^{(p)}(0) = f (p)$. By uniqueness of integral curves, the maximal integral curve $\eta^{f(p)}$ must be defined on $\cD^{(p)}$ and $\gamma = \eta^{f(p)}$. Therefore Diagram \eqref{relatedflowdiagram} commutes.

Conversely, if Diagram \eqref{relatedflowdiagram} commutes, then $X$ and $Y$ are $f$--related by
\[
 df_p(X_p)  = df_p (\theta^{(p)}{}'(0)) = (f\circ \theta^{(p)})'(0) = \eta^{(f(p))}{}'(0) = Y_{f(p)} \quad\text{for all }p\in M.\qedhere
\]
\end{proof}

The proof of the following exercise is in the spirit of Lemma \ref{relatedflow}.

\begin{exe}\label{commutinginvariant}
Let $X$ and $Y$ be smooth vector fields on $M$ and $\theta$ and $\eta$ the flows of $X$ and $Y$ respectively. Then $[X,Y]$ is equivalent to $X$ being invariant under $\eta$.
\end{exe}

\begin{thm}[Frobenius' Theorem]\label{FrobeniusTheorem} The following are equivalent:
\begin{enumerate}
 \item There exists a flat chart for $\cD$ at every point of $M$;
 \item $\cD$ is completely integrable;
 \item $\cD$ is involutive.
\end{enumerate}
\end{thm}

\begin{proof}
Certainly for any $p \in M$, if there exists a flat chart for $\cD$, and the slice of the form $x^{k+1}=p^{k+1},\ldots, x^{n}=p^n$ is an integral manifold through $p$. Therefore existence of a flat chart at every point implies complete integrability.

To see that complete integrability implies involutivity, consider two smooth vector fields $X$ and $Y$ lying in a completely integrable distribution $\cD$ on $M$ and $p \in M$. Let $(H,i)$ be an integral manifold of $\cD$ through $p$ and suppose $i(q) = p$. Since $di \co TN \to \cD|_{i(N)}$ is a bundle isomorphism, there exist smooth vector fields $\tilde X$ and $\tilde Y$ on $N$, such that $d i \circ \tilde X  = X \circ i$ and $di \circ \tilde Y = Y \circ i$. By Exercise \ref{phirelatedbracket} $[X,Y]$ and $[\tilde X,\tilde Y]$ are $i$--related. Therefore $[X,Y]_p  = d i_q ([\tilde X, \tilde Y]_q) \in \cD_p$, and $\cD$ is involutive.

Now assume that $\cD$ is a $k$--dimensional involutive distribution, $\dim M = n$ and $p\in M$. Since the existence of a flat chart at $p$ is a local question, we may replace $M$ by an open set $U\subset \R^n$ with $0 = p \in U$ and assume $$\cD_p \oplus \operatorname{span}\left\{\left.\tfrac{\partial}{\partial x^{k+1}}\right|_p,\ldots \left.\tfrac{\partial}{\partial x^{n}}\right|_p\right\} = T_p\R^n.$$ Choose a smooth local frame $Y_1,\ldots,Y_k$ of $\cD$ in $U$. Let $\pi\co \R^n \to \R^k$ be the projection onto the first $k$ coordinates. This induces a smooth bundle map $d\pi \co TU \to T\R^k$ with
$$
d\pi\left(\sum_{i=1}^{n}a^i\left.\frac{\partial}{\partial x^i}\right|_q\right) = \sum_{i=1}^{k}a^i\left.\frac{\partial}{\partial x^i}\right|_{\pi(q)}
$$
By our choice of coordinates $\left.d\pi_p\right|_{\cD_p}$ is bijective, by continuity $\left.d\pi\right|_{\cD_q}$ is bijective\footnote{Consider e.g. the determinant map on the matrix with respect to a local frame.} for $q$ in an open neighborhood $U'$ of $p$. Furthermore, $\left.d\pi\right|_{\cD}$ is a smooth bundle map and therefore the matrix entries of  $\left.d\pi\right|_{\cD}$ with respect to our local frame $(Y_i)$ are smooth functions of $q\in U'$, thus so are the matrix entries of $\left(\left.d\pi\right|_{\cD_q}\right)^{-1}\co T \R^{k}_{\pi(q)} \to \cD_q$. Define a smooth local frame $X_1,\ldots,X_k$ for $\cD$ in $U'$ by
$$
\left.X_i\right|_q = \left(\left.d\pi_q\right|_{\cD_q}\right)^{-1}\left.\frac{\partial}{\partial x^i}\right|_{\pi(q)}.
$$
The vector fields $X_i$ and $\left.\tfrac{\partial}{\partial x^i}\right|_{\pi(q)}$ are $\pi$--related since
$$
\left.\frac{\partial}{\partial x^i}\right|_{\pi(q)} = \left(\left.d\pi_q\right|_{\cD_q}\right)\left.X_i\right|_q = d\pi_q\left.X_i\right|_q.
$$
Therefore by Exercise \ref{phirelatedbracket}
$$
d\pi([X_i,X_j]_q) = \left[\frac{\partial}{\partial x^i},\frac{\partial}{\partial x^j}\right]_{\pi(q)} = 0.
$$
Since $[X_i,X_j]_q \in \cD_q$ by involutivity and $d\pi$ is injective on each fiber of $\cD$, $[X_i,X_j]_q =0$.

Now let $\theta_i$ be the flow of $X_i$ for $i=1,\ldots,k$. There exists an $\ep>0$ and a neighborhood $W$ of $p$ such that the composition $(\theta_1)_{t_1}\circ\ldots\circ (\theta_k)_{t_k}$ is defined on $W$ and maps $W$ into $U'$ whenever $|t_1|,\ldots,|t_k|$ are less than $\ep$. Let
$$
S\coloneqq  \{(u^{k+1},\ldots,u^n) \mid (0,\ldots,0,u^{k+1},\ldots,u^n) \in W\}
$$
and define $\psi\co (-\ep,\ep)^k \times S \to U$ by
$$
\psi (u^1,\ldots,u^n) \coloneqq  (\theta_1)_{u^1}\circ\ldots\circ(\theta_k)_{u^k} (0,\ldots,0,u^{k+1},\ldots,u^n)
$$
By Exercise \ref{commutinginvariant} $X_i$ is invariant under the flow of $X_j$, in other words $X_i$ is $(\theta_j)_s$--related to itself. Therefore, since by Lemma \ref{relatedflow} all the flows $\theta_i$ commute with each other, for $i\in\{1,\ldots,k\}$ and any $u_0 \in W$ we have
\begin{align*}
\lefteqn{\left(d\psi_{u_0} \left.\frac{\partial}{\partial u^i}\right|_{u_0}\right)f}\\
&= \left.\frac{\partial}{\partial u^i}\right|_{u_0} f(\psi(u^1,\ldots,u^n))\\
 &= \left.\frac{\partial}{\partial u^i}\right|_{u_0} f((\theta_1)_{u^1}\circ\ldots\circ(\theta_k)_{u^k} (0,\ldots,0,u^{k+1},\ldots,u^n)) \\
 &= \left.\frac{\partial}{\partial u^i}\right|_{u_0} f((\theta_i)_{u^i}\circ(\theta_1)_{u^1}\circ\ldots\circ(\theta_{i-1})_{u^{i-1}}\circ(\theta_{i+1})_{u^{i+1}}\ldots\circ(\theta_k)_{u^k} (0,\ldots,0,u^{k+1},\ldots,u^n))\\
 &= df\left(\left.\frac{\partial}{\partial u^i}\right|_{u_0} (\theta_i)_{u^i}\right)((\theta_1)_{u^1}\circ\ldots\circ(\theta_{i-1})_{u^{i-1}}\circ(\theta_{i+1})_{u^{i+1}}\ldots\circ(\theta_k)_{u^k} (0,\ldots,0,u^{k+1},\ldots,u^n))\\
& = df(\left.X_i\right|_{\psi(u_0)}) = \left.X_i\right|_{\psi(u_0)}(f)
\end{align*}
Then for $i=1,\ldots,k$ we have
$$
d\psi \left.\frac{\partial}{\partial u^i}\right|_{0} = (X_i)_p.
$$
On the other hand, since
$$
\psi(0,\ldots,0,u^{k+1},\ldots,u^n) = (0,\ldots,0,u^{k+1},\ldots,u^n),
$$
we get
$$
d\psi \left.\frac{\partial}{\partial u^i}\right|_{0} = \left.\frac{\partial}{\partial x^i}\right|_p \quad \text{for } i=k+1,\ldots,n.
$$
Therefore, $d\psi$ takes $(\left.\tfrac{\partial}{\partial u^1}\right|_0,\ldots,\left.\tfrac{\partial}{\partial u^n}\right|_0)$ to $(X_1|_p,\ldots,X_k|_p,\left.\tfrac{\partial}{\partial x^1}\right|_p,\ldots,\left.\tfrac{\partial}{\partial x^n}\right|_p)$, which is also a basis. By the inverse function theorem, $\psi$ is a diffeomorphism in a neighborhood of 0, and $\phi = \psi^{-1}$ the flat choordinate chart.
\end{proof}

\begin{prop}\label{smoothness}
 Suppose that $f\co M \to N$ is $C^\infty$, that $(H,i)$ is a $d$-dimensional integral manifold of a $k$--dimensional involutive distribution $\cD$ on $N$, and that $f$ factors through $(H,i)$, that is, $f(M) \subset i(H)$. Let $f_0 \co M \to H$ be the (unique) mapping such that $i \circ f_0 = f$.
$$
\begin{diagram}
 \node{M} \arrow{e,t}{f} \arrow{se,b,..}{f_0}\node N\\
\node{}\node{H}\arrow{n,r}{i}
\end{diagram}
$$
Then $f_0$ is $C^\infty$.
\end{prop}

\begin{proof}
By the injectivity of $i$ there is a unique $q \in H$ for all $p\in M$ such that $i (q) = f(p)$, i.e. there is a unique map $f_0$ such $i\circ f_0 = f$. Let $p \in M$ be arbitrary and set $q = f_0(p) \in H$. Let $(y^1,\ldots,y^n)$ be flat coordinates for $\cD$ on a neighborhood $U$ of $i(q)$ and let $\pi\co U \to \R^{n-k}$ be given by $\pi = (y^{k+1},\ldots,y^n)$. Since $i$ is continuous, $i^{-1}(U)$ is open in $H$ and therefore consists of a countable number of connected components, each of which is open in $H$. Because $\cD$ is spanned by the first $k$ coordinate vector fields in $U$, this implies that $\pi\circ \gamma$ is constant for any smooth path $\gamma$ in a connected component of $i(i^{-1}(U)) = i(H) \cap U$. Therefore each connected component of $i(H) \cap U$ is a subset of a slice and $\pi(i(H) \cap U)$ consists of a countable number of points. Choose smooth coordinates $(x^i)$ for $M$ on a connected neighborhood $V$ of $p$ such that $f(V) \subset U$ and $i^{-1}$. Since $f(V) \subset i(H) \cap U$, $\pi\circ f(V)$ is also countable.
Since $V$ is connected, the coordinate functions are actually constant, and thus $f(V)$ lies in a single slice. On this slice $(y^1,\ldots,y^k)$ are smooth coordinates, so $f_0$ has the local coordinate representation $(y^1,\ldots,y^k)\circ f(x) = (f^1_0(x),\ldots,f^k_0(x))$, which is smooth.
\end{proof}

\begin{thm}
 Let $G$ be a Lie group with Lie algebra $\frakg$, and let $\tilde \frakh \subset \frakg$ be a subalgebra. Then there is a unique connected Lie subgroup $(H,\phi)$ of $G$ such that $d\phi(\frakh) = \tilde \frakh$.
\end{thm}

\begin{proof}
Let $d \coloneqq \dim \tilde\frakh$. Consider the smooth $d$--dimensional distribution $\cD$ given by
$$
\cD(g) \coloneqq \{X(g) \mid X \in \tilde\frakh\} \quad \text{for every }g\in G.
$$
$\cD$ is globally spanned by a basis $X_1, \ldots, X_d$ for $\tilde\frakh$. $\cD$ is involutive, for if $X$ and $Y$ are vector fields lying in $\cD$, then there are $C^\infty$ functions $\{a_i\}$ and $\{b_i\}$ on $G$ such that $X = \sum a_i X_i$ and $Y = \sum b_i X_i$, and therefore
$$
[X,Y] = \sum_{i,j=1}^d\{a_ib_j[X_i,X_j] + a_i X_i (b_i) Y_j - b_j Y_j(a_i) X_i\},
$$
which is a vector field in $\cD$ since $\tilde\frakh$ is a subalgebra of $\frakg$.

Let $(H,\phi)$ be a maximal connected integral manifold of $\cD$ through $e$. Let $g \in \phi(H)$. Since $\cD$ is invariant under left translations, $(H,l_{g^{-1}}\circ \phi)$ is also an integral manifold of $\cD$ through $e$ with $l_{g^{-1}} \circ \phi(H) \subset \phi(H)$. Therefore, $g,h\in \phi(H)$ implies $g^{-1}h\in\phi(H)$ and $\phi(H)$ is a subgroup of $G$. Equip $H$ with the group structure so that $\phi$ is a homomorphism.

We need to check that $H$ is a Lie group, i.e. that $\alpha: H \times H \to H$ with $\alpha(g,h) = gh^{-1}$ is $C^\infty$. Since $H$ is a subgroup of $G$, the map $\beta\co H\times H \to G$ given by $\beta(g,h) = \phi(g) \phi(h)^{-1}$ is $C^\infty$ and the following diagram commutes:
$$
\begin{diagram}
 \node{H\times H} \arrow{e,t}{\beta} \arrow{se,b}{\alpha} \node{G}\\
\node{} \node{H} \arrow{n,r}{\phi}
\end{diagram}
$$
Since $H$ is an integral manifold of an involutive distribution on $G$, $\alpha$ is also $C^\infty$ by Proposition \ref{smoothness}. Thus $(H,\phi)$ is a Lie subgroup of $G$, and if $\frakh$ is the Lie algebra of $H$, $d\phi(\frakh) = \tilde\frakh$ for dimension reasons by Theorem \ref{LieAlgebraHomo}.

To prove uniqueness, let $(K,\psi)$ be another connected Lie subgroup of $G$ with $d\phi(\frakk) = \tilde\frakh$. Then $(K,\phi)$ is also an integral manifold of $\cD$ through $e$. By the maximality of $(H,\phi)$, $\psi(K) \subset \phi(H)$, and there is a uniquely determined map $\psi_0$ such that $\phi \circ \psi_0 = \psi$
$$
\begin{diagram}
 \node{K} \arrow{e,t}{\psi} \arrow{se,b,..}{\psi_0}\node G\\
\node{}\node{H}\arrow{n,r}{\phi}
\end{diagram}
$$
By Proposition \ref{smoothness} $\psi_0$ is smooth, and therefore is an injective Lie group homomorphism. Also, $\psi_0$ is everywhere non-singular, in particular it is a diffeomorphism on a neighborhood of the identity, and therefore surjective by Proposition \ref{generating}. Thus $\psi_0$ is a Lie group isomorphism. This proves uniqueness.
\end{proof}

\begin{cor}
 There is a 1:1 correspondence between connected Lie subgroups of a Lie group and subalgebras of its Lie algebra.
\end{cor}

\kommentar{\section{Simply connected Lie groups}}
Recall that a topological space is simply connected, if its fundamental group is trivial. Then one can show the following theorem with some more work on covering spaces. For a proof see \cite[Theorem 3.28]{warner83}.

\kommentar{\begin{thm}
 Let $G$ and $H$ be Lie groups with Lie algebras $\frakg$ and $\frakh$ respectively and with $G$ simply connected. Let $\psi\co \frakg\to \frakh$ be a homomorphism. Then there exists a unique homomorphism $\phi \co G \to H$ such that $d\phi = \psi$.\qed
\end{thm}}

\begin{thm}
 There is a 1:1 correspondence between isomorphism classes of Lie algebras and isomorphism lasses of simply connected Lie groups.
\end{thm}

\section{The exponential map}

Let $G$ be a Lie group, and let $\frakg$ be its Lie algebra.

\begin{defn}
A homomorphism $\phi\co \R \to G$ is called a {\em 1--parameter subgroup of $G$}.
\end{defn}

We will see that 1--parameter subgroups are precisely the integral curves of left-invariant vector fields.

\begin{lem}
 Every left-invariant vector field on $G$ is complete.
\end{lem}

\begin{proof}
 Let $G$ be a Lie group, let $X \in \frakg$, and let $\theta$ denote the flow of $X$. Suppose some maximal integral curve $\theta^{(g)}$ with starting point $g$ is defined on an interval $(a,b) \subset \R$, and assume $b<\infty$. (The case $a > -\infty$ is handled analogously.) By Lemma \ref{relatedflow}
\begin{equation}\label{FlowCommutesWithLeftTranslation}
 l_g \circ \theta_t (h) = \theta_t \circ l_g (h)
\end{equation}
on the domain of $\theta$. The integral curve $\theta^{(e)}$ starting at the identity is at least defined on $(-\ep,\ep)$ for some $\ep >0$. Choose $c \in (b-\ep,b)$ and define the curve $\gamma\co (a,c+\ep) \to G$ by
$$
\gamma(t) = \begin{cases} \theta^{(g)}(t) & t\in(a,b)\\
l_{\theta^{(g)}(c)}(\theta^{(e)}(t-c)) & t\in(c-\ep,c+\ep).
            \end{cases}
$$
By \eqref{FlowCommutesWithLeftTranslation} we have for $t\in (a,b) \cap (c-\ep,c+\ep)$
\begin{align*}
l_{\theta^{(g)}(c)}(\theta^{(e)}(t-c)) &= l_{\theta^{(g)}(c)}(\theta_{t-c}(e))\\
&=\theta_{t-c}(l_{\theta^{(g)}(c)}(e)) = \theta_{t-c} (\theta_c(g))\\
&=\theta_t(g) = \theta^{(g)}(t),
\end{align*}
where $\theta_{t-c} \circ \theta_{c} =\theta_{t}$ by the Fundamental Theorem of Flows, so the definitions for $\gamma$ agree on the overlap.

On $(a,b)$ is an integral curve of $X$, for $t_0 \in (c-\ep,t+\ep)$ we compute using the chain-rule and the left-invariance of $X$
\begin{align*}
 \gamma'(t_0) &= \left.\frac{d}{dt}\right|_{t = t_0} l_{\theta^{(g)}(c)}(\theta^{(e)}(t-c))  =dl_{\theta^{(g)}(c)}\left(\left.\frac{d}{dt}\right|_{t = t_0} \theta^{(e)}(t-c)\right) \\
&=dl_{\theta^{(g)}(c)} (X_{\theta^{(e)}(t_0-c)})  = X_{\gamma(t_0)}.
\end{align*}
Therefore $(a,b)$ is not the maxmial of $\theta^{(g)}$.
\end{proof}

Using \eqref{FlowCommutesWithLeftTranslation} again, we can easily show the following.

\begin{exe}\label{IntegralCurveDeterminesOneParameterSubgroup}
 Let $X\in \frakg$. The integral curve of $X$ starting at $e$ is a one-parameter subgroup of $G$.
\end{exe}

\begin{thm}\label{OneParameterGroupEquivalence}
 Every one-parameter subgroup of $G$ is an integral curve of a left-invariant vector field. Thus there are one-to-one correspondences
$$
\{\text{one-parameter subgroups of }G\} \longleftrightarrow \frakg \longleftrightarrow T_e G
$$
In particular, a one-parameter subgroup is uniquely determined by its initial tangent vector in $T_e G$.
\end{thm}

\begin{proof}
 Let $\phi \co \R \to G$ be a one-parameter subgroup, and let $X = d\phi \left(\frac{d}{dt}\right) \in \frakg$, where $\frac{d}{dt}$ is considered as a left-invariant vector field on $\R$. Since $X$ is $\phi$--related to $\frac{d}{dt}$ we have
$$
\phi'(t_0) = d\phi_{t_0}\left.\left(\frac{d}{dt}\right|_{t_0}\right) = X_{t_0}.
$$
Therefore $\phi$ is the integral curve of $X$. Together with Exercise \ref{IntegralCurveDeterminesOneParameterSubgroup} and Theorem \ref{LieAlgebraHomo} the proof is complete.
\end{proof}

\begin{defn}
 We call the integral curve $\phi$ of $X$ the {\em one-parameter subgroup generated by $X$}. The {\em exponential map of $G$} is given by
$$
\exp(X)\coloneqq  \phi(1).
$$
\end{defn}

The following proposition explains the choice of name.

\begin{prop}
 For any matrix $A\in \frakgl(n,\R)$, let
$$
e^A = \sum_{k=0}^{\infty} \frac{1}{k!}A^k.
$$
This series converges to an invertible matrix, and the one-parameter subgroup of $\GL(n,\R)$ generated by the left-invariant vector field $\tilde A$ corresponding to $A$ is $\phi(t) = e^{tA}$.
\end{prop}

\begin{proof}
 Since $\|AB\| \le \|A\|\|B\|$ and therefore $\|A^k\| \le \|A\|^k$ with the standard norm on $\R^{n^2}$, the series converges uniformly and $\|e^A\| \le e^{\|A\|}$.

After differentiating each summand, the series still converges uniformly on bounded sets. Therefore term-by-term differentiation is justified and $\phi(t)=e^{tA}$ satisfies the differential equation $\phi'(t) = \phi(t) A$ as well as $\phi(0) = e$.

Consider the left-invariant vector field $\tilde A$ corresponding to $A$. Therefore, we have for our (global) coordinate functions $x_{ij}$ on $\GL(n,\R)$ that $\tilde A_e(x_{ij}) = A_{ij}$. The one-parameter subgroup $\gamma$ generated by $\tilde{A}$ is the unique integral curve of $\tilde A$ starting at the identity, i.e. it is the unique solution to the initial value problem
\begin{align*}
 \gamma'(t) &= \tilde A_{\gamma(t)} = dl_{\gamma(t)} \tilde A_e\\
 \gamma(0) & = e,
\end{align*}
which written in matrix form via our coordinates $x_{ij}$ can be easily seen to be equivalent to
\begin{align*}
 \gamma'(t) &= \gamma(t)A\\
 \gamma(0) & = e.
\end{align*}
Therefore $\gamma(t) = \phi(t)$.
\end{proof}

The following proposition is not surprising, but is not apparent in our previous discussion. In particular, it tells us what the one-parameter subgroups of matrix Lie groups are.

\begin{prop}
 Let $H \subset G$ be a Lie subgroup. The one-parameter subgroups of $H$ are precisely those one-parameter subgroups of $G$ whose initial tangent vectors lie in $T_eH$.
\end{prop}

\begin{proof}
 Let $\phi\co \R \to H$ be a one-parameter subgroup. Then $$\R \stackrel{\phi}{\longrightarrow} H \hookrightarrow G$$ is a Lie group homomorphism and thus a one-parameter subgroup of $G$.

On the other hand, if $\phi\co \R \to G$ and $\psi \co \R \to H$ are one-parameter subgroup with the same initial vector in $T_e H$, then $$\R \stackrel{\psi}{\longrightarrow} H \hookrightarrow G$$ must be equal $\phi$.
\end{proof}

\begin{prop}[Properties of the exponential map]\label{expproperties}
 Let $G$ be a Lie group and $\frakg$ be its Lie algebra.
\begin{enumerate}
 \item The exponential map $\exp\co \frakg \to G$ is smooth.
 \item For any $X\in \frakg$, $\phi(t)  =\exp(tX)$ is the one-parameter subgroup of $G$ generated by $X$.
 \item For any $X \in \frakg$, $\exp(s+t) X = \exp(sX)\exp(tX)$.
 \item Under the canonical identifications of $T_0\frakg$ and $T_eG$, $d\exp_0 = {\rm id}$.
 \item The exponential map restricts to a diffeomorphism from a neighborhood of 0 in $\frakg$ to a neighborhood of $e$ in $G$.
 \item If $H$ is another Lie group and $\frakh$ its Lie algebra, then for any Lie group homomorphism $f\co G \to H$ the following diagram commutes:
$$
\begin{diagram}
 \node{\frakg} \arrow{e,t}{df} \arrow{s,l}{\exp}\node{\frakh}\arrow{s,r}{\exp}\\
\node{G} \arrow{e,b}{f} \node H
\end{diagram}
$$
 \item The flow $\theta$ of $X\in \frakg$ is given by $\theta_t = r_{\exp(t X)}$, where $r$ denotes right translation.
\end{enumerate}
\end{prop}

\begin{proof}  For any $X\in \frakg$ let $\theta_{(X)}$ its flow.
\begin{enumerate}
 \item Define a smooth vector field $Y$ on $G\times \frakg$ by
$$
Y_{(g,X)} \coloneqq  (X_g,0) \in T_gG \times T_X \frakg \cong T_{(g,X)}(G\times \frakg)
$$
The flow of $Y$ is given by $\eta_t(g,X) \coloneqq  (\theta_{(X)}(t,g),X)$. By the fundamental theorem of flows, Theorem \ref{fundthmflows}, $\eta$ is smooth. Since $\exp X = \pi_1(\theta_1(e,X))$, where $\pi_1\co G \times \frakg \to G$ is the projection, $\exp$ is smooth.
 \item We need to show that $\exp(tX) = \theta_{(X)}^{(e)}(t)$, in other words
$$
\theta_{(tX)}^{(e)}(1) = \theta_{(X)}^{(e)}(t).
$$
Fix $t\in \R$ and define $\gamma (s) =  \theta_{(X)}^{(e)}(st)$. Then by the chain rule,
$$\gamma'(s) = t(\theta_{(X)}^{(e)})'(st) = tX_{\gamma(s)}.$$
So $\gamma$ is integral curve of $tX$. Since $\gamma(0) = e$, by the uniqueness of integral curves we have $\gamma(s) =\theta^{(e)}_{(tX)}(s)$. For $s=1$ we get the desired formula.
\item This follows from the one-parameter group being a group homomorphism.
\item Let $X\in T_eG = T_0\frakg$ be arbitrary and $\sigma(t) = tX$. Then $\sigma'(0) = X$ and by (2) we have
$$
d\exp_0 X = d\exp_0\sigma'(0) = (\exp\circ\sigma)'(0) = \eval{\frac{d}{dt}}{t=0} \exp tX = X.
$$
\item This follows from (4) by the inverse function theorem.
\item By (2) $\exp (tdf(X))$ is the one-parameter subgroup generated by $df(X)$. Let $\sigma(t) \coloneqq  f(\exp(tX))$. Then
$$
\sigma'(0) = \eval{\frac{d}{dt}}{t=0} f(\exp(tX)) = df\left(\eval{\frac{d}{dt}}{t=0}\exp(tX)\right) = df(X)
$$
and
$$
\sigma(s+t) =f(\exp(s+t)X) \stackrel{(3)}{=} f(\exp(sX) \exp(tX)) = f(\exp(sX)) f(\exp(tX)) = \sigma(s)\sigma(t).
$$
So $\sigma$ is also a one-parameter subgroup satisying the same differential equation and therefore equals $\exp (tdf(X))$.
\item By (2) and Lemma \ref{relatedflow} we have
\[
r_{\exp(tX)}(g)=g\exp(tX)=l_g(\exp(tX))=l_g(\theta_{(X)t}(e)) = \theta_{(X)t}(l_g(e)) = \theta_{(X)t}(g).\qedhere
\]
\end{enumerate}
\end{proof}

\section{Lie Group actions}

Let $G$ be a Lie group and $M$ a smooth manifold.
\begin{defn}
A {\em left action of $G$ on $M$} is a smooth map
\begin{align*}
\theta\co G\times M &\to M\\
(g,p) &\mapsto \theta_g(p) = g\cdot p
\end{align*}
such that
\[
 g\cdot (h \cdot p) = (g\cdot h) \cdot p \quad \text{and} \quad e\cdot p = p \quad \text{for all } g,h\in G, p\in M.
\]
A {\em right action of $G$ on $M$} is a smooth map
\begin{align*}
M \times G &\to M\\
(p,g) &\mapsto p\cdot g
\end{align*}
such that
\[
 (p\cdot g) \cdot h = g\cdot (g \cdot h) \quad \text{and} \quad p\cdot e = p \quad \text{for all } g,h\in G, p\in M.
\]
\end{defn}

We will list some frequently used terminology regarding group actions. Consider a left action of $G$ on $M$.
\begin{itemize}
 \item The {\em orbit of $p$} under the action is the set $G\cdot p \coloneqq \{g\cdot p \mid g\in G\}$.
 \item The {\em isotropy group of $p$} is given by $G_p \coloneqq \{ g \mid g \cdot p = p\}$.
 \item The action is {\em transitive} if for any $p,q\in M$ there is a $g\in G$ such that $g\cdot p = q$.
 \item The action is {\em free} if $G_p = \{e\}$ for all $p\in M$.
 \item $p\in M$ is a {\em fixed point of the action}, if $g\cdot p = p$ for all $g\in G$.
\end{itemize}

A few relevant examples are the following:
\begin{enumerate}
 \item Any Lie group acts freely and transivitely on itself by multiplication.
 \item Any Lie group $G$ also acts on itself by conjugation $C_g h \coloneqq ghg^{-1}$, for which $G_h$ are all elements which commute with $g$. Notice that $C_g$ is a Lie group homomorphism.
 \item $O(n)$ naturally acts on $\R^n$ and restricts to a transitive action on $S^{n-1}$.
\end{enumerate}

Using the following concept we can generate a lot more important examples.

\begin{defn}
 If $G$ is a Lie group, a (finite-dimensional) {\em representation of $G$} is a Lie group homomorphism $\rho\co G \to \GL(V)$ for some finite-dimensional real or complex vector space. We will also encounter (possibly infinite-dimensional) representations on Hilbert spaces, where $\GL(V)$ gets replaced by the group of bounded linear operators on a Hilbert space. If $\rho$ is injective, $\rho$ is called {\em faithful}.
\end{defn}

Any representation $\rho$ yields a left action $g\cdot v \coloneqq \rho(g)v$ for $g\in G, v\in V$. If $G$ is a Lie subgroup of $\GL(n,\R)$, then $G\hookrightarrow \GL(n,\R)$ is a faithful representation, called the {\em defining representation}.

If $G$ is a Lie group and $\frakg$ its Lie algebra, consider the Lie algebra homomorphism $\Ad$ induced by $C_g$. By the following theorem $\Ad\co G \to \GL(\frakg)$ and $\Ad$ is a representation of $G$, called the {\em adjoint representation}.
\begin{thm}
 Let $\theta\co G\times M \to M$ be a left action for which $p$ is a fixed point, then the map
\begin{align*}
\rho \co G &\to \Aut(T_pM)\\
g&\mapsto (d\theta_g)_p
\end{align*}
is a representation of $G$.
\end{thm}

\begin{proof}
Notice that $\rho$ not only maps into the endomorphisms of $T_pM$ but the Lie group of automorphisms, because $\rho(e) = \Id$ and
\[
\rho(gh) = (d\theta_{gh})p = d(\theta_g\circ\theta_h)_p = \rho(g)\rho(h)
\]
imply that $\rho(g)$ is invertible with $\rho(g)^{-1}=\rho(g^{-1})$. The above remarks also show, that $\rho$ is a homomorphism of groups.

It remains to show that $\rho$ is $C^\infty$. It suffices to show that $\rho$ composed with any coordinate function on $\Aut(T_p M)$ is $C^\infty$. One gets a coordinate system by using a basis of $T_pM$ (and its dual basis) to identify $\Aut(T_pM)$ with non-singular matrices: a coordinate function on $\Aut(T_pM)$ is of the form
\begin{align*}
\phi\co \Aut(T_pM) &\to \R\\
 A &\mapsto v^*(Av),
\end{align*}
where $v \in T_pM$ is a basis vector and $v^*$ is the dual vector with respect to the basis. Then $\phi\circ \rho$ is smooth, because it is the composition of $C^\infty$ maps
\[
 G \stackrel{i}{\to}  T(G) \times T(M) \stackrel{\cong}{\to} T(G\times M) \stackrel{d\theta}{\to} T(M) \stackrel{v^*}{\to} \R
\]
where $i(g)=((g,0),(p,v))$.
\end{proof}

Recall that the equivalent of a directional derivative in Euclidean space is the Lie derivative $L_X Y$ of smooth vector fields $X$ and $Y$ given by
\[
(L_X Y)_p = \eval{\frac{d}{dt}}{t=0}d(\theta_{-t}) Y_{\theta_{t}(p)} = \lim_{t\to 0}\frac{d(\theta_{-t}) Y_{\theta_{t}(p)} - Y_p}{t},
\]
where $\theta$ is the flow of $X$.

We will need to know the following important result.
\begin{prop} For any smooth vector fields $X$ and $Y$ on a smooth manifold $M$
\[
L_X Y= [X,Y].
\]
\end{prop}

\begin{proof}
 Let $f \in C^\infty(M)$. Let $\theta_t$ be the flow of $X$ and define $F(t,p) := f(\theta_t(p)) - f(p)$ and $F' := \tfrac{\partial}{\partial t}F$. Set \[g_t(p) := \int_0^1 F'(ts,p) ds.\]
Then $g \co (-\ep,\ep) \times M \to \R$ is smooth and satisfies for $t > 0$
\[
t g_t(p) = \int_0^1 t F'(ts,p) ds = \int_0^t F'(\tilde s,p)d\tilde s = F(t,p).
\]
Similarly we can extend this to $t g_t(p) = f(\theta_t(p))-f(p)$ for all $t\in (-\ep,\ep)$, which yields
\[
X_pf = \eval{\frac{d}{dt}}{t=0} f (\theta_t(p)) = \lim_{t\to 0} \frac{f(\theta_t(p))-f(p)}{t} = \lim_{t\to 0} g_t(p) = g_0(p)
\]
as well as
\[
(d\theta_{-t}) Y_{\theta_{t}(p)} f = Y_{\theta_{t}(p)}(f \circ \theta_{-t}) = Y_{\theta_{t}(p)} f  - tY_{\theta_{t}(p)} (g_{-t}).
\]
Therefore,
\begin{align*}
 (L_X Y)_p (f) &= \lim_{t\to 0}\frac{d(\theta_{-t}) Y_{\theta_{t}(p)}f-Y_pf}{t} = \lim_{t\to 0}\frac{Y_{\theta_{t}(p)}f-Y_pf}{t} - \lim_{t\to 0}Y_{\theta_{t}(p)} (g_{-t})\\
&= \eval{\frac{d}{dt}}{t=0} Y (f \circ \theta_t(p)) - Y_p  g_0  = X_p (Y f) - Y_p  (X f) = [X,Y]_p f.\qedhere
\end{align*}
\end{proof}

A Lie algebra homomorphism $\frakg \to End(V)$ for some vector space $V$ is also known as a representation of $\frakg$.

\begin{thm}\label{adandliebracket}
 Let $G$ be a Lie group, let $\frakg$ be its Lie algebra. Then the Lie algebra representation induced by $\Ad$
\[
\ad \co \frakg \to \frakgl(\frakg)
\]
is given by $\ad(X) Y = [X,Y]$.
\end{thm}

\begin{proof}
We have
\[
\ad(X) Y = \left(\eval{\frac{d}{dt}}{t=0} \Ad(\exp(tX))\right) Y = \eval{\frac{d}{dt}}{t=0} (\Ad_{\exp(tX)}(Y)) = \eval{\frac{d}{dt}}{t=0} d\left(C_{\exp(tX)}\right)(Y).
\]
By Theorem \ref{expproperties} the flow of $X$ is given by $\theta_t = r_{exp (tX)}$. Then
\begin{align*}
 (\ad(X)Y)_e & = \eval{\frac{d}{dt}}{t=0} d(r_{\exp(-tX)})\left(d(l_{\exp(tX)})(Y)\right)\\
& = \eval{\frac{d}{dt}}{t=0} d(r_{\exp(-tX)})\left(Y_{\exp(tX)}\right)\\
& = \eval{\frac{d}{dt}}{t=0} d(\theta_{-t})\left(Y_{\theta_t(e)}\right)\\
& = (L_X Y)(e) = [X,Y](e).
\end{align*}
Since left-invariant vector fields are determined by their value at $e$, this completes the proof.
\end{proof}

\section{Principal bundles}

\begin{defn}\label{defnprincipalbundle}
Let $M$ be a manifold and $G$ be a Lie group. A {\em principal $G$--bundle} over $M$ consists of a smooth manifold $P$ and an action of $G$ on $P$ satisfying
\begin{enumerate}
 \item $G$ acts freely on the right;
 \item $M$ is the quotient space of $P$ by the equivalence relation induced by $G$, $M = P/G$ and the canonical projection $\pi\co P \to M$ is differentiable;
 \item $P$ is locally trivial, that is, every point $x\in M$ has a neighborhood $U$ such that there exists an diffeomorphism $\psi\co \pi^{-1}(U) \to U \times G$ with
\begin{enumerate}
 \item $\psi$ is equivariant: $\psi(p\cdot g) = \psi(p) \cdot g$, where $(x,h)\cdot g = (x,h\cdot g)$ for $x \in U$ and $g,h\in G$;
 \item $\pi(\psi^{-1}(x,g)) = x$.
\end{enumerate}
\end{enumerate}
$P$ is called the {\em total space}, $M$ the {\em base space}, $G$ is the {\em structure group}, $\pi$ the {\em projection} and $\pi^{-1}(x)$ the {\em fiber} over a point $x\in M$.
\end{defn}

Given a principal bundle $P$, it is not difficult to choose an open covering $\{ U_\alpha \}$ of $M$ and with diffeomorphisms $\psi_\alpha$ from Definition \ref{defnprincipalbundle}. Consider the {\em transition functions} $\psi_{\alpha\beta}$
\[
(U_\alpha\cap U_\beta) \times G \stackrel{\psi_\alpha^{-1}}{\longrightarrow} \pi^{-1}(U_\alpha \cap U_\beta) \stackrel{\psi_\beta}{\longrightarrow}(U_\alpha\cap U_\beta) \times G.
\]
Let $\pi_2\co U_\alpha \times G \to G$ be the projection to the second coordinate. Then $\pi_2\circ \psi_\alpha(p\cdot g) = \pi_2 \circ \psi_\alpha(p)\cdot g$ and thus for $p\in \pi^{-1}(U_\alpha\cap U_\beta)$
\[
 \pi_2\circ \psi_\beta(pg) (\pi_2\circ \psi_\alpha(pg))^{-1} =  \pi_2\circ \psi_\beta(p) (\pi_2\circ \psi_\alpha(p))^{-1},
\]
which shows that $\psi_{\alpha\beta}$ only depends on $\pi(p)$ and not on $p$ and we have $\psi_{\alpha\beta}(x,g) = (x,g h_{\alpha\beta}(x))$ for some smooth map $h_{\alpha\beta}\co U_\alpha\cap U_\beta\to G$.

If $s\co M \to P$ is a section of a principal $G$--bundle $\pi\co P\to M$, i.e. $s$ is smooth with $\pi\circ s = \text{id}$, then $s$ determines a trivialization $P \cong M\times G$ by $p \mapsto (\pi(p), g_p)$, where $g_p \in G$ is determined by $p \cdot g_p = s(\pi(p))$. $P$ is called {\em trivializable}\footnote{We want to stress that there is a  difference between the trivial bundle $M \times G$ and the trivializable bundle $P$.}.

 The {\em frame bundle} $\GL(TM) \to M$ is the principal $\GL(n,\R)$--bundle whose fiber over a point $x\in M$ is the collection of all bases for $T_x M$; this fiber is (not canonically) isomorphic to $\GL(n,\R)$. For a given Riemannian metric on $M$ we have the {\em orthonormal frame bundle} by considering all orthonormal frames. This forms a sub-principal bundle of $\GL(T M)$ with fiber $O(n) \subset \GL(n,\R)$.

\begin{exe}
 If $M$ is orientable, then $O(T M)$ is the union of two components. A choice of orientation for $M$ is the same as a choice of a component of $O(TM)$. This component is a principal $\SO(n)$--bundle, the {\em oriented orthonormal frame bundle}.
\end{exe}

  You can construct a principal $\GL(n,\R)$-- (or $\GL(n,\C)$--) bundle over $M$ from a vector bundle $E \to M$. If $E$ is equipped with a fiberwise positive definite symmetric inner product, then its structure group can be reduced to $O(n)$ or $U(n)$\footnote{In this case, $E$ is called a $O(n)$-- or $U(n)$--vector bundle.} and one can construct a principal $O(n)$-- (or $U(n)$--) bundle from $E$.

\begin{exe}
 The Hopf fibration is a principal $S^1$--bundle over $S^2$, where the projection $\pi\co S^3\subset \C^2 \to S^2 \subset \C^2 \times \R = \R^3$ is given by $\pi(z_0,z_1) = (2z_0\bar z1, |z_0|^2 - |z_1|^2)$.
\end{exe}

\begin{defn}
Let $P \to M$ be a principal $G$--bundle and $\rho\co G \to \GL(V)$ a representation of $G$. Then define (the {\em Borel construction})
\[
 P \times_\rho V
\]
to be the quotient space $P\times V/\sim$ where
\[
 (p,v)\sim (pg,g^{-1}v).
\]
\end{defn}

\begin{exe}
 The function $P\times_\rho V \to M$ given by $(p,v) \mapsto \pi(p)$ is well-defined and makes $P\times_\rho V$ into a vector bundle over $M$ with fiber $V$, the vector bundle associated to $P$ by the representation $V$.
\end{exe}

For the defining representation of $\SO(n)$ the associated vector bundle to a principal $\SO(n)$--bundle is a vector bundle with fiber $\R^n$. Similarly we can take the defining representation of $U(n)$.

For the Adjoint representation $\Ad$ the associated vector bundle is called the Adjoint bundle and is denoted by $\Ad_P \to M$.

\begin{exe}
 Show that the adjoint representation of $\SO(3)$ is isomorphic to its defining representation $i\co \SO(3) \hookrightarrow \GL(3,\R)$: find an $\SO$--equivariant vector space isomorphism $\phi\co \R^3 \to \frakso(3)$ (also known as an {\em intertwining map}), i.e. an isomorphism satisfying $\Ad(g) \phi (v) = \phi (i(g) v)$ for all $g \in \SO(3)$ and $v\in \R^3$.
\end{exe}

Since $\SU(2)$ is the double cover of $\SO(3)$, the adjoint representation of $\SU(2)$ is the same as the composition of the 2--fold cover $\SU(2) \to \SO(3)$ with the defining representation of $\SO(3)$.

\begin{exe}\label{liebracketadjointbundle}
We can define a homomorphism $[\cdot,\cdot]\co \Ad P \tensor \Ad P \to \Ad P$ by setting $[(p,v),(p,v')] := (p,[v,v'])$ on representatives $(p,v)$ and $(p,v')$ of elements in the fiber of $\pi(p)$. Check that this homomorphism is well-defined.
\end{exe}

A useful observation is that if $E=P\times_\rho V$, then sections of $E$ are just $G$--equivariant maps $f\co P\to V$. The correspondence is given by $(f\co P \to V) \mapsto (\phi_f \in \Gamma(E))$ where $\phi_f(x)$ is the equivalence class of $(p,f(p))$ where $p \in \pi^{-1}(x)$.

\begin{exe}\label{EquivariantMapsIsomorphism}
 Check that this map makes sense and is a vector space homomorphism.
\end{exe}

\section{Homology with local coefficients}

\begin{defn}The {\em group ring} $\Z \pi$ is a ring associated to a group $\pi$. Additively it is the free abelian group on $\pi$ with multiplication given by
\[
 \left(\sum_i m_i g_i\right) \left(\sum_j n_j h_j\right) = \sum_{i,j}(m_in_j)(g_i h_j), \quad m_i,n_j\in \Z, g_i,h_j \in \pi.
\]
\end{defn}

Let $A$ be an abelian group and
\[
\rho\co \pi \to \Aut_{\Z}(A)
\]
be a homomorphism.

\begin{exe}
 Show that representations of $\pi$ can be identified with left modules over $\Z \pi$.
\end{exe}

Let $X$ be a path connected and locally path-connected\footnote{$X$ is {\em locally path-connected} if for every point $x\in X$ and every open set $V$ containing $x$ there exists a path-connected open set $U$ with $x \in U \subset V$.} space $X$ with a base point $x_0$ which admits a universal cover\footnote{A {\em covering space} of $M$ is a space $C$ together with a continuous surjective map $p\co C \to M$, such that for every $x \in M$, there exists an open neighborhood $U$ of $x$, such that $p^{-1}(U)$ is a disjoint union of open sets in $C$, each of which is mapped homeomorphically onto $U$ by $p$. A connected covering space is a {\em universal cover} if it is simply connected.}. Let $\tilde X \to X$ be the universal cover of $X$. $\pi_1 X$ naturally acts on $\tilde X$ on the right. Then the singular chain complex\footnote{Singular $q$--chains are linear combinations of $q$--simplices
\[\textstyle
\sigma\co \Delta^q = \{ (t_0,\ldots, t_1) \in \R^{g+1} \mid \sum t_i = 1, t_i \ge 0 \text{ for all } i\} \to X
\]
and form a chain complex with the boundary map given by
\[
 \delta(\sigma) \coloneqq \sum_{m=0}^q (-1)^m \sigma \circ f^q_m,
\]
where $f^q_m (t_0,\ldots,t_{q-1}) \coloneqq (t_0,\ldots,t_{m-1},0,t_m,\ldots,t_{q-1})$ is the face-map.
} $S_*(\tilde X)$ is a right $\Z\pi$-module.

\begin{defn}
Consider a left $\Z\pi$--module $A$ for $\pi = \pi_1 X$ and form the chain complex
\[
 S_*(X;A) = S_*(\tilde X) \tensor_{\Z\pi} A.
\]
The homology $H_*(X;A)$ of this complex is called the {\em homology of $X$ with local coefficients in $A$}. If $A$ is defined via a representation $\rho\co \pi \to A$, then $H_*(X;A)$ is called the {\em homology of $X$ twisted by $\rho$} and denoted by $H_*(X;A_\rho)$.
\end{defn}

We can turn $S_*(\tilde X)$ into a left $\Z \pi$--module by the standard procedure
\[
 g\cdot z \coloneqq z \cdot g^{-1}.
\]

\begin{defn}\label{DefnTwistedHomology}
Consider a left $\Z\pi$--module $A$ for $\pi = \pi_1 X$ and form the cochain complex
\[
 S^*(X;A) \coloneqq \Hom_{\Z \pi}(S_*(\tilde X),A).
\]
The cohomology $H^*(X;A)$ of this complex is called the {\em cohomology of $X$ with local coefficients in $A$}. If $A$ is defined via a representation $\rho\co \pi \to A$, then $H^*(X;A)$ is called the {\em cohomology of $X$ twisted by $\rho$} and denoted by $H^*(X;A_\rho)$.
\end{defn}
See \cite[Chapter 5]{davis-kirk2001} for details.

\begin{exe}
 Convince yourself, that the usual integer cohomology of $X$ is isomorphic to $H^*(X;\Z_\rho)$ for the trivial representation $\rho \co \pi_1X \to \Z$.
\end{exe}

If $M$ is a closed manifold of dimension $n$ and $A$ is vector space equipped with a non-degenerate inner product $\la\cdot,\cdot\ra$, then Poincar\'e duality can be generalized to twisted cohomology without an extra effort
\[
 H^i(M;A_\rho) \cong H_{n-i}(M;A_\rho).
\]
\kommentar{by considering a triangulation $M^\tau$ of $M$ and its
dual polyhedral decomposition $M'$ (which is a CW--complex).}Just
like for the singular cohomology we have a relative version for
twisted cohomology. If $X$ is a manifold with boundary of dimension
$n$, we can generalize Poincar\'e-Lefschetz duality
\[
 H^i(X,\partial X;A_\rho) \cong H_{n-i}(X;A_\rho) \quad \text{and} \quad H^i(X;A_\rho) \cong H_{n-i}(X,\partial X;A_\rho).
\]
The universal coefficient theorem can also be generalized to twisted cohomology
\[
 H^i(X;A_\rho) \cong \Hom(H_i(X;A_\rho),\R).
\]
\kommentar{by showing that the pairing
\[
 S^i(X^\tau;A_\rho) \times S_i(X^\tau;A_\rho) \to \R\\
(f,\sigma \tensor v) \mapsto \la f(\sigma),v\ra
\]
is non-degenerate.}
And certainly, we have long exact sequences in twisted homology and cohomology.

\begin{prop}\label{PropHalfDimension}
 Let $X$ be a manifold of dimension $2n+1$ with boundary. Then
\[
\dim \im\left[j^* \co H^n (X;A_\rho) \to H^n(\partial X;A_\rho)\right] =\tfrac{1}{2} \dim(H^n(\partial X;A_\rho)).
\]
\end{prop}

\begin{proof}
 We will suppress the coefficients in our notation. Consider the following diagram made up out of the middle part of the long exact
sequences in homology and cohomology of the pair $(X,\partial X)$ and
the Poincar\'e-duality isomorphisms:
\[
\begin{diagram}
\node{\cdots} \arrow{e}
 \node{H^n(X)} \arrow{e,t}{j^*} \arrow{s,r}{\cong}
  \node{H^n(\partial X)} \arrow{e} \arrow{s,r}{\cong}
   \node{H^{n+1}(X,\partial X)} \arrow{e} \arrow{s,r}{\cong}
    \node{\cdots}\\
\node{\cdots} \arrow{e}
 \node{H_{n+1}(X,\partial X)} \arrow{e,t}{i_*}
  \node{H_n(\partial X)} \arrow{e,t}{j_*}
   \node{H_n(X)} \arrow{e}
    \node{\cdots}
\end{diagram}
\]
Then $\dim(\im j^*) = \dim (\im i_*) = \dim (\ker j_*)$. Now recall the fact, that if $L: V\to W$ is a homomorphism between finite-dimensional $\R$-vector spaces
and $L^*:W^* \to V^*$ its induced dual homomorphism, then $\ker L
\cong \coker L^* = V^*/\im L^*$. Then together with the universal coefficient theorem and Poincar\'e duality
 we get
\[
\dim(\im j^*) =\dim (\ker j_*) = \dim(\coker j^*) = \dim(H^n (\partial X)) - \dim(\im j^*)
\]
which proves $\dim (\im j^*) = \frac{1}{2}\dim(H^n (\partial X)$.
\end{proof}

Let $\pi$ be an arbitrary group. The {\em group cohomology $H^*(\pi,A_\rho)$ of $\pi$ twisted by $\rho$} is the cohomology of the complex given by $C^0(\pi,A_\rho) = A$, $C^n(\pi,A_\rho) = \Maps(\pi^n,A)$ and the boundary maps
\begin{align*}
\delta^0 (v) (g)= &\rho(g)(v) - v\\
\delta^n (f) (g_1,\ldots,g_{n+1}) = &\rho(g_1)(f(g_2,\ldots,g_{n+1}))\\ &{}+ \sum_{i=1}^n (-1)^i f(g_1,\ldots,g_i g_{i+1},\ldots,g_{n+1}) + (-1)^{n+1}f(g_1,\ldots,g_n).
\end{align*}
The $k$--th cohomology is given by $Z^k(\pi,A_\rho)/B^k(\pi,A_\rho)$, where $Z^k(\pi,A_\rho) = \ker(\delta^k)$ are the $k$--cocycles and $B^k(\pi,A_\rho) = \im(\delta^{k-1})$ are the $k$--coboundaries. For example a $1$-cocycle $\xi$ is a map $\pi \to A$, which satisfies the {\em cocycle condition}
\[
 \xi(gh) = \xi(g) + \rho(g)\xi(h).
\]

Eilenberg-Maclane spaces are spaces $K(G,n)$ defined up to homotopy with the property that all homotopy groups except for $\pi_n(K(G,n)) = G$ are trivial. It is shown in \cite{hempel2004}, that any compact $3$ manifold with torsion-free
  fundamental group is a $K(\pi_1(X),1)$. It turns out that $H^i(K(\pi,1),A_\rho) = H^i(\pi,A_\rho)$.

\section{Algebraic sets and the Zariski tangent space}

We will review necessary definitions and results from real algebraic geometric.  See \kommentar{for example \cite{shafarevich77} or }\cite{bochnak-coste-roy98} for more information.

For a subset $S$ of the polynomial ring $\R[x]$, $x\in \R^n$ we call
\[
 \cZ(S) \coloneqq \{ x = (x_1,\ldots,x_n) \in \R^n \mid g(x) = 0 \text{ for all } g\in S \}
\]
an {\em algebraic set}. By the {\em Hilbert basis theorem}, every ideal in $\R[x]$ is finitely generated. In particularly the ideal generated by any subset $S \subset \R[x]$ is generated by finite set $S'$ so that $\cZ(S) = \cZ(S')$. In fact $\cZ(\{f_1,\ldots,f_s\}) = \cZ(f)$ for $f = \sum_i f^2_i$. A {\em semi-algebraic set} of $\R^n$ is a finite number of unions, intersections and complements of sets $S(f_i)$, where
\[
S(f) = \{x\in \R^n \mid f(x) > 0\}.
\]

\begin{exe}
 The collection of algebraic sets in $\R^n$ is closed under finite union and arbitrary intersection. Also $\emptyset$ and $\R^n$ are algebraic sets.
\end{exe}

We can endow an algebraic set $V$ with a topology, the {\em Zariski topology}, where the closed sets are given by the algebraic sets. Given an algebraic set $V$, we can form the {\em vanishing ideal}
\[
 \cI(V) \coloneqq \{f \in \R[x] \mid f(x) = 0 \text{ for all } x \in V\}.
\]
Then one can show that $V = \cZ(\cI(V))$. The {\em dimension of a semi-algebraic set $V$} is defined to be the {\em dimension} of the ring of polynomial functions on $V$
\[
 \cP(V) = \R[x]/\cI(V),
\]
i.e. the maximal length of chains of prime ideals of $\cP(V)$.
The {\em coordinate ring} $\R[V]$ is the ring of polynomials restricted to $V$. $V\neq \emptyset$ is {\em irreducible}, if $V = V_1 \cup V_2$ implies $V = V_1$ or $V= V_2$. We will only consider irreducible algebraic sets, also known as {\em affine algebraic varieties}.

For an algebraic set $V$ and $y \in V$ the tangent space can be defined as the kernel $T_y V  \coloneqq \ker J_y$ of the Jacobian matrix
\[
J_y \coloneqq \left(\frac{\partial f_i}{\partial x_j} (y)\right)_{i,j}\co \R^n \to \R^s,
\]
where $\cI(V)$ is generated by $\{f_1,\ldots,f_s\}$.\kommentar{Alternatively, if we consider $T_yV$ as an affine subspace of the tangent bundle $T \R^n \cong \R^{2n}$, then it is the kernel of the linear part in the Taylor expansion of $F = (f_i)_{i=1,\ldots,s}$ around $y$
\[
\left\{ (y,x) \in \R^{2n}\mid d_y f_i (x) = 0 \text{ for all } i=1,\ldots,n\right\},
\]
where \[
d_y F (x)\coloneqq \sum_{j=1}^n \frac{\partial F}{\partial x_j}(y) (x_j - y_j) = \grad F|_y\cdot (x-y).
\]}
For our purpose there is a more convenient definition (due to Patrick \cite{patrick85}).

\begin{defn}\label{DefnZariskiTangentSpace}
Let $M$ be a smooth manifold and $V$ be a subset with the induced topology. The {\em tangent cone} $C_x V$ to $V$ at $x\in V$ consists of all tangent vectors $\dot\gamma(0)\in T_x M$ of smooth curves $\gamma\co (-\epsilon,\epsilon)\to M$ with $\gamma(0) = x$ and $\gamma([0,\epsilon)) \subset V$. The {\em tangent space} $T_x V$ at $x\in V$ is the linear span of $C_qV$.
\end{defn}

This definition has the obvious advantage of applying to more general sets, in particular to semi-algebraic sets.

\begin{exe}
 Check that both definitions for the tangent space are equivalent in the case of affine algebraic varieties.
\end{exe}

\kommentar{A point is called nonsingular at $y \in V$ if $\dim V = \dim T_y V$.
\begin{exe}\kommentar{proof or reference?}
If $y$ is nonsingular, then a neighbourhood of $y$ in $V$ is a smooth submanifold of $\R^n$ and the (manifold) tangent space at $y$ coincides with the Zariski tangent space.
 \end{exe}}
\kommentar{Tangent space = linear span of tangent cone, which is defined using curves through the point which lie in the variety on one side of the curve? siehe Arms-Gotay-Jennings Geometric and Algebraic Reduction for singular momentum Maps.}

If $\pi$ is a finitely presented discrete group and $G$ is a Lie group, then $\Hom(\pi,G)$ can be considered as a subspace of the manifold $G^n$, where $n$ is the number of generators. $G$ acts on $\Hom(\pi,G)$ on the right via $(\rho \cdot g) (\gamma) = g^{-1}\rho(\gamma)g$. Therefore, $\Hom(\pi,G)/G$ is a subset of the $G^n/G$, where $G$ acts simultaneously on the right by conjugation. If $G$ is algebraic, then $\Hom(\pi,G)$ is an algebraic set. In general, $\Hom(\pi,G)/G$ may not even be Hausdorff and  one has to pass to the categorical quotient $\Hom(\pi,G)//G$. See Sikora \cite{sikora2009_CharacterVarieties} for a detailed account for complex reductive algebraic groups $G$.

We will compute the tangent spaces to $\Hom(\pi,G)$ and $\Hom(\pi,G)/G$ using Definition \ref{DefnZariskiTangentSpace} for any Lie group $G$, as long as $\Hom(\pi,G)/G$ is a subspace of some manifold. It was Weil who showed how to identify $\Hom(\pi,G)/G$ with $H^1(\pi,\frakg_\rho)$ \cite{weil64}.

\begin{thm} If $\pi$ is a finitely presented discrete group and $G$ a Lie group, such that $\Hom(\pi,G)/G$ is a subspace of some manifold, then
\[
 T_\rho(\Hom(\pi,G)) = Z^1(\pi,\frakg_\rho) \quad \text{and} \quad T_{[\rho]}(\Hom(\pi,G)/G) = H^1(\pi,\frakg_\rho).
\]
\end{thm}

\begin{proof}
If $\pi$ is generated by $x_1,\ldots,x_n$, then
\begin{align*}
\Hom(\pi,G)&\hookrightarrow G\times \ldots \times G\\
\rho &\mapsto (\rho(x_1),\ldots,\rho(x_n)).
\end{align*}
Therefore a tangent vector at a representation $\rho \in \Hom(\pi,G)$ can be considered as an element of $T_{\rho(x_1)}G\times \ldots \times T_{\rho(x_n)}G$. If we want to think of a tangent vector independently of the set of generators, we can simply consider it as a map $\eta \co \pi \to TG$ with $\eta(x) \in T_x G$. Now let $\rho\co (-\epsilon,\epsilon) \to \R^n$ be a path satisfying $\rho(0) = \rho$ and $\rho([0,\epsilon)) \subset \Hom(\pi,G)$. By right translation we can write $\dot\rho_t(x) = \xi_t(x)\rho_t(x)$ for $\xi(x) \in \frakg$. Therefore the vector $\dot\rho_0$ at $\rho$ can be identified with a map $\xi_0 = \xi \co \pi \to \frakg$ satisfying
\[
\xi(xy)\rho(xy)=\eval{\tfrac{d}{dt} \rho_t(xy)}{t=0} = \eval{\tfrac{d}{dt}\rho_t(x) \cdot \rho_t(y)}{t=0}= \xi(x)\rho(x)\rho(y) + \rho(x)\xi(y)\rho(y),
\]
or equivalently the cocycle condition
\[
 \xi(xy) = \xi(x) + \Ad_{\rho(x)}\xi(y).
\]
This shows that $T_\rho \Hom(\pi,G) \cong Z^1(\pi,\frakg_\rho)$.

Let us turn to $T_{[\rho]} (\Hom(\pi,G)/G)$. A similar argument as above shows that two tangent vectors $\xi_1$ and $\xi_2$ to $\Hom(\pi,G)/G$ at $[\rho]$ associated to $\rho_t$ and $g_t\rho_t g_t^{-1}$ respectively---which after conjugation we may assume to be represented by tangent vectors to $\Hom(\pi,G)$ at $\rho$---are equivalent if there is some $\dot g_0 = v\in \frakg$ so that
\[
\xi_1(x)\rho(x)=\eval{\tfrac{d}{dt} \rho_t(x)}{t=0} = \eval{\tfrac{d}{dt} g_t \rho_t(x) g_t^{-1}}{t=0} = v \rho(x) + \xi_2(x)\rho(x)- \rho(x) v.
\]
Therefore $\xi_1, \xi_2 \in Z^1(\pi,\frakg_\rho)$ are equivalent if
\[
\xi_1(x)- \xi_2(x)= v - \rho(x) v\rho(x)^{-1} = \delta^0(v)(x),
\]
which shows
\[
T_{[\rho]} (\Hom(\pi,G)/G)  \cong  Z^1(\pi,\frakg_\rho)/B^1(\pi,\frakg_\rho) = H^1(\pi,\frakg_\rho).\qedhere
\]
\end{proof}

\kommentar{\begin{exe}
 Consider the ideal \[m_y \coloneqq \{ f\in \R[V] \mid f(y) = 0\}\] for an affine algebraic variety $V$. Show that
\begin{align*}
 d_y \co m_y/m_y^2 &\to T^*_y V = \Hom(T_y V ,\R)\\
f & \mapsto d_y f \coloneqq (d_y F)|_{T_y V} \quad \text{ for } F|_V = f.
\end{align*}
is well-defined and an isomorphism of vector spaces.
\end{exe}

The Zariski tangent space is usually defined in a more general context (for example for algebraic varieties or schemes) as the dual space of $\frakm_y/\frakm_y^2$, where $\frakm_y$ is the maximal ideal of the {\em local ring $\sO_y$} of $y$ on $X$. If $X$ is an irreducible algebraic variety, then
\[
\sO_y = \left\{\tfrac{f}{g} \mid f,g \in \R[V], g(y) \neq 0\right\}
\]
and $\frakm_y/\frakm_y^2 \cong m_y/m_y^2$.}

\chapter{Connections on principal bundles}

\section{A few equivalent notions} If $\pi\co P \to M$ is a principal $G$--bundle, then for every $p\in P$ we have a map $i_p\co G \to P$ given by the formula $i_p(g) = p\cdot g$. For $X \in \frakg$ we let $X^*$ be the {\em fundamental vector field corresponding to $X$} given by $X^*_p =  di_p (X_e)$.

\begin{exe}
 $X^*$ is smooth and vertical, i.e. $d\pi(X^*) = 0$, and $X^*_{pg} = di_p(X_g)$ for $p\in P$.
\end{exe}

\begin{defn}\label{defnconnection} A {\em connection} on a principal $G$-bundle $P \to M$ over a manifold $M$ is a $\frakg$--valued
  1--form $A$ on $P$ satisfying
\begin{enumerate}
\item $A(X^*) = X$ for every $X \in \frakg$;
\item $A$ is $G$--equivariant, i.e. $r^*_g(A) = \Ad_{g^{-1}}A$ for each $g\in G$,
\end{enumerate}
where $r_g\co P \to P$ denotes the right action by $g\in G$. We denote the set of connections on $P$ by $\cA_P$.
\end{defn}

$G$--equivariance says that if $Y$ is a vector field on $P$ and $g\in G$, then $A(dr_g(Y)) = \Ad_{g^{-1}} (A(Y))$. Locally, a 1--form on an $n$--dimensional manifold $M$ with values in $\frakg$ has the expression
\[
 (d x_1 \tensor A_1) + \ldots + (d x_n \tensor A_n),
\]
where the $x_i$ are coordinates on $M$ and the $A_i$ are elements of $\frakg$. Given a vector bundle $E \to M$, $\Omega^p(E)$ will denote the
$p$--forms with values in $E$, i.e. smooth sections of the vector
bundle
$$
\bigwedge^p T^*M \tensor E \to M,
$$
where the symbols denote the wedge and tensor products of the
bundles. A connection is therefore an element of $\Omega^1(P;\frakg)$, where
\[
 \Omega^k(P;\frakg):=\Omega^k(P\times \frakg)
\]
are the $\frakg$--valued $k$--forms on $P$. Denote the subspace of $G$--equivariant $\frakg$--valued $k$--forms on $P$ by $\Omega^k(P;\frakg)^G$.

Alternatively consider a {\em $G$--invariant horizontal distribution on $P \to M^n$}, i.e. an $n$--dimensional distribution $H$ on $T P$ satisfying:
\begin{enumerate}
 \item $H_p \subset T_pP \stackrel{d\pi}{\longrightarrow} T_{\pi(p)} M$ is an isomorphism for each $p \in P$;
 \item The distribution is $G$--invariant\footnote{The terminology $G$--equivariant would also be appropriate}, i.e. $H_{pg} = d(r_g) H_p$;
 \item $H_p$ varies smoothly with $p$.
\end{enumerate}
The $H_p$ are called the {\em horizontal subspaces}.

\begin{exe}\label{horizontaldistribution} The map from the connections on $P$ to the $G$--invariant horizontal distributions on $P$ given by $A \mapsto \ker A$ is a bijection. Use the fact that the choice of $H_p$ yields the decomposition $T_p P = H_p  \oplus V_p$, where $V_p$ is the {\em vertical subspace} spanned by all the fundamental vector fields.
\end{exe}

Therefore, there is a slick way to think about connections: A connection is a choice of a $G$--equivariant splitting of the family
\[
 \xymatrix@1{0 \ar[r]& \frakg \ar[r]^-{di_p}\ar@/_1pc/@{<.}[r]_-{A} &  T_pP \ar[r]^-{d\pi_p} \ar@/_1pc/@{<.}[r]_-{(d\pi|_{H_p})^{-1}} & T_{\pi(p)} M \ar[r] & 0}
\]
of short exact sequences parametrized by $p\in P$. This can be
viewed as a splitting of the Atiyah sequence \cite{atiyah57}. As we
have seen above, this splitting can be given by a left-inverse of
$i_p$ ({\em retraction}), namely $A$, or by a right-inverse of
$d\pi_p$ ({\em section}), namely $(d\pi|_{H_p})^{-1}$.

From Exercise \ref{horizontaldistribution} we see that a connection on $P$ determines a unique way to lift tangent vectors from $M$: If $X_x \in T_xM$ let $\tilde X_p \in T_pP$ be the unique vector in $H_p$ lying over $X_x$, where $p \in \pi^{-1}(x)$. The lift $\tilde X$ is called the {\em horizontal lift of $X$}. Since we can lift vectors from $M$ to $P$, it seems believable that we can lift a path in such a  way that the tangent vectors to the lifted path are the horizontal lifts of the tangent vectors to the path in $M$. This is another characterization of a connection and will be discussed in Section \ref{pathlifting}.

Our original definition can be modified a little by using the Maurer Cartan form. This will be discussed in Section \ref{maurercartan}

If we have a representation we can consider the vector bundle associated to a principal bundle and consider its covariant derivative, which is yet another notion of the connection and will be discussed in Section \ref{covariantderivative}.

\section{Structure of the space of connections} If $P$ is trivializable, then a trivialization $P \cong M \times G$ gives a decomposition $T_{(x,g)} P = T_x M \oplus T_g G$. Thus a trivialization gives us a choice of horizontal and vertical subspaces, this choice is called the {\em trivial connection} (with respect  to this trivialization). The corresponding $\frakg$--valued 1--form on $P$ takes a vector $X = (X_1,X_2) \in T_x M \oplus T_gG$ to $dl_{g^{-1}} X_2 \in T_e G = \frakg$ (or equivalently to the left-invariant vector field $h \mapsto dl_{hg^{-1}} X_2$). A different trivialization of the bundle determines a different trivial connection.

In the case of a trivializable bundle we can simplify the notion of a connection.
\begin{prop}\label{ConnectionsSimplified} Let $P$ be a trivializable principal $G$-bundle over a manifold
$M$. A trivialization of $P$ identifies the space of connections $\cA_P$ on $P$ with $\frakg$--valued $1$-forms $\Omega^1(M;\frakg)$ on $M$.
\end{prop}

\begin{proof}
Let $A$ be a connection on $P$. By trivializing $P = M \times G$, we can decompose the
tangent space of $P$ at $p=(x,g) \in P$ as $T_pP = T_xM
\times T_gG$.

Let $s: M \to P = M\times G$ be the section $s(x):=
(x,e)$. Then we can define an affine homomorphism
\begin{align*}
\Psi: \cA_P & \to  \Omega^1(M)\tensor \frakg\\
A &\to s^*(A).
\end{align*}

On the other hand given
$\tilde A \in \Omega^1(M;\frakg)$, where we think of $\frakg$ as $T_eG$, define a $\frakg$--valued $1$-form $\Phi(\tilde A) :=A$ on $P =
M\times G$ where $$A_{(x,g)} (v_x,w_g) = \tilde A_x(v_x) + d(l_{g^{-1}})w_g.$$
Then for $X\in \frakg$ and
$p = (x,g) \in P$ we
have $X_{p}^* = d(i_p) X = (0_x,d(l_{g})X)$ and thus $$A_p(X^*) = \tilde
A_x(0_x) + X = X.$$
Since $r_h(x,g) = (x,gh)$ we have $d(r_h)(v_x,w_g) = (0,d(r_h)w_g)$ and therefore
\begin{align*}r_h^* A (v_x,w_g) &= A(0_x,d(r_h)
w_g)= \tilde A(0_x) + d(l_{(gh)^{-1}})d(r_h) w_g\\
 &= \Ad_{h^{-1}}\tilde A(0_x) + d(l_{h^{-1}})d(r_h)d(l_{g^{-1}}) w_g =
\Ad_{h^{-1}}A (v_x,w_g).
\end{align*}
This way we have a homomorphism $\Psi: \Omega^1(M;\frakg) \to \cA_P$.

Furthermore, the maps $\Phi$ and $\Psi$ are
inverses of each other, because
\[
(\Psi\circ \Phi (\tilde A)) v_x = s^*(A_{(x,g)}) (v_x) =
A_{(x,g)}(v_x,0_g) = \tilde A_{x}v_x
\]
and
\[
(\Phi\circ \Psi (A)) (v_x,w_g) = \Psi A (v_x) + d(l_{g^{-1}})w_g =  A
(v_x,0_g) + A(0_x,w_g) = A(v_x,w_g).
\]
That is, $\cA_P \cong \Omega^1(M;\frakg)$ as (affine) vector spaces.
\end{proof}

We have already mentioned that a different trivialization changes the connection. The following exercise describes this change.

\begin{exe}\label{trivializationchange}
 Choose a connection in a trivializable bundle $P\to M$. Trivialize $P$ by choosing a section $s\co M \to P$ so that the connection pulls back to a 1--form $A \in \Omega^1(M;\frakg)$. Choose a map $h \co M \to G$ and use $h$ to define a new secion by the formula $s_h(x) = s(x)\cdot h(m)$. Then in this trivialization the connection has the form
\[
 \Ad_{h^{-1}} A + d(l_{h^{-1}})dh.
\]
This means that if $X\in T_xM$, then evaluating the pullback of the connection using $s_h$ on $X$ yields
\[
 \Ad_{h^{-1}(x)} A(X) + d(l_{h^{-1}(x)})dh(X).
\]
Notice that $dh(X) \in T_{h(x)} G$. Prove and use the product rule
\[
 \eval{\frac{d}{dt}}{t=0}s(\alpha_t)h(\alpha_t) = d(r_{h(\alpha_0)})ds(\dot\alpha_0) + d(i_{s(x)})dh(\dot\alpha_0)
\]
for a path $\alpha$ with derivative $X$, where $i_p(g) = p\cdot g$ and $r_g(p) = p\cdot g$.
\end{exe}

We have not yet discussed the existence of connections on general principal bundles, but it easily follows from the existence of connections on trivial bundles by the fact that $tA_0 + (1-t) A_1$ is a connection if $A_0$ and $A_1$ are and by using a partition of unity. The collection of all connections on $P$ forms a topological space $\cA$ which is a subset of the vector space of all $\frakg$--valued 1--forms on $P$. Even though $\cA_P$ is not a linear subspace, since the sum of two connections violates the first condition, it is, however, an affine space.

\begin{prop}
 A choice of base connection determines an isomorphism of $\cA_P$ with $\Omega^1(\Ad_P)$.
\end{prop}

\begin{proof}
 Fix a connection $A_0$. By Exercise \ref{EquivariantMapsIsomorphism} sections of $P\times_\rho V$ are the same as equivariant maps from $P$ to $V$. Let $A$ be another connection. Given a vector field $X$ on $M$, we can construct a function $\bar A_X \co P \to \frakg$ by the formula:
\[
 \bar A_X(p) = A(\tilde X_p),
\]
where $\tilde X_p$ is the horizontal lift of $X$ to $T_pP$ with respect to $A_0$. The properties of a connection show that $\bar A_X$ is indeed an equivariant map from $P$ to $\frakg$, and therefore $\bar A_X$ determines an element of $\Omega^1(\Ad_P)$.

Conversely, let $\bar A \in \Omega^1(\Ad_P)$. Define
\[
 A(Y) = \bar A(Y_1)(p) + Y_2 \quad \text{for }Y_2 \in T_pP,
\]
where $Y = Y_1 + Y_2$ is the decomposition into horizontal and vertical part with respect to $A_0$. Thus the choice of $A_0$ identifies $\cA_P$ with $\Omega^1(\Ad_P)$.
\end{proof}

\begin{exe}
 We have seen that a trivialized bundle has a natural trivial connection. We can use it as a base connection to identify $\cA_P$ with $\Omega^1(\ad g)$, which in turn is naturally isomorphic $\Omega^1\tensor \frakg$ for a trivialized bundle $P$. Show that this identification is the same as the one used in Proposition \ref{ConnectionsSimplified}.
\end{exe}

\section{Lifting paths}\label{pathlifting}

Let $A$ be a connection in a principal bundle $P\to M$. Let $\alpha\co I \to M$ be a smooth path and $p_0 \in P$ a point in the fiber over $\alpha_0$.

\begin{prop}\label{horizontallift}
 There exists a unique lifting  $\beta \co I \to P$ so that for each $t\in I$, the $\dot\beta_t$ is the horizontal lift of $\dot\alpha_t$ with starting point $p_0$. The lift $\beta$ is called the {\em horizontal lift} of $\alpha$.
\end{prop}

\begin{proof}
It is easy to see, that there exists some lift $\gamma$ of $\alpha$ to $P$ (e.g. by the lifting property or by trivializing the principal bundle over contractible components along the path). Then any path $\beta_t$ is of the form $\gamma_t\cdot g_t \in P$ for some path $g\co I \to G$. Now $\beta$ is a horizontal lift if and only if $A(\dot\beta_t) = 0$ for all $t \in I$. We can apply the product rule to get
\[
A(\dot\beta_t) = A((r_{g_t})_*\dot\gamma_t + (i_{\gamma_t})_*\dot g_t) = \Ad_{g_t^{-1}}A(\dot \gamma_t) + (l_{g_t^{-1}})_* \dot g_t.
\]
Therefore $\beta$ is a horizontal lift if and only if $A(\dot \gamma_t) = -(r_{g_t^{-1}})_*\dot g_t$. To find the appropriate $g$ with starting point $g_0$ is just a matter of solving an ordinary differential equation in $G$.
\kommentar{Add note about notation d() = ()_*}
\end{proof}

If $\alpha$ is a loop, i.e. $\alpha(0) = \alpha(1)$, then the starting point $p_0$ and the endpoint of the lift $\beta$ live in the same fiber and we can find an element $g \in G$ with $\beta(1) = \beta(0) \cdot g$. We call $\hol_{A}(\alpha,p_0) := g$ the holonomy of $A$ along $\alpha$ with respect to the base point $p_0$. Therefore we get the following from Proposition \ref{pathlifting}

\begin{cor}
 Given a base point $p_0 \in \pi^{-1}(x_0)$, any connection $A$ defines a map $\hol_A\co \operatorname{Loops}(M,x_0) \to G$ called the holonomy map.
\end{cor}

This holonomy map {\em does not} define a map on $\pi_1$, unless the connection is {\em flat}, a term we will define later. One can recapture the connection by knowing what the horizontal lifts of all paths are, just define the horizontal subspace at $p\in P$ to be the space spanned by the derivatives of the horizontal lifts of all paths throough $\pi(p)$. One can therefore give yet another definition of a connection in a principal bundle to be a coherent way to lift paths in $M$ so that equivariance is satisfied.

\section{Connections in vector bundles}\label{covariantderivative}

We now turn to connections in vector bundles. One should think of these as a way to differentiate sections of vector  bundles in the direction of a vector field on the base manifold. For example, given a function $f\co \R^n \to \R$, and a vector field $X$ on $\R^n$, then we can take the derivative of $f$ along the vector field to get another function $df(X)\co \R^n \to \R$. Similarly, if we have a (vector valued) function $f\co \R^n \to \R^p$, and a vector field $X$ on $\R^n$ we can apply the preceding procedure componentwise to obtain a new function $df(X)\co \R^n \to \R^p$. Again this can be done by replacing $\R^n$ by any manifold $M$, thus there is a natural way to take derivatives of functions $M\to \R^p$ along vector fields. But functions are just sections of trivial bundles, and the general notion of a connection in a vector bundle is a way to do this when the bundle is not trivial.

\begin{defn}
 An {\em affine connection}, also called {\em covariant derivative} on $E$, is a linear mapping
\[
 \nabla\co \Omega^0(E) \to \Omega^1(E)
\]
satisfying the Leibniz rule
\[
 \nabla(f\phi) = df \tensor \phi + f \nabla \phi \quad \text{for } f\in C^\infty(M,\R) \text{ and } \phi \in \Omega^0(E).
\]
Here we use that $\Omega^0(E)$ is isomorphic to $\Gamma(E)$ via the map $f\tensor \phi \mapsto f\phi$, since $\Lambda^0T^*M$ is just the trivial line bundle.
\end{defn}

 If $X$ is a vector field on $M$, then we will write $\nabla_X(\phi)$ for the evaluation $(\nabla\phi)(X)$. If $E$ is associated to a principal $U(n)$ or $\SO(n)$ bundle via the standard representation, then $E$ has a fiberwise inner product $(\cdot,\cdot)$. We say $\nabla$ is {\em compatible with the metric} if
\[
 d(\phi,\theta) = (\nabla\phi,\theta) + (\phi,\nabla\theta).
\]
This expression makes sense since $(\phi,\theta)$ is a smooth function on $M$, thus $d(\phi,\theta) \in \Omega^1(M)$. Also $(\nabla \phi,\theta) \in \Omega^1(M)$ since we can evaluate it on a vector field $X\co (\nabla \phi,\theta)(X) = (\nabla_X\phi,\theta)$. Compatibility with the metric refers to the metric on the vector bundle, {\em not} the Riemannian metric on $M$.

\begin{exa}
 Let $E = M \times \C^n$, the trivial $\C^n$ bundle. Sections of $E$ are simply functions $\phi\co M \to \C^n$. Let $d\co \Omega^0 \to \Omega^1$ denote the exterior derivative. It is defined for complex valued functions on $M$. Then $d$ defines the trivial connection on $E$:
\begin{align*}
 d\co \Omega^0(E) &\to \Omega^1(E)\\
\phi = (\phi_1,\ldots,\phi_n) &\mapsto (d\phi_1,\ldots,d\phi_n).
\end{align*}
Using this example and a partition of unity we see that any vector bundle admits an affine connection.
\end{exa}

For $E = M \times \C^n$ and $A\in \Omega^1(M)\tensor \frakg$, where $G \subset GL_n(\C)$, we can consider
\begin{align*}
 d^A\co &\Omega^0(E) \to \Omega^1(E)\\
\phi = (\phi_1,\ldots,\phi_m) &\mapsto (d\phi_1,\ldots,d\phi_m) + A\phi.
\end{align*}
The term $A\phi$ is a 1--form with values in $E$, since given a vector field $X$, $A(X) \in \frakg$ which acts on sections of $E$ by left matrix multiplication.

\begin{exe}
 Show that $d^A$ is an affine connection on $E$. If $G= U(n)$, show that $d^A$ is compatible with the (trivial) Hermitian metric on $E$ if and only if $A \in \Omega^1(M)\tensor\fraku(n)$.
\end{exe}

\begin{exa}\label{localconnection}
We will now show that a connection on any $\GL(n,\C)$--vector bundle $E$ can be written locally as $d+A$ for some $\frakgl(n,\C)$--valued 1--form $A$  on $M$. A convenient way to show this is to use local frames. Let $\{ e_i \}$ be a local frame for a vector bundle $E$. Let $\nabla$ be a covariant derivative and let $A_{r,s}$ be the matrix of 1--forms determined by the expression
\[
 \nabla e_s = \sum_r A_{r,s} \tensor e_r.
\]
(If the structure group of $E$ is $O(n)$ or $U(n)$ and the $e_i$ are orthogonal then $A_{r,s}$ is a $\frako(n)$-- or a $\fraku(n)$--valued form if $\nabla$ is compatible with the metric.) Now if $\phi\in \Gamma(E)$, we can write $\phi = \sum \phi_i e_i$ with $\phi_i\in C^\infty$ and so we must have
\begin{align*}
 \nabla \phi & = \sum d\phi_r e_r = \sum_r(d\phi_r \tensor e_r + \phi_r \nabla e_r)\\
& = \sum_r(d\phi_r \tensor e_r + \phi_r \sum_s A_{s,r} \tensor e_s)\\
& = \sum_r(d\phi_r + \sum_s A_{r,s}\phi_s ) \tensor e_r.
\end{align*}
Writing $\phi$ as a column vector, we can write this as
\[
 \nabla\phi = (d+A) \phi.
\]
Hence in local coordinates any connection is of the form $d+A$.

This example hints (strongly) at the relationship between the definition in principal bundles and in vector bundles. We saw that in a trivial principal bundle, a connection is determined by a Lie algebra valued 1--form on $M$. The same is true in the preceding example. One sees a local equivalence of the two notions of connections. In the next section we will give a global identification of these two notions, without any reference to a trivialization of the bundle.
\end{exa}

\begin{exa}
 The {\em Fundamental Theorem of Riemannian Geometry} states that every Riemannian manifold possesses a unique connection on its tangent bundle which satisfies two conditions (compatible with the metric and torsion free\footnote{$\nabla$ is called {\em torsion-free} if $\nabla_X Y-\nabla_Y X =[X,Y]$ for all vector fields $X$ and $Y$ on $M$}). Thus invariants of Riemannian manifolds can be constructed by working with this connection called the {\em Riemannian} or {\em Levi-Civita} connection. It should be stressed that the Levi-Civita connection only depends on the metric.
\end{exa}

\section{Equivalence of connections in vector bundles and principal bundles}\label{equivalenceconnectionsvectorbundlesprincipalbundles}

We have seen that the frame bundle of a vector bundle $E$ is a principal bundle $P$, so that the vector bundle associated to $P$ via the defining representation is $E$ again. Therefore, principal bundles are more general than vector bundles and we do not loose any generality if we always start with a principal bundle. We will see that a connection in a principal bundle always gives us a connection in an associated vector bundle. Furthermore, if we restrict ourselves to principal $G$-bundles where $G\subset GL(n,\C)$ or $G\subset GL(n,\R)$, the notions of connections are equivalent.

Let $E = P\times_\rho V \to M$ be a vector bundle associated to a principal $G$--bundle $P\to M$ and $\rho\co G \to GL(V)$ a representation of $G$. We saw that the sections of $E$ are $G$--equivariant maps $\phi\co P \to V$. Notice that for a vector space $V$, the tangent space $T_vV$ is canonically identified with $V$ itself.

Let $A$ be a connection in the principal bundle $P \to M$; we think of $A$ as a horizontal distribution. Given a $G$--equivariant map $\phi\co P \to V$ and a vector field $X$ on $M$, we want to construct another $G$--equivariant map $\nabla^A_X \phi \co P \to V$. Given $p \in P$ let $\tilde X_p$ be the unique horizontal lift of $X_{\pi(p)}$ to $T_pP$ given by the connection. Define $(\nabla^A_X(\phi))(p)$ to be $d\phi(\tilde X_p) \in T_vV =  V$. This construction yields a bijection between connections in principal $\GL(n,\C)$--bundles and connections in $\GL(n,\C)$--vector bundles associated via the defining representation. Naturally, this will only be a bijection, when we can reconstruct the principal bundle from the vector bundle, e.g. by considering the frame bundle.

\begin{exe}
 Verify that the above construction defines a bijection between connections in principal $U(n)$-- or $O(n)$--bundles and affine connections in $\R^n$-- or $\C^n$--vector bundles compatible with the metric.
\end{exe}

Path lifting in a principal bundle corresponds to {\em parallel transport} of frames in a vector bundle. If I choose a frame in the fiber of $E$ over $m\in M$ it is easy to see that lifting a path $\alpha_t$ starting at $p\in P$ gives me a way to choose a frame in the fiber over each point of the path, i.e. it hands me a trivialization of the pullback bundle $\alpha^*(E)$ over an intervall. This is one of the older definitiions of a connection, the ``rep\`ere mobile'' (moving frame) of \'Elie Cartan.

\section{Extending the covariant derivatives}

We can extend an affine connection $\nabla$ in a vector bundle $E$ to
\[
 \nabla\co \Omega^k(E) \to \Omega^{k+1}(E)
\]
by forcing the Leibniz rule to hold; so
\[
 \nabla(\omega \tensor \phi) = d\omega \tensor \phi + (-1)^{k+1} \omega \wedge \nabla\phi \quad\text{for }\omega\in\Omega^k(M) \text{ and }\phi\in\Gamma(E).
\]
The resulting sequence
\[
\cdots \to \Omega^{k-1}(E) \stackrel{\nabla}{\to} \Omega^{k}(E) \stackrel{\nabla}{\to} \Omega^{k+1}(E) \to \cdots
\]
is not necessarily a complex. In fact, we will see in the Section \ref{curvature} that this sequence is a complex if and only if $\nabla$ is {\em flat}.

There are two special cases of this which are particularly important, namely the bundle associated to a principal bundle $P\to M$ via the defining representation and the adjoint bundle. If $E$ is the bundle associated to $P$ via the defining representation (where $G$ is a subgroup of $\GL(n,\R)$ or $\GL(n,\C)$) then the adjoint bundle is a sub-vector bundle $\Hom(E,E)$, since $\Hom(E,E) = P\times_{\Ad} \frakgl(n)$. If you choose a local basis of sections $\{e_i\}$ of $E$, then the covariant derivative on $E$ has the form $d^A = d+A$ for some $A \in \Omega^1 \tensor \frakg$, as explained in Section \ref{covariantderivative}. Furthermore if we fix a connection $\nabla$, then any other connection can be written as $\nabla^A = \nabla+A$ globally for some $A \in \Omega^1(\Ad_P)$.

Suppose we are given two differential forms $\alpha\in
\Omega^p(E)$ and $\beta \in \Omega^q(E')$, then a priori we have $\alpha \wedge \beta \in
\Omega^{p+q}(E \tensor E')$ by extending $(\alpha\tensor \phi) \wedge ( \alpha' \tensor \phi') :=  (\alpha\wedge\alpha')\tensor (\phi\tensor \phi')$ linearly. But if we are given a section of
$\Hom(E\tensor E',E'')$, i.e. a smoothly varying bilinear form $\Phi\co E_x
\times E'_x \to E''_x$ on the fibers over $x\in M$,
then using this form we may define the product $\Phi(\alpha \wedge \beta)
\in \Omega^{p+q}(E'')$ by extending $\Phi((\alpha \tensor \phi ) \wedge (\alpha' \tensor \phi')):=(\alpha\tensor\alpha')\tensor\Phi(\phi,\phi')$ linearly.

Typically $E=E'=E''$ is the Lie algebra bundle of
some principal $G$-bundle. We
will write $[\alpha\wedge\beta] = \Phi(\alpha\wedge\beta)$, when the $\Phi$ is
the Lie bracket (see Exercise \ref{liebracketadjointbundle}), and $\alpha\wedge\beta = \Phi(\alpha\wedge\beta)$, when $\Phi$ is ordinary
matrix multiplication for a matrix group $G$ or evaluation $\Hom(E,E) \tensor E \to E$. Notice that for $\alpha\in \Omega^p(E)$ and $\beta\in \Omega^q(E)$  where $E$ is a vector bundle with fiber isomorphic to a Lie algebra $\frakg$ (e.g. $E = \Ad_P$) it is easy to see that the Lie-bracket satisfies
\begin{equation}\label{gradedliebracket}
[\alpha\wedge\beta] = -(-1)^{pq}[\beta\wedge\alpha]
\end{equation}
as well as the graded Jacobi identity
\begin{equation}\label{gradedJacobiidentity}
(-1)^{pr} [\alpha\wedge[\beta\wedge\gamma]] + (-1)^{pq} [\beta\wedge[\gamma\wedge\alpha]] + (-1)^{qr} [\gamma\wedge[\alpha\wedge\beta]]
\end{equation}
for $\alpha\in\Omega^p(E),\beta\in\Omega^q(E),\gamma\in\Omega^r(E).$

\begin{exa}\label{localextended}
 Let $A$ be a connection in a principal $G$--bundle $P\to M$ with $G\subset \GL(n,\C)$ and let $E$ be the vector bundle associated to $P$ via the defining representation. By Example \ref{localconnection} locally $\nabla^A = d+A$, if we write $\phi \in \Gamma(E)$ as a column vector. Therefore the induced connection $d^A\co \Omega^1(E) \to \Omega^2(E)$ is satisfies:
\begin{align*}
d^A (\omega \tensor \phi) &= d\omega \tensor \phi - \omega \wedge \nabla^A\phi\\
&=d\omega\tensor\phi - \omega\wedge\phi - \omega\wedge A\phi\\
&=d(\omega\tensor\phi) + A\phi \wedge \omega\\
&=d(\omega\tensor\phi) + A\wedge (\omega\tensor\phi).
\end{align*}
Therefore locally $d^A\phi = d\phi+A\wedge\phi$.
\end{exa}

\begin{exe}
Show that the induced connection $d^A\co \Omega^1(\Hom(E,E)) \to \Omega^2(\Hom(E,E))$ is given by the formula
\[
 B\mapsto dB + [A\wedge B].
\]
\end{exe}

\section{The Maurer-Cartan form}\label{maurercartan}

In this section we will see yet another equivalent, however only slightly different definition for a connection.

\begin{defn}\label{defnMaurerCartan}
 The {\em Maurer-Cartan form} $\theta$ on $G$ is the unique $\frakg$--valued 1--form on $G$ which assigns to each vector its left-invariant extension.
\end{defn}

Notice that for every smooth vector field $X$ we get a smooth map
\begin{align*}
\theta(X)\co G &\to \frakg\\
g&\mapsto \theta(X_g).
\end{align*}

\begin{exe}\label{ExeMaurerCartan}
Check that $\theta(X_g)_h = dl_{hg{-1}} (X_g)$ and that it satisfies
\begin{enumerate}
 \item $l_g^*\theta = \theta$;
 \item $r_g^*\theta = \Ad_{g^{-1}} \theta$.
\end{enumerate}
\end{exe}

\begin{prop}\label{PropMaurerCartan}
 The Maurer-Cartan form satisfies the Maurer-Cartan equation \[d\theta + \tfrac{1}{2}[\theta\wedge\theta] = 0.\]
\end{prop}
\begin{proof}
By Exercise \ref{ExeMaurerCartan} we have for a left-invariant vector fields $X$
\[
 \theta(X_g)_h = dl_{hg^{-1}} (X_g)  = X.
\]
Therefore the map $\theta(X)\co G \to \frakg$ is constant and $d(\theta(X)) = 0$. Furthermore we have for left-invariant vector fields $X$ and $Y$
\[
 \theta([X,Y]) = [X,Y] = [\theta(X),\theta(Y)],
\]
which implies\kommentar{Vielleicht Rerefenz fuer Strukturgleichung der aeusseren Ableitung. Warum weicht Nomizu ab?}
\begin{align*}
 d\theta(X,Y) &= X (\theta(Y)) - Y (\theta(X)) - \theta([X,Y])\\
& = d(\theta(Y))(X) - d(\theta(X))(Y) - [\theta(X),\theta(Y)]\\
& = -[\theta(X),\theta(Y)].
\end{align*}
The evaluation of either side at $g\in G$ does not rely on the fact that $X$ and $Y$ are left-invariant vector fields.and therefore $d\theta(X,Y) + [\theta(X),\theta(Y)] = 0$ holds for arbitrary smooth vector fields. For $\alpha_i\in \Omega^1(G)$ and $\phi_i\in \frakg$, $i=1,2$, we have
\begin{align*}
[\alpha_1(X)\tensor \phi_1,&\alpha_2(Y)\tensor \phi_2] + [\alpha_2(X)\tensor \phi_2,\alpha_1(Y)\tensor \phi_1]\\ &= \alpha_1(X)\alpha_2(Y)(\phi_1\phi_2-\phi_2\phi_1) + \alpha_2(X)\alpha_1(Y)(\phi_2\phi_1-\phi_1\phi_2)\\
&=\alpha_1\wedge\alpha_2(X,Y)\phi_1\phi_2+\alpha_2\wedge\alpha_1(X,Y)\phi_2\phi_1\\
&=\alpha_1\wedge\alpha_2(X,Y)[\phi_1,\phi_2]\\
&=[(\alpha_1\tensor\phi_1)\wedge(\alpha_2\tensor\phi_2)](X,Y).
\end{align*}
This implies $[\theta(X),\theta(Y)] = \tfrac{1}{2}[\theta\wedge\theta](X,Y)$ and completes the proof.
\end{proof}

Let $P \stackrel{\pi}{\to} M$ be a principal $G$--bundle and consider $p \in P_x$ for some $x\in M$, where $P_x = \pi^{-1}(x)$ is the fiber over $x$. The map $i_p \co G \hookrightarrow P_x$ given by $i_p g = p\cdot g$ induces the $\frakg$--valued 1--form $\theta_x:=i_p^*\theta$ on $P_x$ which satisfies
\[
 r_g^*(\theta_x) = \Ad_{g^{-1}}\theta_x\quad \text{and} \quad d\theta_x + [\theta_x\wedge \theta_x]= 0.
\]

\begin{exe}\label{ExeConnectionMaurer}
Show that replacing the first condition in Definition \ref{defnconnection} by \[j_x^* A = \theta_x \quad \text{for all } x \in M,\] where $j_x\co P_x \hookrightarrow P$, yields an equivalent definition of a connection.
\end{exe}

\section{Curvature}\label{curvature}

Let $\nabla$ be a connection in a vector bundle $E$.

\begin{thm}\label{curvatureformula}
The composition
\[
 \Omega^0(E) \stackrel{\nabla}{\longrightarrow} \Omega^1(E) \stackrel{\nabla}{\longrightarrow} \Omega^2(E)
\]
is a differential operator of order 0. More precisely, there exists an $F^\nabla \in \Omega^2(\Hom(E,E))$ such that
\[
 \nabla\nabla \phi = F^\nabla \cdot \phi.
\]
$F^\nabla$ is called the curvature of the affine connnection $\nabla$.\kommentar{ If $E$ is associated to a principal $G$--bundle $P$ via the defining representation and $\nabla = \nabla^A$ for some connection $A$ on $P$, then it actually lies in the subspace $\Omega^2(\Ad_P) \subset \Omega^2(\Hom(E,E))$.}
\end{thm}

\begin{proof}
 We give a local proof. Choose a local frame $\{e_i\}$. In this frame the connection has the form $\nabla = d+A$ by Example \ref{localconnection}. Let $\phi$ be a section and write it as a column vector. Then by Examples \ref{localconnection} and \ref{localextended}
\begin{align*}
 \nabla \nabla \phi &= \nabla(d\phi + A\phi)\\
 &=d(d\phi + A\phi) + A\wedge (d\phi + A \phi)\\
 &=d^2 \phi + d(A\phi) + A\wedge d\phi + A\wedge (A \phi)\\
&=A\wedge d\phi + dA \wedge \phi - A\wedge d\phi + (A\wedge A) \phi\\
&= (dA + A\wedge A) \phi\\
&= (dA + \tfrac{1}{2}[A\wedge A]) \phi
\end{align*}
where the last equality follows from \eqref{gradedliebracket} and
\[
A\wedge A = \sum_{k} A_{r,k}\wedge A_{k,s}.\qedhere
\]
\kommentar{The fact that $F^\nabla \in \Omega^2(\Ad_P)$ if $E$ is associated to a principal $G$--bundle $P$ via the defining reprresentation and $\nabla = \nabla^A$ for some connection $A$ on $P$ follows immediately from the fact that $\nabla^Ae_i$ is $G$--equivariant that the frames are related by elements in $G$.}
\end{proof}

\kommentar{
Let $A$ be a connection in a principal $G$--bundle $P\to M$ where $G \subset \GL(n,\C)$, $d^A$ its associated covariant derivative in the vector bundle $E$ associated to $P$ via the defining representation.

\begin{thm}
 The composition
\[
 \Omega^0(E) \stackrel{d^A}{\longrightarrow} \Omega^1(E) \stackrel{d^A}{\longrightarrow} \Omega^2(E)
\]
is a differential operator of order 0. More precisely, there exists an $F^A \in \Omega^2(\Hom(E,E))$ such that
\[
 d^A d^A \phi = F^A \cdot \phi.
\]
$F^A$ is called the curvature of the affine connnection $d^A$, and it actually lies in the subspace $\Omega^2(\Ad_P) \subset \Omega^2(\Hom(E,E))$.
\end{thm}

\begin{proof}
 We give a local proof. Choose a local frame $\{e_i\}$. In this frame the connection has the form $d^A = d+A$. Let $\phi$ be a section and write it as a column vector. Then by Examples \ref{localconnection} and \ref{localextended}
\begin{align*}
 d^A d^A \phi &= d^A(d\phi + A\phi)\\
 &=d(d\phi + A\phi) + A\wedge (d\phi + A \phi)\\
 &=d^2 \phi + d(A\phi) + A\wedge d\phi + A\wedge (A \phi)\\
&=A\wedge d\phi + dA \wedge \phi - A\wedge d\phi + (A\wedge A) \phi\\
&= (dA + A\wedge A) \phi\\
&= (dA + \tfrac{1}{2}[A\wedge A]) \phi
\end{align*}
where
\[
A\wedge A = \sum_{k} A_{r,k}\wedge A_{k,s}.\qedhere
\]
\end{proof}
}

\begin{exe}
Give a global proof of Theorem \ref{curvatureformula}:
\begin{enumerate}
 \item Show that $\nabla \nabla (f\phi) = f \nabla \nabla (\phi)$ for $f\in C^\infty(M)$ and $\phi \in \Omega^0(M)$.
 \item Let $L\co \Omega^0(E) \to \Omega^2(E)$ be any $C^\infty(M)$--linear map. Then there is an $\hat L \in \Omega^2(\Hom(E,E))$ such that $L(\phi) = \hat L \wedge \phi$ for all $\phi \in \Omega^0(E)$.
\end{enumerate}
\end{exe}

Now we are ready to define the curvature of connections in principal $G$--bundles.

\begin{defn}
Let $A$ be a connection on a principal $G$--bundle $P\to M$ and consider the induced affine connection $d^A$ in the adjoint bundle $\Ad P$ as described in Section \ref{equivalenceconnectionsvectorbundlesprincipalbundles}. The curvature $F^A$ of a connection $A$ in a principal $G$--bundle $P\to M$ is defined to be the local invariant $F^{d^A}$. We will see below that the curvature $F^A$ of $d^A$ can be viewed as an element of $\Omega^2(\Ad_P)$.
\end{defn}

We denote by $\tilde d^A$ the pull-back of $d^A$ to the trivial vector bundle $P \times \frakg = \pi^*(\Ad_P)$ via $\pi\co P \to M$. It restricts to $G$--equivariant forms on $P$. This yields a sequence
\[
\cdots \longrightarrow \Omega^{k-1}(P;\frakg)^G \stackrel{\tilde d^A}{\longrightarrow} \Omega^{k}(P;\frakg)^G \stackrel{\tilde d^A}{\longrightarrow} \Omega^{k+1}(P;\frakg)^G \longrightarrow \cdots.
\]
Considering the connection on the trivial bundle has the advantage of finding a global formula for it. The curvature of the $G$--equivariant version of $\tilde d^A$ is also often called the curvature of $A$, but it turns out to be the pull-back of a $2$--form in $\Ad P$, which we used as a definition.

\begin{prop} We have the (global) formulas for $\phi \in \Omega^k(P;\frakg)^G$:
\begin{enumerate}
 \item $\tilde d^A\phi = d\phi + [A\wedge\phi]$;
 \item $\tilde d^A\tilde d^A\phi = [(dA + \frac{1}{2}[A\wedge A])\wedge\phi]$.
\end{enumerate}
\end{prop}

\begin{proof} Let $H$ be the horizontal distribution corresponding to a connection $A$. Recall, that given a $G$--equivariant map $P\to \frakg$ and a vector field $X$ on $M$, a section $\phi$ of $\Ad P$ viewed as a $G$--equivariant map $P \to \frakg$ and a point $p\in P$, then $d^A$ is given by
\[
(d^A_X(\phi))(p) = d\phi(\tilde X_p), \quad \text{where } \tilde X_p = (d\pi|_{H_p})^{-1} X_{\pi(p)} \text{ is the horizontal lift of } X_{\pi(p)}.
\]
Given $Z_p \in T_pP$, the connection $A$ induces the decomposition
\[
 Z_p = \tilde X_p + Y^*_p, \quad \text{where } Y^*_p = di_p Y_e
\]
for some $Y \in \frakg$. Then given a $G$--equivariant map $\phi\co P \to \frakg$, the a $G$--equivariant map $\tilde d^A_Z(\phi)\co P \to \frakg$  given by
\begin{align*}
(\tilde d^A_Z(\phi)) (p)  &= d\phi((d\pi_p|_{H_p})^{-1}(d\pi (Z_p)))\\
& = d\phi (\tilde X_p)\\
& = d\phi(Z_p) - d\phi(Y^*_p)\\
& = d\phi(Z_p) - d(\phi\circ i_p) Y_e.
\end{align*}
Notice that $\phi\circ i_p (g) = \phi(p\cdot g) = \Ad_{g^{-1}} \phi(p)$. By Theorem \ref{adandliebracket} we therefore have
\begin{align*}
d(\phi\circ i_p) Y_e &= \ad (\phi(p)) (Y_e) = [\phi(p),Y_e]\\
&=-[Y_e,\phi(p)] = -[A(Y^*_p),\phi(p)]\\
& = -[A(Z_p),\phi(p)].
\end{align*}
This shows that $\tilde d^A(\phi) = d\phi + [A,\phi]$ for $\phi \in \Omega^0(P;\frakg)^G$. Now consider $\omega \in \Omega^k(P)$ for $k>0$ and $\phi \in \Omega^0(P;\frakg)^G$. We use induction and the Leibnitz rule:
\begin{align*}
d^A(\omega \tensor \phi) & = d\omega \tensor \phi + (-1)^k\omega \wedge d^A\phi\\
& = d\omega \tensor \phi + (-1)^k \omega \wedge d\phi + (-1)^k \omega \wedge [A \wedge\phi]\\
& = d(\omega \tensor \phi) + [A\wedge \phi] \wedge \omega\\
& = d(\omega \tensor \phi) + [A\wedge (\omega\wedge \phi)].
\end{align*}
This completes the computation of $\tilde d^A$. Then for $\phi \in \Omega^k(P;\frakg)^G$
\begin{align*}
\tilde d^A \tilde d^A \phi & = d^2 \phi + d[A\wedge \phi] + [A \wedge[A\wedge \phi]] + [A\wedge d\phi]\\
 & = [dA \wedge \phi] - [A\wedge d\phi] - [A\wedge [A \wedge \phi]] + [A\wedge d\phi].
\end{align*}
The graded Jacobi identity gives us
\begin{align*}
 [ A \wedge [A \wedge \phi]] - [A  \wedge [\phi \wedge A]] + [\phi  \wedge [A \wedge A]] &= 0 \quad \text{for }\phi \in \Omega^\text{even}(P;\frakg)^G\\
- [ A \wedge [A \wedge \phi]] - [A  \wedge [\phi \wedge A]] - [\phi  \wedge [A \wedge A]] &= 0 \quad \text{for }\phi \in \Omega^\text{odd}(P;\frakg)^G.
\end{align*}
Therefore graded symmetry implies $2[A\wedge [A \wedge \phi]] = [[A \wedge A]\wedge \phi]$ in both cases and consequently
\[
 \tilde d^A \tilde d^A \phi  = [(dA + \tfrac{1}{2}[A\wedge A]) \wedge \phi].\qedhere
\]
\end{proof}

It follows immediately that $\pi^*(F^A) = F^{\tilde d^A}$. We want to see that $F^{A} \in \Omega^2(\Ad_P)$, where
\begin{align*}
 \Omega^k(\Ad_P) &\hookrightarrow \Omega^k(\Hom(\Ad_P,\Ad_P))\\
\omega &\mapsto f \quad \text{where }f(\phi)= [\omega,\phi]\text{ for }\phi\in \Gamma(\Ad_P).
\end{align*}
It as apparent from $(\tilde d^A_Z(\phi)) (p)  = d\phi (\tilde X_p)$ for $Z_p = \tilde X_p + Y^*_p$, that \[(\tilde d^A_{Y^*}(\phi))(p)= 0.\]
Therefore,
\begin{align}
(dA + \tfrac{1}{2}[A\wedge A]) (Z_p,Z'_p) &= (dA + \tfrac{1}{2}[A\wedge A]) (\tilde X_p,\tilde X'_p)\label{F^Avertical}\\
&= (dA + \tfrac{1}{2}[A\wedge A]) ((d\pi|_{H_p})^{-1}X_{\pi(p)},(d\pi|_{H_p})^{-1}X'_{\pi(p)})\nonumber
\end{align}
Let $q\in M$ be arbitrary and $p\in \pi^{-1}(q)$ any point in the fiber over $q$. Define the curvature $F^A \in \Omega^2(\Ad_P)$ by
\[
F^A (X_q,X'_q) (p):= (dA + \tfrac{1}{2}[A\wedge A]) ((d\pi|_{H_p})^{-1}X_{q},(d\pi|_{H_p})^{-1}X'_{q}).
\]
where we consider the above section of $\Ad_P$ as a $G$--equivariant map $P\to \frakg$. $F^A$ is well-defined, since by $G$-invariance of the distribution $H$ we have
\[
d\pi|_{H_{pg}} = d\pi|_{d(r_g)H_p} = d\pi|_{H_p} \circ (d(r_g))^{-1},
\]
and by the $G$--equivariance of $A$ we get
\begin{align*}
F^A (X_q,X'_q) (pg) &= (dA + \tfrac{1}{2}[A\wedge A]) ((d\pi|_{H_{pg}})^{-1}X_{q},(d\pi|_{H_{pg}})^{-1}X'_{q})\\
&= (dA + \tfrac{1}{2}[A\wedge A]) ((dr_g)\circ (d\pi|_{H_p})^{-1}X_{q},(dr_g)\circ (d\pi|_{H_p})^{-1}X'_{q})\\
&= \Ad_{g^{-1}}(dA + \tfrac{1}{2}[A\wedge A]) ((d\pi|_{H_p})^{-1}X_{q},\circ (d\pi|_{H_p})^{-1}X'_{q})\\
&=\Ad_{g^{-1}}F^A (X_q,X'_q) (p).
\end{align*}
This immediately gives the following proposition.

\begin{prop}\label{PropCurvaturePullback} There is a unique 2--form $F^A \in \Omega^2(\Ad_P)$ such that for $\phi \in \Omega^{k}(\Ad_P)$
\[d^Ad^A \phi = [F^A,\phi],\]
where $\pi^*(F^A) = dA + \frac{1}{2}[A\wedge A] \in \Omega^k(P;\frakg)^G$.
\end{prop}

The following exercise shows that the definitions of curvature of connections in vector and principal bundles are essentially equivalent.

\begin{exe}
 If $P\to M$ is a prinicpal bundle and $E$ is associated to $P$ via a representation $\rho\co G \to \GL(V)$. Show that  the $G$--equivariant linear map $d\rho\co \frakg \to \Hom(V,V)$  induces a map \[d\rho\co \Omega^k(\Ad_P) \to \Omega^k(\Hom(E,E))\] which satisfies
$d\rho(F^A) = F^A_\rho$, where $F^A_\rho$ is the curvature of the affine connection $d^A_\rho$ in $E$ induced by $A$.
\end{exe}

\begin{defn}
 A connection $A$ is {\em flat} if $F^A = 0$ pointwise. Denote by $\cF_P$ the space of flat connections.
\end{defn}

The above exercise shows, that for any vector bundle $E$ associated to a principal bundle $P$ via a representation $\rho\co G \to \GL(V)$ the sequence
\[
 \cdots \Omega^{k-1}(E) \stackrel{d^A_\rho}{\to} \Omega^k(E) \stackrel{d^A_\rho}{\to} \Omega^{k+1}(E) \to \cdots
\]
is a complex, if $A$ is a flat connection.

\begin{prop}\label{involutiveflat}
 Let $H$ be the horizontal distribution associated to a connection $A$. Then $H$ is involutive if and only if $A$ is flat.
\end{prop}

\begin{proof}
 Let $H$ be the horizontal distribution associated to a connection $A$ and $X,Y$ smooth vector fields in $H$. Since $A(X) = 0$ and $A(Y)=0$ we have
\[
 dA(X,Y) = X A(Y) - Y A(X) - A([X,Y]) = -A([X,Y]).
\]
It follows that
\[
F^A(d\pi (X),d\pi (Y)) = dA(X,Y) + \frac{1}{2}[A,A](X,Y) = -A([X,Y]),
\]
which implies that $[X,Y] \in \ker A = H$ if and only if $F^A = 0$. Therefore $H$ is involutive if and only if $A$ is flat.
\end{proof}

\begin{thm}\label{ThmHolPi1}
 If $A$ is flat, then the holonomy representation
\[
 \hol_A\co \text{Loops}(M,x_0) \to G
\]
is well-defined on the fundamental group, and furthermore, the
homology on the chain complex described above is just the cohomology
of $M$ with local coefficients in $V$ given by the composite
\[
 \pi_1(M,x_0) \stackrel{\hol_A}{\longrightarrow} G \stackrel{\rho}{\longrightarrow} \GL(V).\qedhere
\]
\end{thm}

\begin{proof}
Two loops $\gamma$ and $\gamma'$ based at $x_0$ are homotopic, if and only if the concatenation of the path $\gamma_t$ followed by $\gamma_{-t}$ is contractible. Therefore, to see that $\hol_A$ is well-defined on $\pi_1(M,x_0)$, it suffices to show that $\hol_A$ is the identity on any contractible loop. Consider a homotopy $h\co I \times I \to M$ with $h(s,0) = h(s,1) = h(1,t)= x_0$ from a loop to the constant loop. It follows from Frobenius' Theorem (Theorem \ref{FrobeniusTheorem}) and Proposition \ref{involutiveflat} that the horizontal distribution associated to a flat connection has a flat chart in a neighborhood of every point $p\in P$. Since $I\times I$ is contractible there is a flat chart in a neighborhood of $\pi^{-1}(h(I\times I))$. In particular, this implies that the horizontal lift of $h(0,t)$ in $M$ is homotopic to the constant path and that $\hol_A$ is well-defined on $\pi_1(M)$\footnote{Note that this argument also takes care of reparametrizations of the loops in $M$}.
\kommentar{Bild dafuer waere gut}

The statement about the homology of the complex is a generalization of de Rham's Theorem, that the de Rham cohomology is isomorphic to singular cohomology, is standard. We can use sheaf theory to prove de Rham's Theorem as in Bredon \cite{bredon1997} as well as the above generalization.
\end{proof}

\begin{exe}\label{ExeBianchi}
Prove the Bianchi identity
\[
d (\pi^*(F^A)) + [A\wedge\pi^*(F^A)] = 0.
\]
\end{exe}

\kommentar{Vielleicht sollte ich alle F^ und d^ umwandeln in d_ und F_...?}

\section{Gauge transformations}

In the future it will actually make more sense to let $F^A \in \Omega^2(P;\frakg)^G$.

\begin{defn}
 A {\em principal bundle homomorphism} $\Phi\co P\to P'$ is a $G$--equivariant fiber bundle homomorphism. If $P = P'$, then $\Phi$ is an {\em automorphism} or a {\em gauge transformation}. Denote the group of gauge transformations by $\cG_P$.
\end{defn}

We identify the group of gauge transformations with the group of $G$--equivariant maps $u\co P\to G$ by $u\mapsto \Phi(p) = p\cdot u_p$, where $G$ acts on itself on the right by $C_{g^{-1}}$ and $C_gh= ghg^{-1}$ is the conjugation action:
\[
 \cG_P = \Gamma(P\times_C G).
\]

\begin{lem}\label{LemGaugeAssociatedTransform}
There is a natural action of $\cG_P$ on $\cA$ defined via pull-back. In terms of $G$--equivariant maps $u\co P\to G$ it takes the form
\begin{align*}
 \cA_P \times \cG_P &\to \cA_P\\
(A,u) &\mapsto A\cdot u = \Ad_{u^{-1}}(A) + u^*(\theta),
\end{align*}
where $\theta$ is the Maurer-Cartan form on $G$. If we identify $\frakg = T_eG$ we have have \[u^*(\theta)(X_p) = (l_{u_p^{-1}} \circ u)_* (X_p).\]Under this action the curvature transforms as
\[
 F^{A\cdot u} = \Ad_{u^{-1}}(F^A).
\]
In particular $\cF_P$ is invariant under the action of $\cG_P$.
\end{lem}

\begin{proof}
Certainly $\cG_P$ acts on $\Omega^1(P;\frakg)$ on the right because $\Phi^*\circ\Psi^* = (\Phi \circ \Psi)^*$.
Note that for a gauge transformation $\Phi$, $g\in G$ and $p\in P$ we have
\begin{gather*}
 \Phi \circ r_g (p) = \Phi (pg) = \Phi(p) g = r_g \circ \Phi(p)\\
\tag*{and}\Phi \circ i_p(g) = \Phi (pg) = \Phi(p) g= i_{\Phi(p)} g.
\end{gather*}
Therefore for the fundamental vector field $X^*$ corresponding to
$X\in \frakg$ we get
\begin{gather*}
 (\Phi^*A)(X^*_p) = A(\Phi_* (i_p)_*X_e) = A((i_{\Phi(p)})_* X_e) = A(X_{\Phi(p)}^*) = X\\
\tag*{and} r_g^*(\Phi^*(A)) = \Phi^*(r_g^*(A)) = \Phi^*(\Ad_{g^{-1}}A) = \Ad_{g^{-1}}(\Phi^*(A))
\end{gather*}
So $\Phi^*(A)\in \cA_P$ for $A\in \cA_P$. If $\Phi(p) = p\cdot u_p$ for a $G$--equivariant map $u\co P \to G$, then it is a simple consequence of the product formula that
\[
\Phi_* X_p = (r_{u_p})_*X_p + (i_p)_*u_*(X_p)
\]
and therefore we get
\begin{align}
 (A\cdot u)(X_p) & = A((r_{u_p})_*X_p)  + A((i_p)_*u_*(X_p))\nonumber\\
& = \Ad_{u_p^{-1}} A(X_p) + A((i_{pu_p})_*(l_{u_p^{-1}})_*u_*(X_p))\nonumber\\
& = \Ad_{u_p^{-1}} A(X_p) + (l_{u_p^{-1}}\circ u)_*(X_p))\label{T_eGidentified}\\
& = \Ad_{u_p^{-1}} A(X_p) + (u^*\theta) (X_p)\nonumber,
\end{align}
where we also used the identification $\frakg = T_eG$ in line \eqref{T_eGidentified}. Lastly recall from Equation \eqref{F^Avertical} that $F^A$ is zero if one of the tangent vectors is vertical, in particular on the image of $(i_p)_*$, therefore equivariance of $A$ immediately implies
\begin{align*}
 F^{A\cdot u}(X_p,Y_p) &= F^A(\Phi_*X_p,\Phi_*Y_p) = F^A((r_{u_p})_*X_p,(r_{u_p})_*Y_p)\\
&=dA((r_{u_p})_*X_p,(r_{u_p})_*Y_p) + \tfrac{1}{2} [A\wedge A]((r_{u_p})_*X_p,(r_{u_p})_*Y_p)\\
&=\Ad_{u_p^{-1}} A(X_p,Y_p) + \tfrac{1}{2} [(Ad_{u_p^{-1}}A)\wedge (Ad_{u_p^{-1}}A)](X_p,Y_p)\\
&=Ad_{u_p^{-1}}F^A(X_p,Y_p).\qedhere
\end{align*}
\end{proof}

We will often encounter the situation where $P$ is trivializable. If we fix a trivialization $P\cong M\times G$, then we can identify \[\cA_P \cong \Omega^1(M;\frakg), \quad \cG_P \cong C^\infty(M,G).\]
    The action of $\cG_P$ on $\cA_P$ can still be written as
\[
 (A\cdot u)(p)= \Ad_{u^{-1}_p} A + l_{u_p^{-1}}u_*.
\]

\section{Flat connections and the fundamental group}

Fix $p_0 \in \pi^{-1}(x_0)$. Then a flat connection $A$ gives us a group homomorphism
\[\hol_A\co \pi_1(M,x_0)\to G.\] In other words, we have a map $\cF_P \to \Hom(\pi_1(M,x_0),G)$. Let $\cM_P = \cF_P/\cG_P$ be the moduli space of flat connections in a principal $G$--bundle $P$. The following Proposition shows that for $M$ connected
\[
\hol \co \cM_P \to \Hom(\pi_1(M),G)/G
\]
is well-defined, particularly independent of $x_0$ or its lift, where $G$ acts on $\Hom(\pi_1(M),G)$ on the right via $(\rho \cdot g) (\gamma) = g^{-1}\rho(\gamma)g$. If $G$ is a subgroup of $\Aut(V)$ for some finite-dimensional vector space $V$ (or even a Hilbert space), then $\Hom(\pi_1(M),G)$ are representations of $\pi_1(M,x_0)$.

\begin{prop}\label{PropHolonomyTransformation} Let $A$ be a flat connection in $P$, $\alpha \in \pi_1(M,x_0)$ and $p_0\in \pi^{-1}(x_0)$.
\begin{enumerate}
 \item Let $p_0'$ be another point in the fiber over $x_0$. Then $p_0' = p_0\cdot g$ and
\[
\hol_{A} (\alpha,p_0 g) = g^{-1}\hol_A(\alpha,p_0)g.
\]
 \item Let $u\co P \to G$ be a gauge transformation. Then we have
\[
 \hol_{A\cdot u} (\alpha,p_0) = u_{p_0}^{-1} \hol_A(\alpha,p_0) u_{p_0}.
\]
\item Let $x_0'$ be a different base point, $p_0' \in \pi^{-1}(x'_0)$ and $\gamma\co I \to M$ a path from $x_0$ to $x_0'$, then there exists $g\in G$ such that
\[
 \hol_A (\gamma* \alpha * \gamma^{-1},p_0') = g \hol_A(\alpha,p_0) g^{-1}.
\]
\end{enumerate}
\end{prop}

\begin{proof}
Let $\beta$ be a horizontal lift of $\alpha$ with base point $p_0$. To show (1) note that for the path \[\beta'_t = \beta_t \cdot g\] for $g\in G$ we have
\[
 A(\dot\beta'_t) = A((r_g)_*\beta_t) = \Ad_{g^{-1}} A(\beta_t) = 0.
\]
Therefore $\beta'_t$ is the horizontal lift with base point $p_0\cdot g$. Then
\begin{align*}
\beta'(1) &= \beta(1) g \\
&= \beta(0) \hol_A(\alpha,p_0) g\\
&= \beta'(0) g^{-1} \hol_A(\alpha,p_0) g
\end{align*}
and the holonomy of $A$ along $\alpha$ with respect to the base point $p_0 h$ is given by
\begin{equation*}
\hol_{A} (\alpha,p_0 g) = g^{-1}\hol_A(\alpha,p_0)g.
\end{equation*}
To show (2) let $\beta'(t)= \beta(t) w_{\beta(t)}$ where $w = u^{-1}$. Then the product rule implies
\[
\dot \beta'_t = (i_{\beta(t)})_*  w_*\dot \beta_t + (r_{w(\beta_t)})_* \dot\beta_t.
\]
If $B$ is any connection then
\begin{align*}
B(\dot \beta'_t)& = B((i_{\beta(t)w_{\beta(t)}})_* (l_{w^{-1}_{\beta(t)}})_* w_*\dot \beta_t) +
B((r_{w(\beta_t)})_* \dot\beta_t)\\
& = (l_{w^{-1}_{\beta(t)}})_* w_*\dot \beta_t + \Ad_{w_{\beta(t)}^{-1}} B (\dot\beta_t)\\
& = w^*(\theta) (\dot \beta_t)+ \Ad_{w_{\beta(t)}^{-1}} B (\dot\beta_t)
\end{align*}
as well as
\begin{equation*}
 (B\cdot w)(\dot \beta_t)  = \Ad_{w_{\beta(t)}^{-1}} B (\dot\beta_t) + w^*(\theta)  (\dot\beta_t).
\end{equation*}
Therefore we get for $B = A \cdot u$
\[
 A(\dot\beta_t) = (B \cdot w)(\dot\beta_t) = B(\dot\beta'_t) = (A \cdot u)(\dot\beta'_t).
\]
Since $\beta$ is the horizontal lift of $\alpha$ with respect to $A$ starting at $p_0$, $\beta'$ is the horizontal lift of $\alpha$ with respect to $A\cdot u$ starting at $\beta'_0 = p_0 u^{-1}_{p_0}$. Then
\[
u_{\beta(1)} = u_{\beta(0)\cdot \hol_A(\alpha,p_0)} = C_{(\hol_A(\alpha,p_0))^{-1}} u_{p_0}
\]
implies
\begin{align*}
\beta'(1) & = \beta(1) \cdot u_{\beta(1)}^{-1} \\
& = \beta(0) \cdot \hol_A(\alpha,p_0)\cdot C_{(\hol_A(\alpha,p_0))^{-1}} u_{p_0}^{-1}\\
&= \beta'(0) \cdot \hol_A(\alpha,p_0).
\end{align*}
Therefore $\hol_{A\cdot u} (\alpha,p_0u^{-1}_{p_0}) = \hol_A(\alpha,p_0)$ and by (1)
\begin{align*}
\hol_{A\cdot u} (\alpha,p_0) &= u^{-1}_{p_0}\hol_{A\cdot u} (\alpha,p_0u^{-1}_{p_0})u_{p_0}\\
&= u_{p_0}^{-1}\hol_A(\alpha,p_0)u_{p_0}.
\end{align*}
The horizontal lift $\beta_{\gamma}$ of $\gamma$ starting at $\beta(0)=p_0$ ends at $\beta_{\gamma}(1)= p_0'g$ for some $g\in G$. We know from the proof of (1) that any other horizontal lift is of the form $\beta_{\gamma}\cdot g'$ for some $g'\in G$. In particular the horizontal lift $\beta_\gamma \cdot \hol_A(\alpha,p_0)$ starts at $\beta(1)$. This shows that $(\beta_{\gamma} \cdot \hol_A(\alpha,p_0)) * \beta * \beta_\gamma^{-1}$ is a horizontal lift of $\gamma * \alpha * \gamma^{-1}$ based at $p_0'$.
\begin{align*}
((\beta_{\gamma} \cdot \hol_A(\alpha,p_0)) * \beta * \beta_\gamma^{-1})(1) &=  \beta_{\gamma}(1)\cdot \hol_A(\alpha,p_0)\\
& = \beta_\gamma^{-1}(0) \cdot \hol_A(\alpha,p_0)\\
& = ((\beta_{\gamma} \cdot \hol_A(\alpha,p_0)) * \beta * \beta_\gamma^{-1})(0) \cdot \hol_A(\alpha,p_0)
\end{align*}
Therefore by (1)
\[
 \hol_A (\gamma* \alpha * \gamma^{-1},p_0') =  g \hol_A (\gamma* \alpha * \gamma^{-1},p_0'\cdot g) g^{-1} = g \hol_A(\alpha,p_0) g^{-1}. \qedhere
\]
\end{proof}

We have seen how to get a homomorphism $\rho\co \pi_1(M,x_0) \to G$ from a flat connection $A$ in a principal $G$--bundle $P$. To associate a principal $G$--bundle to a homomorphism $\rho\co \pi_1(M,x_0) \to G$ with $x_0\in M$, let $\wtilde M$ be the universal cover of $M$. Fix a  base point $p_0\in \wtilde M$ and identify $\wtilde M$ with the space of homotopy classes of paths in $M$ starting at $x_0$. Then $\pi_1(M,x_0)$ naturally acts on $\wtilde M$ from the right. For $\rho\co \pi_1(M,x_0) \to G$ we define the principal $G$--bundle
\[
 P_{\rho} := \wtilde M \times_{\rho} G = (\wtilde M \times G) / \sim
\]
where
\[
 (\tilde x,g) \sim (\tilde x \cdot \alpha,\rho(\alpha)^{-1}g), \quad (\tilde x, g) \in \wtilde M\times G, \alpha \in \pi_1(M).
\]
Pulling back the Maurer-Cartan form $\theta \in \Omega(G;\frakg)$ to $\wtilde M\times G$ defines a natural flat connection $\tilde A = \pi_G^*(\theta)$ on $\wtilde M \times G$ since
\begin{align*}
\tilde A(X^*_{(\tilde x,g)}) = \tilde A((i_{(\tilde x,g)})_*(X_e)) = \theta ((l_g)_*X_e) = X &\quad \text{for }X\in \frakg,\\
r^*_h\tilde A = \pi^*_Gr^*_h\theta = \pi^*_G\Ad_{h^{-1}} \theta = \Ad_{h^{-1}}\tilde A&\quad\text{by Exercise \ref{ExeMaurerCartan}}\\
\tag*{and} d\tilde A + \tfrac{1}{2}[\tilde A\wedge\tilde A] = \pi^*_G(d\theta + \tfrac{1}{2}[\theta,\theta]) = 0& \quad\text{by Proposition \ref{PropMaurerCartan}}.
\end{align*}
This also induces a flat connection $A_\rho$ on $P_\rho$ because
\[
\tilde A ((r_{\alpha})_*X_{\tilde x},(l_{\rho(\alpha)^{-1}})_*Y_{g}) = l^*_{\rho(\alpha)^{-1}} \theta(Y_g) = \theta(Y_g) = \tilde A (X_{\tilde x},Y_g).
\]

\begin{exe}\label{rightinverse}
Show that with respect to the base point $[x_0,e] \in P_\rho$ we have
\[
 \hol_{A_\rho} (\alpha,[x_0,e]) = \rho(\alpha).
\]
\end{exe}

\begin{defn}
 A flat principal $G$--bundle over $M$ is a pair $(P,A)$ consisting of a principal $G$--bundle $P$ over $M$ and a flat connection in $P$. We call $(P,A)$ and $(P',A')$ isomorphic, if there is a principal bundle isomorphism $\Phi\co P \to P'$ such that $A = \Phi^*(A')$. Let $\cM(M,G) $ be the moduli space of flat principal $G$--bundles over $M$.
\end{defn}

\begin{thm}\label{ThmFlatConnEqualRep}
 Let $(P,A)$ be a flat principal $G$--bundle and let $\rho\co \pi_1(M) \to G$ be a homomorphism from the fundamental group to $G$. Then $(P,A)$ is isomorphic to $(P_\rho,A_\rho)$ if and only if there exists $g \in G$ with $\hol_A = g^{-1} \rho g$. In particular, we have the bijection
\begin{align*}
 \cM(M,G) &\stackrel{\cong}{\to} \Hom(\pi_1(M),G)/G\\
[P,A] &\mapsto [\hol_A].
\end{align*}
\end{thm}

\kommentar{Man haette wirklich fiber bundle definieren sollen und die lifting property beweisen sollen, oder wenigstens zitieren...}
\begin{proof}
That $(P,A) \cong (P_\rho,A_\rho)$ implies $\hol_A = g^{-1}\rho g$ for some $g\in G$ follows immediately from a slight generalization of Proposition \ref{PropHolonomyTransformation}. So the holonomy map is well-defined on the quotients. On the other hand, if $\tilde \rho := g_0^{-1}\rho g_0$ for some $g_0\in G$, then
\begin{align*}
 \wtilde M \times G &\to \wtilde M \times G\\
(\tilde x, g) &\mapsto (\tilde x,g_0g).
\end{align*}
induces a bundle isomorphism $\Phi\co P_\rho \to P_{\tilde \rho}$
since
\[
 (\tilde x \cdot \alpha,\tilde\rho(\alpha)^{-1}g) \mapsto (\tilde x\cdot \alpha,g_0\rho(\alpha)^{-1} g) = (\tilde x\cdot \alpha,\rho(\alpha)g_0 g).
\]
Certainly, $\Phi$ is an isomorphism of principal $G$--bundles. Furthermore $\Phi^*(A_{\tilde\rho}) = A_\rho$ since
\[
\Phi^*(\pi^*_G(\theta)) (X_{\tilde x},Y_g) = \pi^*_G(\theta) (X_{\tilde x},(l_{g_0})_*Y_g) = \theta((l_{g_0})_*Y_g) = \theta(Y_g) = \pi^*_G(\theta)(X_{\tilde x},Y_g).
\]
Therefore, the proposed inverse of the holonomy map is well-defined, which by Exercise \ref{rightinverse} is a right-inverse. To show that it is a left-inverse, assume that $\hol_A = \rho$ and consider
\begin{align*}
\Phi\co \wtilde M\times G & \to P\\
(\tilde x,g) &\mapsto \beta_{\tilde x}(1) \cdot g
\end{align*}
where $\beta_{\tilde x}$ is the horizontal lift of the path in $M$ representing $\tilde x$ starting at some $p_0 \in \pi^{-1}(x_p) \subset P$. $\Phi$ is well-defined because $\beta_{\tilde x}$ only depends on the homotopy class of $\tilde x$ by similar arguments as in Theorem \ref{ThmHolPi1}. $\Phi$ is certainly $G$--equivariant and surjective. Furthermore if $\alpha\in \pi_1(M,x_0)$, then
\begin{gather*}
 \beta_{\tilde x \cdot \alpha} (1) = \beta_{\tilde x}(1) \cdot \hol_A(\alpha,p_0)\\
\tag*{and therefore} \Phi^{-1}(\beta_{\tilde x}(1)) = \left\{\left(\tilde x\cdot \alpha,\hol_A(\alpha,p_0)^{-1}\right)\mid \alpha \in \pi_1(M,x_0)\right\}.
\end{gather*}
This implies together with $\hol_A(\alpha,p_0) = \rho(\alpha)$ that
$\Phi$ descends to a bundle isomorphism $P_\rho \to P$. Furthermore
$\Phi^*(A) = A_\rho$.  Let $(X_{\tilde x},Y_g)\in H_{(\tilde x,g)}$,
where $H_p$ is the horizontal distribution given by $A_\rho$. Then
\[
0= A_\rho(X_{\tilde x},Y_g) = \theta(Y_g)
\]
implies $Y_g = 0$. We assume by choosing a suitable homotopic path that $\dot{\tilde x}(1) = X_{\tilde x}$ and let $\gamma(t) = (\tilde x(t),g)$. Then
\begin{align*}
 \Phi_*(X_{\tilde x},Y_g) (f) &= (X_{\tilde x},0)(f\circ \Phi)\\
&= \eval{\tfrac{d}{dt}}{t=1} (f\circ \Phi \circ \gamma) = \eval{\tfrac{d}{dt}}{t=1} (f(\beta_{\tilde x}(t)\cdot g))\\
&= (r_g)_*\dot\beta_{\tilde x}(1)(f)
\end{align*}
yields
\begin{align*}
A(\Phi_*(X_{\tilde x},Y_g)) = \Ad_{g^{-1}}A(\beta_{\tilde x} (1) = 0
\end{align*}
because $\beta_{\tilde x}$ is a horizontal lift with respect to $A$. This completes the proof that the holonomy map is a bijection as stated.
\end{proof}
\kommentar{F^A schon umdefiniert als 2-form in P?}

\section{Chern-Weil Theory}
\kommentar{Add Signature theorem}

Characteristic classes are mainly used in obstruction theory. For example the {\em Euler class $e(\pi) \in H^n(M;\Z)$} is the primary obstruction to trivializing a real vector bundle $\pi\co E \to M$ of rank $r$ or a $\GL(n,\R)$--principal bundle $\pi\co P \to M$. This suggests, that characteristic classes are more naturally situated in the framework of vector bundles, however characteristic classes for principal bundles seems a bit more general. Since so far we tried to not restrict ourselves to matrix Lie groups, we will continue to move within the framework of principal bundles.

Chern-Weil theory is a way of describing characteristic classes of vector bundles or principal bundles using differential geometry, instead of the topological method of pulling back universal cohomology classes. Chern classes and Pontrjagin classes are represented in DeRham theory by differential forms which are functions of the curvature of a connection in the bundle. There are two approaches to defining characteristic forms, one uses invariant polynomials, the other formal power series.

\subsection{Invariant polynomials}

Let $V$ be a complex vector space. For $k\ge 1$ let $S^k(V^*)$ be the vector space of linear maps
\[
f \co V\tensor \ldots \tensor V \to \C.
\]
If we let $S^0(V^*) = \C$, then
\[
S^*(V^*)\bigoplus_{k=0}^\infty S^k(V^*)
\]
is a commutative ring with unit $1\in S^0(V^*)$ and product
\[
 f\cdot g(v_1 ,\ldots ,v_{k+l}) = \frac{1}{(k+l)!} \sum_\sigma f(v_{\sigma_1}, \ldots , v_{\sigma_k}) g(v_{\sigma_{k+1}}, \ldots , v_{\sigma_{k+l}})
\]
for $f \in S^k(V^*)$ and $g \in S^l(V^*)$, where $\sigma$ runs over all all permutations of $1,\ldots,k+l$.

\begin{exe}\label{ExePolynomialFunction}
 A choice of basis $\{e_1,\ldots,e_n\}$ for a vector space $V$ determines an isomorphism of $S^k(V^*)$ with homogeneous polynomials of degree $k$, as well as a ring isomorphism $S^*(V^*)\cong\C[x_1,\ldots,x_n]$, where $p(x_1,\ldots,x_n) := f(v,\ldots,v)$ for $v = \sum x_i e_i$ and $f\in S^k(V^*)$.
\end{exe}

Consider $V = \frakg$. Then the adjoint representation of $G$ on the Lie algebra $\frakg$ induces an action of $G$ on $S^k(\frakg^*)$ for every $k$:
\[
 (g\cdot p) (X_1,\ldots,X_k) = p(\Ad(g^{-1})X_1,\ldots,\Ad(g^{-1})X_k), \quad\text{where } X_1,\ldots,X_k \in \frakg \text{ and } g\in G.
\]

\begin{defn}
Denote the set of $G$--invariant elements in $S^*(\frakg^*)$ by $I^*(G)$. Exercise \ref{ExePolynomialFunction} justifies us calling $f\in I^*(G)$ an {\em invariant polynomial}.
\end{defn}

\begin{exa}
 If $G = \GL(n,\C)$ and $\frakg = M^{n,n}(\C)$ (or $G= \Aut(V)$ and $\frakg = \End(V)$ and we use a basis to identify $\End(V)\cong M^{n,n}(\C)$), then it is easy to see that the invariant polynomials $I^k(\GL(n,\C))$ correspond to the $\GL(n,\C)$--invariant homogeneous polynomials of degree $k$.
Therefore, particularly trace and determinant correspond to invariant polynomials. If $A \in M^{n,n}(\C)$,
\[
\det(\Id + t A) = \sum_{k+0}^n t^k\sigma_k(A),
\]
where the homogeneous polynomials $\sigma_k$ of degree $k$ are called the elementary symmetric polynomials. It turns out that all $\GL(n,\C)$--invariant homogeneous polynomials are linear combinations of the elementary symmetric polynomials $\sigma_k$, which can be shown to be
\begin{align*}
 \sigma_0(A) &= 1\\
 \sigma_1(A) &= a_1+ \ldots + a_n = \tr A\\
 \sigma_2(A) & = \sum_{i<j} a_i a_j\\
 \vdots\\
 \sigma_n(A) &= a_1 \cdots a_n = \det A
\end{align*}
where $a_i$ are the the eigenvalues of $A$.
\end{exa}

Let $P\to M$ be a principal $G$--bundle with a connection $A$. Then $f\in I^*(G)$ defines $f(F_A^k) \in \Omega^{2k}(P)$, where \[
F_A^k = F_A \wedge \ldots \wedge F_A \in \Omega^{2k}(P;\frakg\tensor \ldots \tensor \frakg).
\] By Proposition \ref{PropCurvaturePullback} there exists a unique $2$--form in $\Omega^2(\Ad_P)$ which pulls back to $F_A$. Hence we may consider $f(F_A^k)$ as an element of $\Omega^{2k}(M)$, which is called the {\em characteristic form} corresponding to $f$. The curvature being closed immediately implies $df(F_A^k) = 0$, so $c_f(A) \coloneqq [f(F_A^k)]$ is a DeRham cohomology class called {\em characteristic class} corresponding to $p$. Similarly, if $p$ is a $\Aut(\frakg)$--invariant homogeneous polynomial of degree $k$ in $\End(\frakg)$, Exercise \ref{ExePolynomialFunction} allows us to define $p(F_A) \in \Omega^{2k}(M)$, where $F_A \in \Omega^2(M;\Ad_P) \subset \Omega^2(M;\End(\Ad_P))$. The determinant and the trace are of particular interest.

\kommentar{Ersetzt F^A durch F_A}

\begin{defn} Consider a principal $G$--bundle $P$ for $G \subset M^{n,n}(\C)$ (or equivalently a complex vector-bundle with structure group $G$). The characteristic form
\[
c(P,A)\coloneqq \det(\Id +\tfrac{i}{2\pi}F_A) \in \Omega^\ev(M)
\]
is called the {\em total Chern form}, where we consider matrix multiplication in $\frakg$. The normalization $\tfrac{i}{2\pi}$ is chosen so that the corresponding Chern classes are integer-valued for $G = \SU(N)$. The $k$--th Chern form $c_k(P,A) \in \Omega^{2k}(M)$ is defined as the component of degree $2k$ of the total Chern form, i.e.
\[
c(P,A) = \sum_k c_k(P,A) = \sum_{k} \left(\tfrac{i}{2\pi}\right)^k \sigma_k(F_A).
\]
\end{defn}

\kommentar{Dieser Abschnitt ist sehr duerftig und sollte noch aufgefuellt werden mit Details}

Other characteristic forms are the Pontrjagin form, the Hirzebruch $\hat L$--form and the $\hat A$--form, for which we also have nice formulas.

\subsection{Formal power series}

One can also use power series to define these characteristic forms. Even though we could do a similar construction for general principal bundles let us restrict ourselves to groups $G \subset M^{n,n}(\C)$. See Zhang \cite{zhang2001} for detailed information. If $f(z) = \sum_{n\ge 0} a_n z^n$ is a formal power series, then we can define the characteristic form
\[
\tr \{f(A)\} \coloneqq \tr\left\{\sum_{n\ge 0} a_n\left(\tfrac{i}{2\pi} F_A\right)^n \right\}\in \Omega^\ev(M)
\]
where $F_A \in \Omega^2(M,\Ad_P)$ and the $n$--th power corresponds to matrix multiplication. The normalization $\tfrac{i}{2\pi}$ is again chosen with the Chern-classes in mind.

\begin{defn}
 The characteristic form associated to $\exp(z)$ is called the Chern character form
\[
 \ch(P,A) \coloneqq \tr \left\{\exp\left(\tfrac{i}{2\pi} F_A\right)\right\}
\]
\end{defn}

\begin{exe}\label{ExeDetExpTrLog} For every $\alpha \in \Omega^{2k}(\End(\Ad_P))$, $k\ge 1$, we have
\[
 \det(1+\alpha) = \exp(\tr\{\log(1+\alpha)\}),
\]
where $\exp$ is considered as a formal power series and the logarithm is defined using the formal power series
\[
 \log(1+z) = \sum_{n\ge 0} \frac{(-1)^n}{n+1} z^{n+1}
\]
\end{exe}

Exercise \ref{ExeDetExpTrLog} implies that for any normalized formal power series $f(z) = 1 + \sum_{n\ge 1} a_n z^n$ we have
\[
 \det(f(\alpha)) = \exp(\tr\{\log f(\alpha)\}),
\]
so that the Chern character form is associated to a formal power series.

\section{The Chern-Simons form}

The famous Chern-Weil theorem states that the difference
\[
 \tr\{f(A_1)\} - \tr\{f(A_0)\}
\]
is an exact form. For a proof of the following result see Berline-Getzler-Vergne \cite{berline-getzler-vergne2004}.
\begin{thm}
 Let $P \to M$ be a principal $G$--bundle over a manifold $M$, and let $f(z) = \sum_{n\ge 0} a_n z^n$ be a formal power series. If $A_t$ is a smooth path of connections on $P$, then
\[
 \tfrac{d}{dt}\tr\{f(A_t)\} = d\tr \{\tfrac{i}{2\pi} (\tfrac{d}{dt} A_t) \wedge f'(A_t)\}.
\]
In particular, if $a = A_1-A_0 \in \Omega^1(M,\Ad_P)$ is the difference of two connections, then
\[
 \tr\{f(A_1)\}-\tr\{f(A_0)\} = d\int_0^1 \tfrac{i}{2\pi} \tr \{a\wedge f'(A_0 + ta)\}\, dt.
\]
Therefore, we have an equality of cohomology classes
\[
 \tr\{f(A_1)\} = \tr\{f(A_0)\} \in H^\ev(M).
\]
\end{thm}

Therefore, the cohomology class is independent of the connections and we call the characteristic class associated to $f$ the $f$-class of $P$. The Chern classes $c(P)$ are integer valued, i.e. $\int_M c(P) \in \Z$ for all principal bundles $P \to M$. See Milnor-Stasheff \cite[App. C]{milnor-stasheff1974}. The other classes are in general only $\Q$--valued.

\begin{defn}
 If $A_t$ is a path of connections, we call
\[
 T c_f(A_t) \coloneqq \int_0^1 \tfrac{i}{2\pi} \tr\left\{(\tfrac{d}{dt} A_t) \wedge f'(A_t)\right\} \, dt \in \Omega^\odd(M)
\]
the {\em transgression form of the $f$--class associated to $A_t$}. If $A_t = A_0 + ta$ we also use the notation
\[
 T c_f(A_0,A_1) \coloneqq \int_0^1 \tfrac{i}{2\pi} \tr \{a\wedge f'(A_0 + ta)\}\, dt \in \Omega^\odd(M).
\]
The transgression form of the Chern character is called the {\em Chern-Simons form} of $A_1$ with respect to $A_0$,
\[
\alpha(A_0,A_1) = \int_0^1 \tfrac{i}{2\pi} \tr [a\wedge \exp(A_0 + ta)]\, dt \in \Omega^\odd(M).
\]
\end{defn}

For us the Chern-Simons form for the degree 3 is most relevant due to its importance in 3--manifold topology. We want to derive an explicit formula for it. Set
\[
\alpha(A_0,A_1)
\]
We abbreviate $A \coloneqq A_0$ and let $F_t$ denote the curvature of $A_t = A + ta$. Then
\[
 F_t = F_A + t (da + [A,a]) + \tfrac{1}{2}t^2 [a\wedge a].
\]
For the component of degree 4 of the Chern character we have $f(z) =
\tfrac{z^2}{2}$ so that $f'(z) = z$. Then by definition the degree 3
component is
\begin{align*}
 \alpha_3(A_0,A_1) &= -\tfrac{1}{4\pi^2} \int_0^1 \tr\{a\wedge (F_A + t (d^A a) + \tfrac{1}{2}t^2[a\wedge a])\}\, dt\\
&= -\tfrac{1}{4\pi^2} \tr\{a\wedge (F_A + \tfrac{1}{2} (d^A a) + \tfrac{1}{6}[a\wedge a])\}.
\end{align*}
If $A$ is flat then
\[
 \alpha_3(A_0,A_1) = -\tfrac{1}{8\pi^2} \tr\{a\wedge ( d^A a + \tfrac{1}{3}[a\wedge a])\}.
\]
If furthermore $A$ happens to be trivial then
\[
 \alpha(a)\coloneqq \alpha_3(A_0,A_1) = -\tfrac{1}{8\pi^2} \tr\{a\wedge  d a + a\wedge \tfrac{1}{3}[a\wedge a]\}.
\]
This motivates defining for any connection $A$ and an $\Ad$--invariant symmetric bilinear form $\la\cdot,\cdot\ra$ on $\frakg$
\[
 \alpha(A) = \la A \wedge F_A \ra - \tfrac{1}{6} \la A\wedge[A\wedge A] \ra.
\]\kommentar{Check that this makes sense even if $A$ is not trivial, and figure out what is the issue with the transgressed form for $A$ non-trivial}In the case of a simply connected, connected Lie group over a $3$--manifold, $P$ turns out to be trivializable (see Lemma \ref{LemPtrivial}). Note that $\la \cdot,\cdot\ra = -\tfrac{1}{8\pi^2}\tr$ will provide a convenient normalization for $\SU(N)$ (see Exercise \ref{ExeSU2normalization} in the case $\SU(2)$). For a general Lie group, we can always take the Killing form. Alternatively, we can decompose the corresponding Lie algebra into the direct sum of an abelian Lie algebra with simple Lie algebras and consider the different symmetric bilinear forms on the summands: every $\Ad$--invariant symmetric bilinear form on a simple Lie algebra is proportional to the Killing form, every symmetric bilinear form on an abelian Lie algebra is $\Ad$--invariant.
\begin{exe}\label{ExeCSform}
$\alpha(A)$ satisfies:
\begin{enumerate}
 \item $i^*_x\alpha(A) = -\frac{1}{6}\la\theta_x \wedge[\theta_x\wedge \theta_x]\ra$;
 \item $d \alpha(A) = \la F_A  \wedge F_A \ra$;\label{CWform}
 \item $r^*_g \alpha(A) = \alpha(A)$;
 \item If $\Phi\co P'\to P$ is a bundle map and $A$ a connection on $P$, then
\[
 \alpha(\Phi^*(A)) = \Phi^*(\alpha(A));
\]
\item If $\Phi\co P \to P$ is a gauge transformation with associated map $u\co P \to G$, then
\begin{equation}\label{EqCStransform}
\Phi^*(\alpha(A)) = \alpha(A) + d\la \Ad_{u^{-1}} A\wedge u^*\theta\ra - \tfrac{1}{6}\la u^*\theta \wedge[u^*\theta \wedge u^*\theta] \ra.
\end{equation}
\end{enumerate}
\end{exe}

\chapter{Chern-Simons theory}

In this section we restrict ourselves to compact oriented 3--manifolds and simply connected, connected, compact Lie groups. Before getting to the classical field theory, we will accustom ourselves to the setting by discussing 3--manifolds and providing details to a beautiful gauge theoretic construction by Taubes \cite{taubes90} using Chern-Simons gauge theory.

\section{Principal bundles on 3--manifolds}

The Chern-Simons form is most relevant in degree 3. We can integrate it over $3$--manifolds to give a well-behaved function on the space of connections.

\begin{lem}\label{LemPtrivial}
 Let $G$ be a simply connected Lie group. Every principal $G$--bundle $P$ over a manifold $M$ with $\dim M \leq 3$ is trivializable.
\end{lem}

\begin{proof}
 For any topological group $G$ there is a contractible space $EG$ on which $G$ acts freely. The projection $EG \to BG$ is a principal $G$--bundle, called the {\em universal bundle}. $BG$ is called the {\em classifying space for $G$}. For any paracompact\footnote{A topological space is paracompact if every open cover admits an open locally finite refinement.} space $B$ the pull-back of a map $B\to BG$ induces a bijection between the set $[B,BG]$ of homotopy classes of maps from $B$ to $BG$ and isomorphism classes of principal $G$--bundles over $B$. See \cite{husemoller94} for details.

Therefore we need to show that $[M,BG]$ is trivial for $\dim M \leq 3$. Associated to the fibration \[G \to EG \to BG\] there is a long exact sequence of homotopy groups (see \cite[Chapter 6]{davis-kirk2001}) which yields
\[
 0 = \pi_n EG \to \pi_n BG \stackrel{\cong}{\to} \pi_{n-1}G \to \pi_{n-1} EG = 0 \quad \text{for }n\geq 1.
\]
Since $\pi_2(G)=0$ (see for example \cite{browder61}), $G$ is 2--connected and thus $\pi_i BG = 0$ for $1\leq i \leq 3$. Since $EG$ is contractible, $BG$ is connected. Thus $BG$ is 3--connected. By working cell-by-cell one dimension at a time, one can see that any map from an $n$--dimensional $CW$--complex to an $n$--connected space is null-homotopic. This completes the proof.
\end{proof}

\section{The Chern-Simons function}

Let $G$ be a connected, simply connected, compact Lie group and $P$ a principal $G$-bundle over a closed, oriented $3$--manifold $M$. A trivialization of $P$ corresponds to a section $s\co M \to P$.

\begin{defn}\label{DefnCSaction} Let $M$ be a closed oriented $3$-manifold. The
  {\em Chern-Simons action} is given by
\[
\cs_s (A)\coloneqq \int_M s^* \alpha(A).
\]
\end{defn}

If $\Phi\co P \to P$ is a gauge transformation with associated map $u\co P \to G$, let $\theta_u \coloneqq (u\circ s)^*\theta$. Then by \eqref{EqCStransform}
\[
 \cs_{\Phi \circ s} (A) = \cs_s(\Phi^*(A)) = \cs_s(A) - \int_M \tfrac{1}{6} \la \theta_u \wedge [\theta_u\wedge \theta_u]\ra.
\]

\begin{exe}\label{ExeSU2normalization}
We have for $G = \SU(2)$ and $\la\cdot ,\cdot\ra = -\frac{1}{8\pi^2}\tr$
\[-\int_G \tfrac{1}{6} \la \theta \wedge [\theta\wedge \theta]\ra = 1.\]
\end{exe}

We will from now on assume that $-\tfrac{1}{6} \la \theta \wedge
[\theta\wedge \theta]\ra$ represents an integral class in
$H^3(G;\R)$. If we use a section $s\co M \to P$ to identify
$G$--connections with $\Omega^1(M;\frakg)$  and the gauge
transformations with $C^\infty(M,G)$, then the Chern-Simons action
takes the familiar form
\[
\cs_s(A) =  \int_M \la dA \wedge A +
\tfrac{1}{3} A \wedge [A \wedge A] \ra \quad \text{for } A\in \Omega^1(M;\frakg).
\]
and we have \[\cs_s(A) = \cs_s(A\cdot g) \mod \Z \quad \text{for } A\in \Omega^1(M;\frakg), g\in C^\infty(M,G).\] This yields a (smooth) Chern-Simons action
\[
 \cs_s\co \cA_P/\cG_P \to \R/\Z.
\]
By Exercise \ref{ExeCSform}(\ref{CWform}), the Chern-Simons action of a connection $A$ which extends as $\tilde A$ over a 4--manifold $W$ can be computed by integrating the Chern-Weil form
\[
 \cs_s(A) = \int_W \la F_{\tilde A}\wedge F_{\tilde A} \ra \mod \Z.
\]
\kommentar{Erweiterung von G zu G'? Siehe Seite  16 in Freed.}

We will also encounter compact oriented 3--manifolds $X$ with nonempty boundary. We can still define the Chern-Simons action as in Definition \ref{DefnCSaction}, but the integral of \eqref{EqCStransform} does not vanish and we get for $A \in \Omega^1(X;\frakg)$ and $g\in C^\infty(X,G)$
\begin{equation}\label{EqCSGauge}
 \cs_s(A\cdot g) = \cs_s(A) + \int_{\partial X} \la Ad_{g^{-1}} A\wedge g^* \theta\ra - \int_X\tfrac{1}{6} \la g^*\theta\wedge[g^*\theta\wedge g^*\theta]\ra.
\end{equation}

\begin{exe}\label{ExeWZW} The functional
 \[W_{\partial X}(g) \coloneqq - \int_X\tfrac{1}{6} \la g^*\theta\wedge[g^*\theta\wedge g^*\theta]\ra \]
depends only on the restriction of $g$ to $\partial X$.
\end{exe}

The {\em Wess-Zumino-Witten functional} $W_{\partial X}$ is the action of a 1+1 dimensional field theory (see \cite[Appendix A]{freed95}).

\section{An Euler characteristic of the gauge equivalence classes of \texorpdfstring{$\SU(2)$}{SU(2)}--connections}

In 1988 Taubes \cite{taubes90} defined an invariant for homology 3-spheres $M$ by defining an euler-characterisitic on the space of gauge equivalence classes of $\SU(2)$--connections. Then he proved that his invariant is actually the same as Casson's invariant for Homology 3-spheres. This is a little survey particularly of the gauge theoretic view on Casson's invariant for homology 3--spheres, which relates it to Chern-Simons theory and leads to a refinement of the Casson invariant, Floer's instanton homology.

\subsection{The Casson invariant}

Let $M$ be a closed $3$--manifold with $H^i(M)=H^i(S^3)$. Consider a Heegaard decomposition of $M$, i.e. let $X_k$, $k=1,2$ be handlebodies with boundary $\Sigma$ such that $X_1$ and $X_2$ glued along their boundary is $3$:
\[
M = X_1 \cup_\Sigma X_2.
\]
We will write $R(M)$ for the $\SU(2)$--representation variety $\Hom(\pi_1(M),\SU(2))/\SU(2)$. $R(M)$, $R(X_k)$ and $R(\Sigma)$ are (compact) real algebraic varities.

A $U(n)$--representation $\rho$ is called {\em irreducible}, if its
{\em stabilizer} $S_\rho = \{g \mid g \rho g^{-1}= \rho \}$
coincides with the center $C_G = \{ h \mid gh=hg \text{ for all } g
\in G\}$ of $G$, in our case with $C_{\SU(2)} = \{\pm
1\}$.\footnote{By Schur's Lemma, this agrees with the usual notion
of irreducibility when $G = U(n)$ acts on its Lie algebras
$\fraku(n)$ by the adjoint representation.} We will denote by
$R^*(M)$ the $\SU(2)$--representation variety of irreducibles.

\begin{exe}\label{ExeDimRep}
 $R^*(X_k)$, $k=1,2$, is a smooth open manifold of dimension $3g-3$ and $R(\Sigma)$ is a smooth open manifold of dimension $6g-6$.
\end{exe}

The inclusions $i_k \co \Sigma \to X_k$ and $j_k \co X_k \to M$ induce the commutative diagram
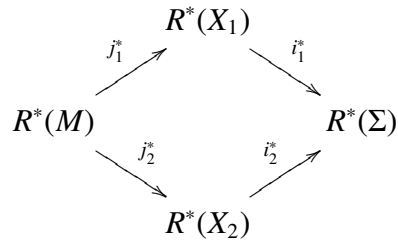
\[
\begin{diagram}
 \node{} \node{R^*(X_1)}\arrow{se,t}{i^*_1}\\
 \node{R^*(M)} \arrow{ne,t}{j^*_1}\arrow{se,t}{j^*_2}\node{} \node{R^*(\Sigma)}\\
 \node{} \node{R^*(X_2)}\arrow{ne,t}{i^*_2}
\end{diagram}
\]
with all its maps being injective. Therefore we can view $R^*(M)$ as the intersection of $R^*(X_1)$ and $R^*(X_2)$ inside $R^*(\Sigma)$.

Since every reducible $\SU(2)$--representation factors through a copy of $U(1) \subset \SU(2)$, it factors through the abelianization $H_1(M,\Z)$ of the fundamental group, which is trivial in our case. Therefore there is only one reducible representation, namely the trivial representation $\theta$. By a Mayer-Vietoris argument and the identification of the Zariski tangent spaces at $\rho$ with the group cohomology $H^1(\pi_1(M),\fraksu(2)_\rho)$ (see \cite{goldman84}) and therefore with the cohomology $H^1(M,\fraksu(2)_\rho)$ of the Eilenberg-Maclane space $M = K(\pi_1(M),1)$ twisted by $\rho$ we can show that the intersection $R(X_1) \cap R(X_2)$ is transversal at $\theta$ so that $R^*(X_1) \cap R^*(X_2)$ is compact (see \cite[Lemma 3.6]{saveliev2002}). Then we can choose an isotopy of $R^*(\Sigma)$ with compact support that carries $R^*(X_2)$ to $\tilde R^*(X_2)$ so that $\tilde R^*(X_2)$ is transversal to $\R^*(X_1)$. By Exercise \ref{ExeDimRep} $R^*(X_1) \cap \tilde R^*(X_2)$ is a finite number of points. There is a natural way to orient the representation varieties, so that the algebraic intersection number of $R^*(X_1)$ with $\tilde R^*(X_2)$ depends only on the orientation of $M$, not on the Heegaard decomposition or the choices of orientations. This algebraic count is called the {\em Casson invariant} of $M$.

The Casson invariant can be extended to more general classes of manifolds and knots, has lots of nice properties and is related to other topological invariants of $3$--manifolds and knots, particularly the ones stemming from Chern-Simons theory. See \cite{akbulut-mccarthy90} or \cite{saveliev99} for a gentle introduction to the Casson invariant.

\subsection{The Euler characteristic}

Let us for a moment consider a finite dimensional Riemannian manifold
$X$ and a Morse function $f$ on $X$. Then by the
Poincar\'e-Hopf-Theorem
\[
\chi(X)=\sum_{p\in X: \grad(f)|_p=0} (-1)^{\Hess f|_p}.
\]

\begin{exe} At
critical points $p$ we have $\la\hs f|_p (X_p),Y_p\ra = \Hess f|_p(X_p,Y_p)$, where $\la\cdot,\cdot\ra$ denotes the Riemannian metriic on $M$ and $\nabla$ is the Levi-Cevita connection.
\end{exe}
Since $\hs f|_p$ is symmetric for all points $p\in X$ we can rewrite
\[
\chi(X)=(-1)^{\ind(\Hess f|_{p^*})}\sum_{p\in X: \grad(f)|_p=0} (-1)^{-\SF(\hs f (p_t))}
\]
where the spectral flow $\SF$ counts how many times (with sign) the paths of eigenvalues of $\hs f (p_t)$ for an arbitrary path $p_t$ from a fixed point $p^*$ to $p$ cross $0$.

Taubes' idea was to apply the above thoughts to an infinite-dimensional manifold and a well-studied function on it, such that $\SF$ makes sense.
\begin{defn}
 A connection $A \in \cA$ is called irreducible, if $d_A$ is injective. Denote by $\cA^*$ the space of irreducible $\SU(2)$--connections and by $\cB^* = \cA^*/\cG$ the space of gauge equivalence classes of irreducible $\SU(2)$--connections.
\end{defn}

We will see in the next sections $\cB^*$ is a smooth connected infinite-dimensional manifold and we can consider a perturbation of the Chern-Simons action as a Morse function. In the case of a homology 3--sphere we can define an Euler characteristic on $\cB^*$ in the spirit of the Poincar\'e-Hopf-Theorem, where the fixed point $p^*$ corresponds to the unique reducible flat connection, namely the trivial connection.

\subsection{Fr\'echet manifolds}

Roughly speaking, infinite-dimensional manifolds are manifolds modelled on an infinite-dimensional locally convex vector space.

\begin{defn}
A {\em seminorm} on a vector space $V$ is a positive function $p\co V \to \R$ satisfying
\begin{align*}
 p(av) = |a| p(v) \quad \text{ and} \quad p(v+w) \le p(v) + p(w).
\end{align*}
A {\em locally convex space} is a vector space $V$, equipped with a countable family of seminorms $p_j$ satisfying
\begin{equation}\label{EqSeparated}
v \neq 0 \Longrightarrow p_j(v) \neq 0, \text{ for some } j.
\end{equation}
\end{defn}

We have a distance function on a locally convex vector space defined by
\[
d(u,v) = \sum_{j=0}^\infty 2^{-j} \frac{p_j(u-v)}{1+p_j(u-v)}
\]

\begin{exe}
 Show that $d$ is a metric on $V$.
\end{exe}

Note that Property \eqref{EqSeparated} is equivalent to $V$ being Hausdorff. Sometimes this is not included in the definition of a locally convex space.

\begin{defn} A {\em Fr\'echet space} is a complete locally convex space.
\end{defn}

The prototype of a Fr\'echet space is $C^\infty([0,1])$ with the seminorms
\[
p_k(f) \coloneqq \sup_{x\in [0,1]} \left\{|f^{(k)}(x)|\right\}.
\]
In this space a sequence $(f_n)$ of functions converges towards the element $f$ of $\C^\infty([0, 1])$ if and only if for every integer $k \ge 0$, the sequence $(f_n^{(k)})_n$ converges uniformly towards $f^{(k)}$. This can be extended to $\C^\infty(\R)$ by defining
\[
p_{k,l}(f) \coloneqq \sup_{x\in [-l,l]} \left\{|f^{(k)}(x)|\right\}.
\]
Certainly this can easily be generalized to $\C^\infty(M)$ for a manifold $M$.

\begin{defn}
 Let $V,W$ be Fr\'echet spaces and let $U$ be an open subset of $V$. A map $f\co U \subset V \to W$ is {\em differentiable} at a point $u \in U$ in a direction $v \in V$ if the limit
\[
d_u f(v) = \lim_{t \to 0} \frac{f(u+tv) -f(u)}{t}
\]
exists. The function is continuously differentiable on $U$, if the limit exists for all $u \in U$ and all $v\in V$, and if the function $df \co U \times V \to W$ is continuous as a function on $U \times V$. In the same way we may get higher derivatives
\[
d^{k} f\co U \times V^k \to W
\]
A function $f \co U \to W$ is called smooth, if all its derivatives exist and are continuous.
\end{defn}

\begin{defn}
 A {\em Fr\'echet manifold} is a Hausdorff space with a coordinate atlas taking values in a Fr\'echet space such that all transition functions are smooth maps.
\end{defn}

\begin{exe}
 Show that $\cA$, $\cG$ and $\cB^*$ are connected Fr\'echet manifolds.
\end{exe}

On a Fr\'echet manifold we can define tangent vectors and the tangent spaces in the usual way. The tangent bundle can be given the structure of a Fr\'echet manifold. A smooth vector field is a smooth section $M \to TM$. We also have directional derivatives of a function, the Lie bracket of two vector fields as well as the differential of a map between two Fr\'echet manifolds.

It is often more convenient to consider Banach and Hilbert manifolds.

\subsection{The differential topology of $\cA$, $\cG$ and $\cs$}\label{SecB*}

Since $\cA$ is an affine space modelled on $\Omega^1(M;\fraksu(2))$ the tangent space of $\cA$ at a connection $A$ can be indentified with $\Omega^1(M;\fraksu(2))$. A smooth path $g_t$ in $\cG \cong C^\infty(M,\SU(2))$ can be viewed as a smooth map from $M$ to $C^\infty(I,\SU(2))$. Therefore the differential of $g_t$ at $t=0$ is a map from $M$ to $T_{g_0} \SU(2)$. In particular this shows, that the tangent space at the identity map 1 is
\[
T_{1}\cG = C^\infty(M,\fraksu(2)) = \Omega^0(M;\fraksu(2)).
\]
Fix a connection $A$ and consider the action map
\begin{align*}
f:\cG&\rightarrow\cA\\
g&\mapsto A \cdot g = \Ad_{g^{-1}} A + g^{-1}dg.
\end{align*}

\begin{exe} Tor a path $g_t$ with $g_0 = 1$ and $\dot g_0 = a \in \Omega^0(M,\fraksu(2))$ we have
 \begin{enumerate}
  \item  $\eval{\tfrac{d}{dt}}{t=0} g_t^{-1} = -a$ and
  \item $\eval{\tfrac{d}{dt}}{t=0} \tfrac{g_t^{-1}dg_t}{t} = da$.
 \end{enumerate}
\end{exe}

Therefore we get
\[
df_1 (a) = \ad(-a)A + da = -[a,A] + da = da + [A,a] = d_A a
\]
By considering the identifications in the following commutative diagram
\begin{equation}
\begin{diagram}
  \node{T_1\cG}\arrow{s,r,<>}{\cong}\arrow{e,t}{df_1}
    \node{T_A\cA^*}\arrow{s,r}{\cong}\\
  \node{\Omega^0(M;\fraksu(2))}\arrow{n}\arrow{e,t}{d_A}\node{\Omega^1(M;\fraksu(2))}\arrow{n}
\end{diagram}\label{identifications}
\end{equation}
we find, that at an irreducible connection $A$ the derivative $df_1$
is injective and thus
\[
T_{[A]}\cB^* \cong T_A\cA^*/{\im(df_1)} \cong \coker d_A \cong \ker d_A^*,
\]
which allows us to work equivariantly on $\bigcup_{A\in\cA^*}\ker d_A^*
\subset T_*\cA$ rather than on $T_*\cB^*$. Choosing a connection $A \in \cA$ with $d_A^*=0$ also known as Lorentz gauge fixing.

We can compute for
\[
\cs(A) =  \int_M \la dA \wedge A +
\tfrac{2}{3} A \wedge A \wedge A \ra \quad \text{for } A\in \Omega^1(M;\fraksu(2)).
\]
the differential at $A$ as
\begin{align*}
d\cs_A(a) &= \left.
\frac{d}{d s}\right|_{s=0} cs(A+sa)\\
& = \int_M \la dA \wedge a + da \wedge A +
2 A \wedge A \wedge a \ra\\
& = 2 \int_M \la dA \wedge a +
A \wedge A \wedge a \ra \quad \quad \text{by Stokes' Theorem}\\
& = 2 \int_M \la F_A \wedge a \ra.
\end{align*}
If we use the $L^2$ inner product on $\Omega^1(M;\fraksu(2))$ given by
\[
 \la A,B \ra_{L^2} = 2 \int_M \la A \wedge * B \ra
\]
to identify vectors with covectors, we can define $\grad cs|_A$ via $\langle \grad cs|_A , a
\rangle_{L^2} = d cs_A(a)$, we get
\[
\grad cs|_A = * F_A \co \Omega^1(M;\fraksu(2)) \longrightarrow \Omega^1(M;\fraksu(2)).
\]
Using the Bianchi-Identity from Exercise \ref{ExeBianchi} we see that $d_A^* (\grad cs|_A) = 0$. Therefore $\grad cs_A$ descends to $T_{[A]}\cB^*$ as
\[
\grad \cs|_A = * F_A \co  \ker d_A^* \longrightarrow \ker
d_A^* \cong T_{[A]}\cB^*.
\]

\begin{exe} Similarily we get the linearization of $\grad \cs$
\[
H_A\coloneqq *d_A \co \Omega^1(M;\fraksu(2)) \longrightarrow \Omega^1(M;\fraksu(2))
\]
by the rule $\langle H_A(a),b \rangle_{L^2} = \Hess \cs_A(a,b) =
\left.\frac{\partial}{\partial s\partial
t}\right|_{s,t=0}cs(A+sa+tb)$.
\end{exe}
Notice that the $\im(H_A) \subset \ker d_A^*$ if and only if $A$ is flat. This motivates setting
\[
\tilde H_A\coloneqq  {\rm proj}_{\ker{d_A^*}}*d_A\co \ker
d_A^* \longrightarrow \ker d_A^*.
\]
(If the Levi-Civita connection $\nabla$ on $\cB^*$ exists, one can even see that $\hs cs|_{A} = \tilde H_A$, i.e. $\hs cs|_{A}$ satisfies the defining properties of the Levi-Civita connection. However, since $\cB^*$ is only a Fr\'echet manifold and not a Hilbert manifold with respect to the $L^2$--inner product, it is not clear if the Levi-Civita connection exists on $\cB^*$.\kommentar{Another way to say this is that $\tilde H_A$ is defined by $langle \tilde H_A(a),b \rangle = \Hess \cs_A(a,b)$})

\kommentar{fiber bundles haette ich einmal kurz am Anfang machen sollen}

\subsection{Differential operators on manifolds}

This is a brief overview of the necessary terminology and results. See \cite{higson-roe2009} for a detailed account. For further information see \cite{taylor96}, \cite{schechter2002}, \cite{shubin2001}, \cite{egorov-shubin94}.

A differential operator of order $k$ on $U\subset \R^r$ is a linear combination
\[
 D = \sum_{|\alpha|<k} a_\alpha D^\alpha \quad \text{for some } a_\alpha \in C^\infty(U),
\]
where $D^\alpha = \frac{\partial^{|\alpha|}}{\partial x^\alpha}$ and $\alpha$ is a multi-index. More generally we can define a differential operator from $U \subset \R^r \to \R^s$ if we let $a_\alpha \in C^\infty(U,\Hom(\R^r,\R^s))$. We write $D \in \Diff^k(U;\R^r,\R^s)$ or $D \in \Diff(U;\R^r,\R^s)$.

Let $E,F$ be vector bundles over a smooth manifold $M$. A differential operator $D \in \Diff^k(E,F)$ of order $k$ between sections of $E$ and $F$ is a linear map
\[
 D \co \Gamma(E) \to \Gamma(F)
\]
such that
\begin{enumerate}
 \item $D$ is local, i.e. $\supp(D s) \subset \supp(s)$ for $s\in \Gamma(E)$.
 \item For $U \subset M$  open, any bundle charts $\Phi \co E|_U \to U \times \R^r$, $\Phi \co E|_U \to U \times \R^s$ induce a differential operator in $\Diff^k(U;\R^r,\R^s)$.
\end{enumerate}
If $M$ is closed, $E$ comes equipped with an inner product and $D \in \Diff(E,E)$, then $D$ is {\em formally self-adjoint} if $\la Df, g\ra_{L^2} = \la f,Dg \ra_{L^2}$ for $f,g\in \Gamma(E)$ in the induced $L^2$--metric on $\Gamma(E)$.

The {\em principal symbol} $\sigma_D^k \in \Gamma (T^*M, \pi^*\Hom(E,F))$ of $D \in \Diff^k(E,F)$ is defined to be
\[
 \sigma_D^k(p,\xi) e \coloneqq \frac{1}{k!} D(f^ks) (p),
\]
where $p\in M$, $\xi \in T^*_p M$, $e\in E_p$, $f\in C^\infty_0(M)$ and $s \in C^\infty_0(E)$ satisfy $f(p) = 0$, $df|_p = \xi$, $s(p) =e$. $D$ is called {\em elliptic}, if $\sigma_D^k(p,\xi)$ is an isomorphism for all $(p,\xi) \in T^*M$ with $\xi \neq 0$. For example the symbol of the Laplacian $\Delta = -\sum \tfrac{\partial^2}{\partial x_k^2}$ on the Euclidean space is elliptic with principal symbol $-|\xi|^2$. Similarly, the Hodge Laplacian $dd^* + d^*d$ as well its twisted versions are elliptic.

A linear operator $T\co B_1 \to B_2$ between Banach spaces is {\em bounded} if
\[
||T|| := \sup_{x\neq 1} \frac{\|Tx\|_{B_2}}{\|x\|_{B_1}} < \infty.
\]
It is sometimes convenient to allow more general operators. An {\em unbounded} operator $T\co B_1 \to B_2$ between Banach spaces $B_1$ and $B_2$ is a linear operator on a linear subspace of $B_1$ (the {\em domain} $\cD(T)$) to $B_2$. Two operators are equal, if their domains are equal and they coincide on the common domain. The graph $G_T \coloneqq \{ (x,Tx) \in B_1\oplus B_2 \mid x\in \cD(T)\}$ is a linear subspace of $B_1\oplus B_2$. If $G_T$ is closed in $B_1 \oplus B_2$, $T$ is called {\em closed}. By the {\em closed-graph theorem}, $T$ is bounded, if $T$ is closed and $\cD(T) = B_1$. $T$ is called {\em Fredholm}, if $\cD(T)$ is dense in $B_1$, $T$ is closed, $\im T$ is closed in $B_2$ and $\ker T$ as well as $\coker T$ are finite-dimensional. The index of a Fredholm operator $T$ is defined to be $\ind T \coloneqq \dim (\ker T) - \dim (\coker T)$. A bounded operator between Banach spaces is {\em compact} if the image of bounded sets are {\em relatively compact}, i.e. the closure of the image is compact. Compact operators are a generalization of finite rank operators in an infinite-dimensional setting.

\begin{thm}\label{ThmFred}
Bounded Fredholm operators are invertible modulo compact operators. If $T$ is Fredholm and $K$ is compact, then $\ind(T) = \ind (T+K)$.
\end{thm}

We will mainly consider Hilbert spaces $H_1$ and $H_2$. If $T \co H_1 \to H_2$ is densely defined, then the {\em adjoint $T^*$ of $T$} is the operator $T^*\co H_2 \to H_1$ with \[\cD(T^*) = \{ y \mid x\mapsto \la  T x, y\ra_{H_2} \text{ is a continuous on } \cD(T)\}\] such that $\la Tx , y\ra_{H_2} = \la x , T^*y \ra_{H_1}$ for all $x \in \cD(T)$. $T$ is {\em self-adjoint}, if $T=T^*$. If $T\co H \to H$ is densely defined, then $T$ is {\em symmetric} if $\la Tx, y\ra_{H} = \la x,Ty \ra_H$. A symmetric operator $T$ is {\em essentially self-adjoint} if has a self-adjoint extension.

\begin{exe}\label{ExeOperators}
 Show the following:
\begin{enumerate}
 \item A densely defined operator $T$ with $\im T$ closed satisfies $\ker T^* \cong \coker T$.
 \item A bounded operator is closed. (Note that by the closed graph theorem, a closed operator is
 \item A bounded operator is Fredholm, if $\im T$ is closed and $\ker T$ as well as $\coker T$ are finite-dimensional.
 \item \label{ExeUnBounded}Let $T$ be an unbounded Fredholm operator. Then $\cD(T)$ equipped with the graph norm \[
 \|x\|_{\gr} \coloneqq \|x \| + \| Tx \|.
\] is a Banach space and $T$ is bounded on $\cD(T)$.
 \item A self-adjoint Fredholm operator has index 0.
 \item If $B\co B_1 \to B_2 $ is bounded and $K \co B_2 \to B_3$ is compact, then $K \circ B$ is compact.
\end{enumerate}
\end{exe}

The {\em resolvent set} of a densely defined operator $T \co H \to H$ on a Hilbert space $H$ is the set of all complex numbers $\lambda$ for which $T_\C-\lambda I$ is injective on $\cD(T)$ with dense range such that $(T-\lambda I)^{-1}$ can be extended to a bounded operator, where $T_\C\co H_\C \to H_\C$ is the complexified operator given by $T_\C(x \tensor \lambda) = T(x) \tensor \lambda$ and $H_\C = H \tensor_\R \C$. The operators $(T_\C-\lambda)^{-1}$ are the {\em resolvents} of $T$. The complement (in $\C$) of the resolvent set is called the {\em spectrum} $\spec(T)$ of $T$. If $(T-\lambda)^{-1}$ is compact for some $\lambda \in \C$, we say that $T$ has {\em compact resolvent}.

Let us collect a few relevant results from \cite{higson-roe2009}, particularly Theorem 2.16, Proposition 2.25 and Theorem 3.18.

\begin{thm}\label{ThmDiffOp} Let $E$ and $F$ be vector bundles over a closed manifold $M$ equipped with a metric, then we can extend $\Gamma(E)$ and $\Gamma(F)$ to a Hilbert space using the induced $L^2$--inner product. A formally self-adjoint first-order elliptic operator $T \in \Diff(E,F)$ can be extended to an unbounded self-adjoint operator $T \co L^2(\Gamma(E)) \to L^2(\Gamma(F))$ with compact resolvent (which is in particular also Fredholm). Furthermore, there is an orthonormal basis for $H$ of eigenvectors $\{u_j\}$ of $T$ belonging to $\cD(T)$ with eigenvalues $\{\lambda_j\}$ such that
\[
lim_{j\to \infty} | \lambda_j | = \infty \quad \text{and}\quad  \{\lambda_j\} = \spec(T).
\]
\end{thm}

Even though in these notes we only need to consider first-order operators on closed manifolds, the above theorem is true for higher order operators and compact manifolds with boundary (see \cite{bruening-lesch2001}). An elliptic differential operator of order $k$ has a bounded Fredholm extension $L^2_s(\Gamma(E)) \to L^2_{s-k}(\Gamma(F))$ (compare with Exercise \ref{ExeOperators}(\ref{ExeUnBounded})), where the $L^2_k$--Hilbert-space completion is taken with respect to a generalization of the Sobolev-norm on $C^\infty(\R^n)$ given by
\[
 \| f\|_k \coloneqq \sum_{|\alpha|\le k} \int_{\R^n} \left(| D^{\alpha} f |^2\right)^{\frac{1}{2}}.
\]
to $\Gamma(E)$. By the {\em Sobolev embedding theorem} the inclusion $L^2_{k} \hookrightarrow L^2_{k-1}$ is compact. Since the composition of a compact operator with a bounded operator is again compact by Exercise \ref{ExeOperators}, we get the following useful result.
\begin{thm}\label{ThmCpctDiffOp}
A differential operator of order $k-l$ is compact when extended as $L^2_s(\Gamma(E)) \to L^2_{s-k}(\Gamma(F))$ for $l\ge 1$.
\end{thm}

\kommentar{  For an unbounded, closed operator $T$, $\cD(T)$ is a Banach space in the {\em graph norm}
\[
 \|x\|_{\gr} \coloneqq \|x \| + \| Tx \|.
\]\begin{exe}
 The following are equivalent:
\begin{enumerate}
 \item $T\co B_1 \to B_2$ is an unbounded Fredholm operator
 \item $T \co \cD(T) \to B_2$ is a bounded Fredholm operator, where $\cD(T)$ is equipped with the graph norm.
 \item If $T$ is self-adjoint, the bounded
\end{enumerate}
\end{exe}}

\kommentar{Finde raus, in welcher Allgemeinheit das gilt! If $M$ is a closed Riemannian manifold and we have vector bundles $E$ and $F$ over $M$ with a metric, then we can extend $\Gamma(E)$ and $\Gamma(F)$ to a Hilbert space using the induced $L^2$--inner product and every elliptic operator $D \in \Diff(E,F)$ extends to an unbounded Fredholm operator from the $L^2$--extension $L^2(\Gamma(E))$ of $\Gamma(E)$ to $L^2(\Gamma(F))$.}

\subsection{Signature and spectral flow}

Let $W$ be a closed, orientable $4$-manifold. If $W$ is smooth, then $a,b\in
H_2(W;\Z)$ can be represented by oriented surfaces $A,B$ in $M$ which
intersect transversely. Then define $Q(a,b)=A \cdot B$, where $A\cdot B$ is
the (oriented) intersection number. E.g. for the torus ($n=2$) we get
$$
Q = \left( \begin{array}{cc} 0 & 1 \\ -1 & 0 \end{array} \right).
$$
\begin{exa} In four dimensions is symmetric, because $A$ and $B$ are
$2$-dimensional.
\begin{enumerate}
\item For $S^4$, $Q=0$, because the second homology is $0$.
\item For $S^2 \times S^2$ we get
$Q=\sigma_1$, where $\sigma_1$ is the Pauli spin matrix
$$
\sigma_1 = \left( \begin{array}{cc} 0 & 1 \\ 1 & 0 \end{array} \right)
$$
\item For $\C P^2$ and $\overline{\C P^2}$ we have $Q=(1)$ and $Q=(-1)$ respectively.
\end{enumerate}
\end{exa}

The {\em signature} of a closed, oriented $4$--manifold $W$ is the signature of its intersection form. {\em Hirzebruch's Signature Theorem} states that
\begin{equation}\label{EqSignThm}
 \sign(W) = \frac{1}{3}\int_W p_1(TW,\nabla),
\end{equation}
where for a real vector bundle $E$ equipped with a connection $\nabla$
\[
p_1(E,\nabla) = -\frac{\tr(F_\nabla^2)}{8\pi^2}
\]
is the first Ponrjagin form\footnote{The total Pontrjagin form of $\nabla$ is
\[
p(E,\nabla) = \det\left(\left(I - \left(\frac{F_\nabla}{2\pi}\right)^2\right)^{\frac{1}{2}}\right) =\exp(\tfrac{1}{2}\tr[\log(1+(\tfrac{i}{2\pi}) F_\nabla)]).
\]}
The integral in \eqref{EqSignThm} is independent of $\nabla$ because $W$ is closed.

If $\rho \co \pi_1(W) \to \Aut(V)$ is a representation of the fundamental group, one defines the twisted signature $\sign_{\rho}(W)$ using cohomology with local coefficients. An application of the signature formula \eqref{EqSignThm} and its twisted version to the closed double $W \cup_M -W$ shows that the difference, the {\em signature defect},
\[
 \sign_\rho(W) - k\cdot \sign(W)
\]
only depends on the topology of the boundary $\partial W = M$ as well as the restriction of $\rho$ to $\pi_1(M)$. This can be defined for arbitrary closed and oriented 3--manifolds $M$ independently of $W$ and is called the {\em rho invariant of $M$}.

The {\em Atiyah-Patodi-Singer Index Theorem} identifies for manifolds with boundary identifies the correction term in great generality. For a formally self-adjoint elliptic differential operator $D$ of first order, acting on sections of a vector bundle over a closed manifold $M$, one defines the eta function
\[
 \eta(D,s) \coloneqq \sum_{0 \neq \lambda \in \spec(D)} \frac{\sgn(\lambda)}{|\lambda|^s}, \quad \re(s) \text{ large}.
\]
The function $\eta(D,s)$ admits a meromorphic continuation to the whole $s$-plane with no pole at the origin. Then $\eta(D) \coloneqq \eta(D,0)$ is called the {\em eta invariant of $D$}. A special case of the Atiyah-Patodi-Singer Index Theorem for a compact, oriented 4--manifold $W$ with boundary $M$ is
\[
 \sign_\rho W = \frac{k}{3} \int_W p_1(TW,\nabla^g) - \eta(D_A)
\]
where $\hol(A) = \rho|_{\pi_1 M}$ and $D_A$ is the odd signature operator\kommentar{already defined?}.

\subsection{The Hessian of the Chern-Simons
function}\label{HessianOfCS}

We need to consider a family of operators, for which we can study spectral flow. In particular, in the most natural definition of spectral flow, the spectrum of each operator needs to be real and discrete with 0 not an accumulation point, vary continuously along the path with some restrictions to the starting and endpoint of the path. A detailed study is bound to be long and technical, therefore we refer to \cite{kato76,booss-lesch-phillips2005} for details and proofs. In view of this it is essential, that $H_A$ is formally self-adjoint operator with respect to the $L^2$ inner product on $\Omega^1(M;\fraksu(2))$:
\[
\la * d_A f, g \ra_{L^2} = \la f, *d_A g \ra_{L^2}.
\]
At a flat connection $A$, the kernel of $H_A$ contains the infinite-dimensional subspace of tangent vectors to the $\cG$--orbit through $A$ because of the gauge invariance of $\cs$. This is the underlying reason for considering the Morse theory of the quotient space $\cB^*$.

The differential operator $\tilde H_A$ is also formally self-adjoint with respect to the $L^2$--inner product restricted to $\ker d_A^*$. Furthermore, it is an elliptic first order differential operator and we can therefore apply Theorem \ref{ThmDiffOp} to get a self-adjoint extension to the $L^2$-completion $L^2(\Omega^1(M;\fraksu(2)))$ with compact resolvent. Furthermore, the spectrum consists only of the eigenvalues of $\tilde H_A$ with no finite accumulation point.

The spaces $\ker d_A^*$ form a smooth vector bundle over $\cA^*$. One can therefore show that the eigenvalues change continuously along a path of irreducible connections. However, it is unclear how to define spectral flow for paths starting at the trivial connection. Taubes observed that the spectral flow of $\tilde H_{A_t}$ for $A_t$ irreducible equals the spectral flow of a better behaved family of operators.

Consider, for any connection $A$, the twisted de Rham sequence
\begin{equation}\label{EqDeRhamSeq}
 0 \to \Omega^0(M;\fraksu(2)) \stackrel{d_A}{\to} \Omega^1(M;\fraksu(2)) \stackrel{d_A}{\to} \Omega^2(M;\fraksu(2)) \stackrel{d_A}{\to} \Omega^3(M;\fraksu(2)) \to 0
\end{equation}
The sequence \eqref{EqDeRhamSeq} is an elliptic complex\footnote{i.e. a complex of differential operators, whose induced complex of principal symbols is exact.}, when $A$ is flat. It is not a complex at a non-flat connection, but can be made into a complex if one substitutes $d'A \coloneqq \tilde H_A \circ \proj_{\ker d_A^*}$ for the middle map. When $A$ is flat we have $d'_A = d_A$. Whether or not $A$ is flat, the {\em odd signature operator} (obtained by folding up the sequence \eqref{EqDeRhamSeq})
\begin{align*}
D_A\co \Omega^{0}(M;\fraksu(2)) \oplus \Omega^{1}(M;\fraksu(2))&\longrightarrow \Omega^{0}(M;\fraksu(2)) \oplus \Omega^{1}(M;\fraksu(2))\\
(\alpha,\beta) &\longmapsto (d_A^* \beta, d_A \alpha+ *d_A \beta)
\end{align*}
is a formally self-adjoint, elliptic, first order differential operator. One also has the formally self-adjoint operator
\[
D'_A(\alpha,\beta)\coloneqq (d_A^* \beta, d'_A \alpha+ *d_A \beta).
\]
For $A$ flat, $D_A = D'_A$. For a general irreducible connection $A$, the difference of the self-adjoint Fredholm extensions of $D_A$ and $D_A'$ to $L^2$ can be extended to a compact operator. To see this, consider the Hodge Laplacian $\Delta_A = d_A^* d_A$ on the 0-forms. It is injective for $A$ irreducible because $d_A^* d_A \alpha  = 0$ implies $\|d_A \alpha\|_{L^2} = \la d_A^*d_A\alpha, \alpha\ra_{L^2} = 0$ so that $\ker \Delta_A \subset \ker d_A = \{0\}$. By a ``boot-strapping'' argument for the bounded extensions $L^2_{k+2}\to L^2_{k}$ of $\Delta_A$ for all $\beta \in \Omega^0(M;\fraksu(2))$ and $k\in \N$ the equation $\Delta \alpha = \beta$ has a smooth solution (a property known as {\em elliptic regularity}). $\Delta^{-1}_A\co \Omega^0(M;\fraksu(2)) \to \Omega^0(M;\fraksu(2))$ extends to a bounded operator $L^2_k \to L^2_{k+2}$ for $k \in \N$.  Since $\proj_{\ker d^*_A} \beta = \beta - d_A \Delta_A^{-1} d^*_A \beta$ we can compute
\[
 (D_A - D_A') (\alpha,\beta) = (*F_A \Delta_A^{-1} d^*_A - d_A \Delta_A^{-1} *F_A + d_A \Delta_A^{-1} * F_A d_A \Delta_A^{-1} d_A^*) (\beta).
\]
which extends to a bounded operator $L^2_k\to L^2_{k+1}$ and by the Sobolev embedding theorem can be considered as a compact operator
\[
L^2\left(\Omega^{0}(M;\fraksu(2)) \oplus \Omega^{1}(M;\fraksu(2))\right) \to L^2\left(\Omega^{0}(M;\fraksu(2)) \oplus \Omega^{1}(M;\fraksu(2))\right).
\]
Since the space of compact operators is contractible and the spectral flow is a homotopy invariant rel endpoints, the spectral flow of $D_{A_t}$ and $D'_{A_t}$ agree for $A_t$ irreducible, $t\in [0,1]$, and $A_0$ and $A_1$ flat. If $A$ is an irreducible connection, then $\coker d_A = \ker d_A^*$ gives the decomposition
\[
\Omega^0(M;\fraksu(2))\oplus \Omega^1(M;\fraksu(2))\cong
\Omega^0(M;\fraksu(2)) \oplus \im d_A \oplus \ker d_A^*
\]
and allows us to define
\[ D''_A\coloneqq d_A \oplus d_A^* \oplus \tilde H_A \co \Omega^0(M;\fraksu(2)) \oplus \im d_A \oplus \ker d_A^* \to \im d_A \oplus \Omega^0(M;\fraksu(2)) \oplus \ker d_A^*.
\]
Again, the spectrum of (the $L^2$--extensions of) $D''_A$ consists only of their respective eigenvalues with no finite accumulation point by Theorem \ref{ThmDiffOp}.
\begin{exe}
 Show that $\lambda$ is an eigenvalue of $d_A$ if and only if $-\lambda$ is an eigenvalue of $d_A^*$. Therefore, $D''_A$ differs from $\tilde H_A$ only by an operator with symmetric spectrum.
\end{exe}
It follows that the spectral flow of $\tilde H_{A_t}t$ equals the spectral flow of $D_{A_t}$ for $A_t$ irreducible, $t\in [0,1]$, $A_0$ and $A_1$ flat, the main advantage of $D_{A_t}$ being, that the spectral flow of $D_{A_t}$ makes sense whether or not $A_t$ is a path of irreducible connections, i.e. even when the subspaces $\ker d_A^* \subset \Omega^1(M;\fraksu(2))$ do not vary continuously, notably for paths starting at the trivial connection.

Using the {\em Atiyah-Patodi-Singer index theorem} one can express the spectral flow of a path $A_t$ of $\SU(2)$--connections as twice the integral of the second Chern-class of $A_t$ over $M\times I$ up to correction terms which vanishes when $A_0$ is gauge-equivalent to $A_1$. By applying Exercises \ref{ExeCSform}(\ref{CWform}), \ref{ExeSU2normalization} and Stokes' Theorem we see that that on $\cB$, $\SF$ is well-defined modulo 8.
\kommentar{The last paragraph needs to be carefully checked and preferably explicitely written down.}

\subsection{Perturbations}

Just like in the original definition by Casson, the set of critical points of $\cs$ may be complicated and it is often necessary to perturb the $\cs$ to make it nice. We want to replace $\cs$ by $\cs + h$ where the function $h$ satisfies the following basic requirements:
\begin{enumerate}
 \item The perturbation $h \co \cB \to \R$ should be smooth.
 \item The operator defined by the hessian of $\cs + h$ at $A$ should be a compact perturbation of $\tilde H_A$.
 \item The family of admissible $h$ should be large enough so that one can prove various general position results about $\cs + h$ on the strata and their normal bundles of $\cB$.
\end{enumerate}

We will be content with defining the perturbations without proving the above requirements. First, fix a collection of smooth embeddings $\gamma_i \co D^2 \times S^1 \to M$, $i=1,\ldots, N$. Taubes requires that the embeddings have disjoint images, however we will assume that  $\gamma_i(x,1) = \gamma(x)$ for all $i$ and $x \in D^2$ as it is done \cite{floer88}.

Denote by $\cP$ the set of $C^r$ functions $f\co \SU(2)^N \to \R$ invariant under the conjugation action of $\SU(2)$ (for some fixed large $r$). Let $P|_{D^2}$ denote the restriction of the principal bundle $P$ to the disc $D^2 = \gamma_i(D^2\times \{1\})$. Given a connection $A \in \cA$, consider the map $\Hol_A \co P|_{D^2} \to \SU(2)^N$ defined by $\Hol_A(x) = (\hol_A(\gamma^x_1,x),\ldots,\hol_A(\gamma^x_N,x))$ where $\gamma^x_i (t)= \gamma (x,t)$ and $t\in [0,2\pi]$ parametrizes $S^1$. Then for $f\in \cP$, $f \circ \Hol$ descends to a well-defined function on $\cB \times D^2$. Fix a smooth cut-off function $\eta$ on $D^2$ which vanishes near the boundary. The space of {\em admissible perturbations} is defined to be the Fr\'echet space of functions of the form
\begin{align*}
 h_f \co \cA &\to \R\\
A &\mapsto \int_{D^2} f(\Hol_A(x))\eta(x) \, d^2x.
\end{align*}
Just like in \ref{EqDeRhamSeq} the sequence
\[ 0 \to \Omega^0(M;\fraksu(2)) \stackrel{d_A}{\to} \Omega^1(M;\fraksu(2)) \stackrel{d_{A,f}}{\to} \Omega^2(M;\fraksu(2)) \stackrel{d_A}{\to} \Omega^3(M;\fraksu(2)) \to 0
\]
is a complex at an $f$-perturbed flat connection and can be folded up to a self-adjoint operator, the {\em perturbed twisted odd signature operator}
\begin{align*}
D_{A,f}\co \Omega^{0}(M;\fraksu(2)) \oplus \Omega^{1}(M;\fraksu(2))&\longrightarrow \Omega^{0}(M;\fraksu(2)) \oplus \Omega^{1}(M;\fraksu(2))\\
(\alpha,\beta) &\longmapsto (d_A^* \beta, d_{A,f} \alpha+ *d_A \beta).
\end{align*}
It can be shown that $D_{A,f}$ is a compact perturbation of $D_A$ (see \cite[Lemma 2.3]{himpel-kirk-lesch2004}).

\begin{exe} Show the following:
 \begin{enumerate}
  \item $d_{A,f}$ is independent of the metric.
  \item $\ker(D_{A,f})$ is independend of the metric.
 \end{enumerate}
\end{exe}
Let $A_i$ be an $f_i$--perturbed flat connection, $i=0,1$. Since the space of connections as well as the space of perturbations are contractible and the spectral flow is a homotopy invariant, $\SF(D_{A_t,f_t})$ is independent of the choice of paths $A_t$ and $f_t$. By the above exercise $\SF(D_{A_t,f_t})$ is also independent of the Riemannian metric.

Taubes \cite{taubes90} proves that for a generic perturbation, his invariant equals Casson invariant for homology 3--spheres. Note that this also shows the independence of the perturbation. For a new proof (up to an overall sign) see \cite{himpel-kirk-lesch2004}.

\section{Classical Chern-Simons theory}

\subsection{The Chern-Simons line bundle}

Recall that the behaviour of the Chern-Simons function $\cs_s\co \cA_P \to \R$ for a section $s \co X \to P$ of a principal $G$--bundle $P \to X$ under gauge transformations \eqref{EqCSGauge} only depends on the restriction of the gauge transformation to the boundary and the fact that
\[
 \cs_{\Phi\circ s}(A) = \cs_s(\Phi^* A) \quad \text{for } A \in \cA_P, \Phi \in \cG_P.
\]
It can easily be seen that any two sections are related via a gauge transformation, which implies that $\cs_s$ only depends on the restriction of $s$ to the boundary. We will use this dependence to define a principal $U(1)$--bundle over $\cA_{P|_{\partial X}}$ such that
\begin{equation}\label{EqVectorInLinebundle}
\CS_X(A) \in \cL_{A|_{\partial X}}
\end{equation}
for $\CS_X(A)(s) \coloneqq \exp(2\pi i \cs_s(A))$\footnote{Freed \cite{freed95} formally uses the expression $\exp(2\pi i \cs_X(A))$ instead of $\CS_X(A)$. Despite its advantage of having a suggestive meaning, it is a bit clumsy and can possibly be misunderstood.}, which is essential for defining a {\em local Lagrangian field theory}\footnote{To a closed oriented 2--manifold we assign the set of possible boundary values of fields, the space of connections on $\Sigma$. To a compact oriented 3--manifold we assign a {\em space of fields}, the space of connections on $X$ with these boundary values. The dynamics are determined by a local Lagrangian, in our case the Chern-Simons form.}. The fact, that all principal $G$--bundles $P \to X$ are trivializable and sections correspond to trivializations of $P$, allows us to suppress the reference to $P$ in \eqref{EqVectorInLinebundle}.

Let us therefore consider a principal $G$--bundle $Q \to \Sigma$ over a closed surface $\Sigma$. We may think of a principal $U(1)$--bundle $\cL$ over $\cA_Q$ as the (complex) line bundle $\cL$ associated to the defining representation $U(1) \hookrightarrow \C^*=\GL(\C)$ over $\cA_Q$ known as the {\em Chern-Simons line bundle}.

 As the space of connections $\cA_Q$ is contractible, $\cL$ will be trivializable. We could therefore describe it using one single chart and even define it to be $\cA_Q \times \cL$. However, we need the trivialization \[\phi_s \co \cL_Q \to \cA_Q\times \C\] to depend (non-trivally) on the section $s\co \Sigma \to Q$ in the same way the Chern-Simons function behaves, so that Equation \eqref{EqVectorInLinebundle} is satisfied. In order for $\cL$ to be a line bundle, the transition functions $\phi_{ss'} \coloneqq \phi_{s'}\phi_{s}^{-1}$ must then satisfy the cocycle condition
\begin{equation}\label{EqCocycleCondition}
\phi_{s_3s_2} \phi_{s_2s_1} = \phi_{s_3s_1}.
\end{equation}
If we want to have a concrete description of $\cL$, we can set $\cL \coloneqq \cA_Q \times \C$ for a fixed section $s$ and require that the cocycle condition \eqref{EqCocycleCondition} is satisfied. However, setting $\cL \coloneqq \cA_Q \times \C$ for a section $s'$ will give an isomorphic bundle (via $\phi_{ss'}$), which satisfies the same cocycle condition and will therefore be equivalent for our purpose. Notice, that fixing a section $s\co \Sigma \to Q$ then also gives an isometry $\cL_a \cong \C$.

A section $s\co \Sigma \to Q$ gives identifications $s^*\co \cA_Q \to \Omega^1(\Sigma;\frakg)$ and $g^s \co \cG_Q \to C^\infty(\Sigma,G)$ determined by $\Phi\circ s(x) = s(x) \cdot g^s(\Phi)(x)$. In view of the behaviour \eqref{EqCSGauge} of the Chern-Simons function under a gauge transformation $\Phi$ we want to have for $A\in \cA_Q$
\begin{equation}
\phi_{\Phi\circ s} = c_{\Sigma}(s^*A,g^s(\Phi))\phi_s
\end{equation}
where
\begin{equation}\label{TransitionLineBundle}\textstyle
 c_{\Sigma}(a,g) \coloneqq \exp\left(2\pi i \left(\int_\Sigma \la \Ad_{g^{-1}} a \wedge g^*\theta\ra + W_\Sigma(g)\right)\right)
\end{equation}
for $a\in \cA_Q$ and $W_\Sigma$ the Wess-Zumino-Witten functional defined in Exercise \ref{ExeWZW}. The cocycle condition \eqref{EqCocycleCondition} then translates to
\[
 c_\Sigma(s_2^*A,g^{s_2}(\Phi_2))  c_\Sigma(s_1^*A,g^{s_1}(\Phi_1)) = c_\Sigma(s_1^*A,g^{s_1}(\Phi_2\circ\Phi_1))
\]
where $s_3 = \Phi_2 s_2$ and $s_2 = \Phi_1 s_1$. If we let $g_1 = g^{s_1}(\Phi_1)$, $g_2 = g^{s_2}(\Phi_2)$ and $a = s_1^*A$, then we have
\begin{align*}
s_1(x) \cdot g^{s_1}(\Phi_2\circ \Phi_1)(x) &= \Phi_2 \circ \Phi_1 \circ s_1 (x)\\
&= \Phi_1\circ s_1(x)\cdot g^{\Phi_1\circ s_1}(\Phi_2)(x)\\
&= s_1 (x) \cdot g^{s_1}(\Phi_1)(x)\cdot g^{s_2}(\Phi_2)(x)
\end{align*}
and $s_2^*A = s_1^*\Phi_1^*A = s_1^*A \cdot g^{s_1}(\Phi) = a \cdot g_1$ so that we can rewrite the cocycle condition as
\begin{equation}\label{EqCocycleCondSimple}
c_\Sigma(a\cdot g_1,g_2)c_\Sigma(a,g_1) = c_\Sigma(a,g_1g_2).
\end{equation}

\begin{prop}
 The map $c_\Sigma$ satisfies the cocycle condition \ref{EqCocycleCondSimple}.
\end{prop}

\begin{proof}
If $\omega(g) = -\frac{1}{6}\la g^*\theta\wedge [g^*\theta\wedge g^*\theta]\ra$, then \[\omega(g_1g_2) = \omega(g_1) + \omega(g_2) + d\sigma(g_1,g_2),\] where $\sigma(g_1,g_2) = \la g_1^* \theta \wedge \Ad_{g_2}g_2^*\theta\ra$.\kommentar{Notice that $(g_1g_2)^*\theta (x) = \theta(g_1(x)g_2(x)) = $.} Then we can compute
\[
\textstyle
 \exp(2\pi i W_\Sigma(g_1g_2)) = \exp\left(2\pi i \int_\Sigma \sigma(g_1,g_2)\right) \exp(2\pi i W_\Sigma (g_1)) \exp(2\pi i W_\Sigma(g_2))
\]
and
\[
 c_\Sigma(a,g_1g_2) = \exp\left(2\pi i\left(\la \Ad_{g_1^{-1}}a \wedge g_1^*\theta + \Ad_{(g_1g_2)^{-1}} a \wedge g_2^*\theta\ra + W_{\Sigma}(g_1g_2)\right)\right).
\]
Comparing this to
\begin{align*}
 c_\Sigma(a\cdot g_1, g_2) & = \exp(2\pi i \la (\Ad_{g_2^{-1}}(\Ad_{g_1^{-1}} a) + \Ad_{g_2^{-1}}(g_1^*\theta)) \wedge g_2^*\theta + W_\Sigma (g_2) )\\
\tag*{and} c_\Sigma(a,g_1) &= \exp(2\pi i \la \Ad_{g_1^*}a \wedge g_1^{-1}\theta + W_\Sigma (g_1) )
\end{align*}
completes the proof.
\end{proof}

If $u \mapsto A_u$ is a smooth family of connections varying over a smooth manifold $U$. Then it is clear from the definition of $c_\Sigma$, that the transition functions $u\mapsto c_\Sigma(s^*A_u,g^s(\Phi))$ are smooth, so that $\cL$ is a smooth vector bundle over $U$. Furthermore, we constructed it to satisfy \eqref{EqVectorInLinebundle} so that $A \mapsto \CS_X(A)$ is a section of $\cL$, which is the {\em Chern-Simons invariant for manifolds $X$ with boundary}. It is convenient to define $\CS_\emptyset(A) \coloneqq 1$ and $\cL_\emptyset \coloneqq \C$ so that for any 3--manifold $X$ with $\partial X = \emptyset$ we can write $\CS_X(A) \in \cL_{A|_{\partial X}}$.

For a fixed (connected, simply-connected, compact) Lie group $G$ we introduce the category \[\cA_X = \bigcup_{P} \cA_P\] where $\{P\}$ is the collection of all principal $G$--bundles over $X$ with morphisms being bundle maps covering the identity. Objects $A_i \in \cA_X$ are called equivalent if there is a morphism $\Phi$ with $\Phi^*A_2 = A_1$. We denote the set of equivalence classes by $\overline{\cA_X}$. Certainly $\overline{\cA_X} \cong \cA_P/\cG_P$ for a fixed principal $G$--bundle $P$, but ${\overline\cA_X}$ has the advantage of clearly being independent of $P$.

  Now it is straight-forward to see why $\cs$ is the action of a local Lagrangian field theory, i.e. that the assignments
\begin{alignat*}{3}
 a &\mapsto \cL_a, \quad &&a &\in \cA_\Sigma\\
\tag*{and} A &\mapsto \CS_X(A), \quad &&A &\in \cA_X
\end{alignat*}
for a closed oriented 2--manifold $\Sigma$ and a compact oriented 3--manifold $X$ satisfy the following four properties.
\begin{setlength}{\leftmargini}{2em}
\begin{description}
 \item[Functoriality] If $\Psi\co Q' \to Q$ is a bundle map covering an orientation preserving diffeomorphism $\psi\co \Sigma' \to \Sigma$ and $a \in \cA_Q$, a section $s'\co \Sigma' \to Q'$ gives an isometry $\cL_{\psi^*a} \cong \C$, the induced section $s = \Psi s' \psi^{-1}\co \Sigma \to Q$ gives an isometry $\cL_a \cong \C$, then there is an induced isometry
\[
\Psi^*\co \cL_a \to \cL_{\Psi^*a}
\]
given by the identity map, satisfying $\Psi_1^*\Psi_2^* = (\Psi_2\circ\Psi_1)^*$.
\begin{exe}$\Psi^*$ is independent of the choice of $s'$.
\end{exe}
If $\Phi\co P'\to P$ is a bundle map covering an orientation preserving diffeomorphism $\phi\co X' \to X$ and $A\in \cA_P$, then similarly
\[
(\Phi|_{\partial X})^*\CS_X(A) = \CS_{X'}(\Phi^*A).
\]
\item[Orientation] Consider the line bundle $\overline{\cL_{\Sigma,a}}$ determined by $-c_\Sigma$. There is a natural isometry\kommentar{Duales Linienbuendel einfuehren!}
\begin{align*}
 \cL_{-\Sigma,a} &\cong \overline{\cL_{\Sigma,a}},\\
\CS_{-X}(A) &= \overline{\CS_X(A)},
\end{align*}
since the $c_\Sigma$ and $\cs_s$ (after fixing $s\co X \to P$) change sign when the orientation of $\Sigma$ and $X$ are reversed.
\item[Additivity] If $\Sigma = \Sigma_1  \sqcup \Sigma_2$ and $a_i \in \cA_{\Sigma_i}$, then---since the integral over a disjoint union is the sm of the integrals---there is a natural isometry
\[
 \cL_{a_1 \sqcup a_2} \cong \cL_{a_1} \tensor \cL_{a_2}.
\]
If $X=X_1 \sqcup X_2$ is a disjoint union and $A_i \in \cA_{X_i}$, then
\[
 \CS_{X_1\sqcup X_2} (A_1 \sqcup A_2) = \CS_{X_1}(A_1) \tensor \CS_{X_2}(A_2).
\]
\item[Gluing] Suppose $\Sigma \hookrightarrow X$ is a closed, oriented submanifold and $X^c$ is the manifold obtained by cutting $X$ along $\Sigma$. Then $\partial X^c = \partial X \sqcup \Sigma \sqcup -\Sigma$.  For $A\in \cA_X$, $A^c \in \cA_{X^c}$ the induced connection and $a = A|_\Sigma$ the fact $\int_{X^c} = \int_X$ implies
\[
 \CS_X(A) = \tr_a(\CS_{X^c}(A^c)),
\]
where $\tr_a$ is the contraction
\[
\tr_a \co \cL_{A^c|_{\partial A^c}} \cong \cL_{A|_{\partial X}} \tensor \cL_a \tensor \overline{\cL_a} \to \cL_{A|_{\partial X}}
\]
using the Hermitian metric on $\cL_a$.
\end{description}
\end{setlength}

\kommentar{\subsection{Index Theory}}

\subsection{Classical solutions and the Hamiltonian Theory}

We have seen that for a closed connected oriented manifold $X$, the map $\cs_s\co \cA_P \to \R$ induces a map $\cs_X \co \cA_X\to \R/\Z$ with
\[
 d\cs_A(\eta) = 2 \int_X \la F_A\wedge \eta \ra,
\]
so that the critical points are precisely the flat connections. Even if $X$ has non-empty boundary, the equivalence classes $\cM_X = \cM(X,G)$\kommentar{Notation vorher veraendern.} of flat $G$--connections on $X$ are a subset of $\overline{\cA_M}$.\kommentar{As we have seen in Theorem \ref{ThmFlatConnEqualRep}, there is a bijection between $\cM_M$ and $\Hom(\pi_1(M),G)/G$ for any connected manifold $M$.}

For a compact oriented 3--manifold $X$ with boundary we consider the category\kommentar{Kategorie frueher einfuehren...}
\begin{align*}
 \cA_X(a)  & = \{ A\in \cA_X \mid A|_{\partial X} = a\} \quad \text{for } a \in \cA_X
\end{align*}
with morphisms being bundle maps which are the identity over $\partial X$. Then we set
\[
 \overline{\cA_X(a)} \coloneqq \cA_X(a)/\sim,
\]
where $A,A' \in \cA_X(a)$ are equivalent if $A' = \Phi^*A$ by some morphism $\Phi$.

It is not difficult to see that for $A \in \cA_X(a)$ we still have
\[
 d\cs_A(\eta) = 2 \int_M \la F_A\wedge \eta \ra,
\]
since and any path of connections in $\cA_X(a)$ is constant along the boundary and therefore tangent vectors vanish on the boundary.
The equivalence classes of the critical points form the subcategory of equivalence classes of flat connections $\cM_X(a) \subset \overline{\cA_X(a)}$ which restrict to $a$ on the boundary, which is certainly empty if $a$ is not flat. In order to relate $\cM_X(a)$ to $\cM_X$ we need to divide out by the symmetries/morphisms on the boundary. One can formulate this in terms of a functor $F$ from $\cA_\partial X$ to a category containing all $\cM_X(a)$, e.g. the category of algebraic varieties.\kommentar{This needs to be checked... category of semi-algebraic sets? What are the morhpisms?}

Consider $a' = \Psi^* a$ for a morphism $\Psi$ in $\cA_{\partial X}$. If $A$ is a connection on $P\to X$ with $A|_{\partial X} = a$ and $\Phi\co P' \to P$ a morphism for $\cA_{X}(a)$ with $\Psi = \Phi|_{\partial X}$, then the $\Phi$ induces a functor from $\overline{\cA_X(a')}$ to $\overline{\cA_X(a)}$ mapping $[\Phi^*A]$ to $[A]$.
\begin{exe}
Check that $\Phi$ always exists, that the functor is independent of the choice of $\Phi$. Define the functor on morphisms of $\overline{\cA_X(a)}$.
\end{exe}
The functor $\overline{\cA_X(a')} \to \overline{\cA_X(a)}$ maps $\cM_X(a')$ isomorphically to $\cM_X(a)$.
Therefore, there is a functor from $\cA_{\partial X}$ to the category of algebraic varieties sending $a$ to $\cM_{X}(a)$ and morphisms to isomorphisms. $\cM_X$ is simply the space of equivariant sections of this functor. Alternatively,
\[
\cM_X = \bigcup_{a} \cM_X(a)/\sim,
\]
where $A \in \cM_X(a)$ and $A'= \cM_X(a')$ are equivalent, if $A' = \Phi^*(A)$ for a morphism $\Phi$ such that $a' = (\Phi|_{\partial X})^* a$.
The restriction to $\partial X$ gives the diagram
\[
\begin{diagram}
 \node{\cM_X} \arrow{e,J} \arrow{s,l}{r_X} \node{\overline{\cA_X}}\arrow{s,r}{r_X}\\
\node{\cM_{\partial X}} \arrow{e,J} \node{\overline{\cA_{\partial X}}}
\end{diagram}
\]

We would like to construct a Hamiltonian theory consisting of assignments
\begin{align*}
 \Sigma &\mapsto \cM_\Sigma\\
 X & \mapsto (r_X \co \cM_X \to \cM_{\partial X}).
\end{align*}
It turns out that (a subset of) $\cM_\Sigma$ is a symplectic manifold, namely the sympletic quotient corresponding to the curvature as the moment map on the space of connections over $\Sigma$ with the group of gauge transformations acting freely. More precisely, we we will have
\begin{align*}
 \Sigma &\mapsto \cL'_\Sigma\\
 X & \mapsto (\CS\co \cM_X \to r^*_X\cL'_{\partial X}),
\end{align*}
where $\cL' = \cL|_{\cM_\Sigma}$.

\begin{prop}
 Let $\Sigma$ be a closed oriented 2--manifold and $Q\to \Sigma$ a principal $G$--bundle. Then the Chern-Simons action defines a unitary connection $B$ on the Hermitian line bundle $\cL_Q \to \cA_Q$ and under the assumption that $\la\cdot,\cdot\ra$ is nondegenerate
\[
 \omega(\eta_1,\eta_2) \coloneqq - 2\int_\Sigma \la\eta_1\wedge\eta_2\ra = \tfrac{i}{2\pi}F_B (\eta_1,\eta_2)\quad \text{for } \eta_1,\eta_2 \in T_a \cA_Q = \Omega^1(\Sigma;\frakg).
\]
is a symplectic form. The action of $\cG_Q$ on $\cA_Q$ lifts to $\cL_Q$, and the lifted action preserves the metric and connection. The induced moment map is
\[
 \mu_\xi(a) = 2 \int_\Sigma \la F_a \wedge \xi\ra,
\]
where $\xi \in T_{\Id} \cG \cong \Omega^0_\Sigma(\frakg)$ and $a \in \cA_Q \cong \Omega^1(\Sigma;\frakg)$ using a section $s \co \Sigma \to Q$. $\cG_Q$ acts freely on the space $\cA_Q^*$ of irreducible connections, so that the moduli space $\cM^*_Q$ of irreducible flat connections is the symplectic or Marsden-Weinstein quotient $\cA^*_Q//\cG_Q$ and therefore a symplectic manifold. There is an induced line bundle $\overline{\cL_Q} \to \cM^*_Q$ with metric and connection $\overline{B} = B|_{\cM^*_Q}$, and $\tfrac{i}{2\pi} F_{\overline{B}}$ is the symplectic form on $\cM^*_Q$.
\end{prop}

\begin{proof}Fix a section $s\co \Sigma \to Q$. We have seen in the previous section that this induces a (unitary) trivialization $\phi_s \co \cL_Q \to \cA_Q \times \C$.
Define the $1$--form $B_s \in \Omega^1(\cA_Q,\R i)$

\[
 (B_s)_a (\eta) = 2\pi i \int_\Sigma s^*\la a\wedge \eta\ra, \quad a \in \cA_Q, \eta \in T_a\cA_Q = \Omega^1(\Sigma,\frakg),
\]
where $\R i$ is the Lie algebra of $U(1)$. In order for this 1--form to give a connection in the line bundle $\cL_Q \to \cA_Q$ we must check that this transforms properly under gauge transformations. We have seen that
\[
 \phi_{\Psi s} = c_\Sigma(s^*a,g^s(\Psi))\phi_s.
\]
for a gauge transformation $\Psi  \co Q \to Q$, equivalently for the corresponding sections $\sigma_s(a) = \phi_s^{-1}(a,1)$
\[
c_\Sigma(s^*a,g^s(\Psi)) \sigma_{\Psi s} = \sigma_s.
\]
Then by applying Exercise \ref{trivializationchange} to $\cL$ (thinking of it as a $U(1)$--bundle) we deduce that in order for $B_s$ to induce a connection $B$ on $\cL$ it needs to satisfy
\begin{align*}
 (B_{\Psi s})_a(\eta) &= (B_{s})_a(\eta) - \frac{d c_\Sigma(s^*a,g^s(\Psi))(\eta)}{c_\Sigma(s^*a,g^s(\Psi))}\\
&=(B_{s})_a(\eta) - 2\pi i\int_\Sigma \la \Ad_{g^{-1}}\eta \wedge g^*\theta\ra,
\end{align*}
where $d c_\Sigma(s^*a,g^s(\Psi))$ is the exterior derivative at $a$ of the map
\begin{align*}
\cA_Q &\to U(1)\\
a &\mapsto c_\Sigma(s^*a,g^s(\Psi)).
\end{align*}
We compute for $u\co P \to G$ associated to $\Psi$ and $g = g^s(\Psi) = u\circ s$ using Lemma \ref{LemGaugeAssociatedTransform} that this is indeed the case
\begin{align*}
(B_{\Psi s})_a(\eta) &= 2\pi i \int_\Sigma s^*\la \Psi^*a\wedge \Psi^*\eta\ra\\
&=2\pi i \int_\Sigma s^*\la\Ad_{u^{-1}}a\wedge \Ad_{u^{-1}}\eta\ra + 2\pi i\int_\Sigma s^*\la u^*\theta\wedge \Ad_{u^{-1}}\eta\ra\\
&=(B_s)_a(\eta) - 2\pi i\int_\Sigma \la \Ad_{g^{-1}}\eta \wedge g^*\theta\ra.
\end{align*}
\kommentar{Waere wirklich gut, das mit g_s und u etc. abzugleichen...}Therefore, $B$ is a unitary connection on $\cL_Q \to \cA_Q$.

Since the structure group is abelian we have
\begin{align*}
 \tfrac{i}{2\pi}(F_B)_a(\eta_1,\eta_2) &= \tfrac{i}{2\pi}(dB)_a(\eta_1,\eta_2) = \tfrac{i}{2\pi} \left(\eval{\frac{d}{dt}B_{a+t\eta_1}(\eta_2)}{t=0} - \eval{\frac{d}{dt}B_{a+t\eta_2}(\eta_1)}{t=0}\right)\\
& =  -2\int_\Sigma\la\eta_1,\eta_2\ra = \omega(\eta_1,\eta_2).
\end{align*}
If $\la\cdot,\cdot\ra$ is nondegenerate, then so is $\omega$. Since the curvature on a line bundle is always closed, it follows that $\omega$ is a symplectic form.

In the trivialization $\phi_s$, a lift of $a_t$, $t \in [0,1]$, to $\cL_t$ corresponds to a path $g_t \in U(1)$. The horizontal lift then satisfies
\[
 B_s(\dot a_t) = -\frac{\dot g_t}{g_t}.
\]
We have
\[
 B_s(\dot a_T) = 2\pi i \int_\Sigma s^*\la a_T\wedge \dot a_T\ra = - \frac{\eval{\frac{d}{dt}\exp\left(\pi i\int_\Sigma s^* \la a_t \wedge a_t\ra\right)}{t=T}}{\exp\left(\pi i\int_\Sigma s^* \la a_T \wedge a_T\ra\right)}.
\]
Therefore the parallel transport $\PT_s(a_t)$ along $a_t$ is given by multiplication by
\[
\exp\left(\pi i\int_\Sigma s^* \la a_1 \wedge a_1\ra\right) \cdot \exp\left(-\pi i\int_\Sigma s^* \la a_0 \wedge a_0\ra\right) = \exp\left(-2\pi i\int_0^1\left(\int_\Sigma s^* \la a_t \wedge \dot a_t\ra\right) dt\right).
\]

On the other hand, a smooth path $a_t$ in $\cA_Q$ determines a connection $A$ on the principal bundle $P = [0,1]\times Q$ over $X = [0,1]\times \Sigma$ via $A_{(t,q)}(T,X) \coloneqq (a_t)_q(X)$.
Then $A|_{\partial X} = a_1 \sqcup a_0$ is a connection on $Q\sqcup Q \to \Sigma \sqcup -\Sigma$ and $\cL_{A|_{\partial X}} \cong \overline{\cL_{a_0}} \tensor \cL_{a_1}$. Using the metric on $\cL_{a_0}$, we identify $\overline{\cL_{a_0}} \tensor \cL_{a_1}$ as the line bundle $\cL_{a_0}^* \tensor \cL_{a_1}$ or equivalently as the (linear) bundle maps $\cL_{a_0} \to \cL_{a_1}$. The Chern-Simons action of $A$ considered as a bundle map
\[
 \CS_X(A)\co \cL_{a_0} \to \cL_{a_1}
\]
has unit norm, since the corresponding element of $\overline{\cL_{a_0}} \tensor \cL_{a_1}$ has unit norm.

We will see that $\CS_X(A)$ as a bundle map is the parallel transport along $a_t$. We have
\begin{equation}\label{dAOnCylinder}
dA_{(t,q)}(T,X) = (d a_t)_q (X)- (\dot a_t)_q \wedge dt_t (T,X).
\end{equation}
Since all summands in a 3--form without a $dt$ component vanish, we have
\[
 \alpha(A) = - \la A \wedge \dot A\ra  \wedge dt.
\]
where $\dot A$ is the time derivative of $A$. Therefore $\CS_X(A)$ is multiplication by
\[
\exp\left(-2\pi i\int_{[0,1]}\left(\int_{\Sigma} s^* \la a_t \wedge \dot a_t\ra \right) dt\right).
\]
and $\CS_X(A) = \PT(a_t)$.

By functoriality the action by $\cG_Q$ on $\cA_Q$ lifts to $\cL_Q$ and preserves $\CS$ and therefore the parallel transport\kommentar{Wurde bewiesen, dass Paralleltransport und Zusammenhang equivalent sind?}, the connection $B$ as well as $\omega = \tfrac{i}{2\pi} F_B$. We compute the moment map for this action using the lift to $\cL_Q$ using the trivialization $\phi_s$.\footnote{In general, if $L \to M$ is a Hermitian line bundle (or $U(1)$--bundle) over a symplectic manifold with a connection $A$, and $\rho\co G \to \Aut(L)$ is a $G$ action on $L$ preserving the metric and $A$, then the moment map of the quotient $G$ action on $M$ is
\[
 \mu_\xi(p) = \tfrac{-i}{2\pi}\vertical(\dot\rho(\xi)_l) \in \R, \quad \xi \in \frakg, p \in M
\]
where $l \in L_p$ is a point of unit norm,
\[
\dot\rho(\xi)_l \coloneqq \eval{\tfrac{d}{dt} \rho(g_t)(l)}{t=0} \in T_lL,\quad g_0 = 1, \dot g_0 = \xi
\]
is the vector field on $L$ corresponding to $\xi \in \frakg$, and $\vertical(\cdot)$\kommentar{alternativ: $g_t = \exp(t\xi)$}
is the vertical part of a vector in $T_lL$ computed with respect to $A$. The moment map is the obstruction for the connection to descend to the quotient $L/G$.} If $\alpha = s^* a$ and $g = g_s(\Psi)$ for a gauge transformation $\Psi \co Q\to Q$, then
\[
 \phi_{\Psi s} = c_\Sigma(\alpha,g) \phi_s.
\]
and the action
\[
\rho\co \cG_Q \to \Aut(\cL_Q)
\]
for any $l\in (\cL_Q)_a$ is multiplication by $c_\Sigma(\alpha,g)$. Let $g_t$ be a path in $C^\infty(\Sigma,U(1))$ with $g_0 = 1$ and $\dot g_0 = \xi \in \Omega^0(\Sigma;\R i)$. Since
\[
g_0^*\theta = 0 \quad \text{and} \quad \eval{\tfrac{d}{dt}g_t^*\theta}{t=0} = d\xi
\]
we get by the Definition of $c_\Sigma$ in \eqref{TransitionLineBundle}
\[
\dot\rho(\xi)_l = \eval{\tfrac{d}{dt}c_\Sigma(\alpha,g_t)}{t=0} =2\pi i \int_\Sigma\la \alpha \wedge d\xi\ra \in T_l\cL_Q.
\]
On the other hand, for $f(g) =a\cdot g$ and $a_t = f(g_t)$ we have by \eqref{identifications}
\[
\dot a_0 = \eval{\tfrac{d}{dt} a g_t }{t=0} = df_1(\xi) = d_\alpha \xi.
\]
Then the infitesimal parallel transport in the direction $\xi$ is given by
\begin{align*}
\eval{\tfrac{d}{dt}\PT(a_t)(l)}{t=0} & = \eval{\frac{d}{dt}\exp\left(\pi i\int_\Sigma s^* \la a_t \wedge a_t\ra\right) \cdot \exp\left(-\pi i\int_\Sigma s^* \la a_0 \wedge a_0\ra\right)}{t=0}\\
 &= -2\pi i\int_\Sigma s^*\la a_0 \wedge \dot a_0\ra = -2\pi i\int_\Sigma \la \alpha \wedge d_\alpha\xi\ra \in T_l\cL_Q.
\end{align*}

Since the horizontal part of $\rho(\xi)_l$ equals $\eval{\tfrac{d}{dt}\PT(a_t)}{t=0}l$, we get
\[
\mu_\xi(a) = -\tfrac{i}{2\pi} \left(\dot\rho(\xi)_l - \eval{\tfrac{d}{dt}\PT(a_t)(l)}{t=0}\right)= \int_\Sigma \la \alpha \wedge (2d\xi + [\alpha \wedge \xi])\ra = 2\int_\Sigma \la F_a \wedge \xi\ra.
\]
Since the flat connections are the zeros of $\mu$ and since $\cG_Q$ acts freely on the irreducibles, the space of equivalence classes of irreducible flat connections $\cM^*_Q$ is the symplectic quotient $\cA^*_Q //\cG_Q$. The line bundle, together with its metric and connection, as well as the symplectic form, pass to a line bundle $\overline{\cL_Q} \to \cM^*_Q$ on the quotient.
\end{proof}

The above proposition can be generalized: $\cM_\Sigma$ is a stratified symplectic space (see Huebschmann \cite{huebschmann95,huebschmann95b}). The  de Rham cohomology classes of the symplectic form on $\overline{\cL_Q}$ is an integral element of $H^2(\cM_Q;\R)$, since it is the first Chern class of $\overline{\cL_Q}$. Thus, the symplectic form satisfies the integrality constraint in the geometric quantization theory.

\begin{prop}
 The image of $r_X \co \cM_X \to \cM_{\partial X}$ is a Lagrangian submanifold (on the smooth part). More precisely, the action $\CS_X$ is a flat section of the pullback bundle $r_X^* \overline{\cL_{\partial X}} \to \cM^*_{X}$, and therefore the induced symplectic form $r^*_X \omega$ vanishes.
\end{prop}

\begin{proof}
If $a_t$ is a path of flat connections, there is a corresponding connection $A$ over $I \times X$. The curvature of $A$ is $F_A = -\dot A \wedge dt$ by \eqref{dAOnCylinder} and therefore,
\begin{equation}\label{ChernClassZero}
\la F_A \wedge F_A\ra = 0.
\end{equation}
Hence by Exercise \ref{ExeCSform}, Stokes' theorem and $\PT(a_t) = \CS_{I\times \partial X}(A)$,
\[
\exp\left(2\pi i\int_{I\times X} \la F_A \wedge F_A\ra\right) = \overline{\PT(a_t)} \CS_X(a_1) \overline{\CS_X(a_0)}.
\]
Therefore $\CS$ is a flat section of the bundle $r_X^* \overline{\cL_{\partial X}} \to \cM^*_{X}$. In particular if $a_0 = a_1$, then the holonomy of $r^*_X\overline{\cL_{\partial X}}$ around $S^1$ is $\PT(a_t) = 1$ so that $r^*_X\overline{\cL_{\partial X}}$ is flat. Therefore, the symplectic form $r^*_X\omega$ vanishes.

Furthermore, Proposition \ref{PropHalfDimension} implies that the image of $(r_X)_*$ is in fact a Lagrangian subspace
\[
 \dim \im[H^1(X;\frakg_{\hol(A)}) \to H^1(\partial X;\frakg_{\hol{A}})] = \tfrac{1}{2}\dim H^1(\partial X ;\frakg_{\hol(A)}).\qedhere
\]
\end{proof}

\subsection{Quantization}

Notice, that $(\overline{\cL_\Sigma},\overline{B})$ are known as the {\em prequantization of $\cM^*_\Sigma$} respectively, where $\overline{\cL_\Sigma}$ is the Hermitian line bundle over $\cM_\Sigma$ with connection $B$ compatible with the metric such that $\tfrac{i}{2\pi}F_B = \omega$ is the symplectic form on $\cM_\Sigma$. (Similarly $(\cL_\Sigma,B)$ is the prequantization of $\cA_\Sigma$.) The {\em prequantum Hilbert space} consists of the $L^2$--sections $L^2(\overline{\cL_\Sigma})$ of $\overline{\cL_\Sigma}$, the gauge transformations lifted to an action on $\cL$ are the diffeomorphisms of $\cL$ preserving its structure (the fibration over $\cM_\Sigma$, $B$ and the Hermitian structure), which induces an action on $L^2(\overline{\cL_\Sigma})$. An element $Z$ in the Lie algebra of the group of gauge transformations is just a vector field on $\cM_\Sigma$ lifted to act on $\cL$. Acting on $L^2(\overline{\cL_\Sigma})$, this corresponds to a first-order differential operator $D_Z$ on $L^2(\overline{\cL_\Sigma})$. If $h \in C^\infty(\cM_\Sigma)$ and we choose $Z$ to be the flow of $h$ defined by $V_f \coloneqq \omega^{-1}(dh)$, this yields an assignment $f \mapsto D_{V_f}$.

The {\em geometric quantization} of $\cM_\Sigma$ is the assignment
of operators on (a subspace of) $L^2(\overline{\cL_\Sigma})$ to
functions in $C^\infty(\cM_\Sigma)$ satisfying a set of axioms.
Prequantization is the first step, which satisfies most of the
axioms. In order to further {\em quantize} $\cM_\Sigma$, we need to
first introduce a {\em polarization}, that is, an integrable
Lagrangian subbundle $\cP$ of the complexified tangent bundle
$T^{\C}_{\cM_\Sigma}$ of $M_\Sigma$, in order to cut down the
prequantum Hilbert space to all sections $s \in
L^2(\overline{\cL_\Sigma})$ satisfying $D_Z s= 0$ for all $Z \in
\cP$. This is the content of \cite{axelrod-dellapietra-witten91}.
See Blau \cite{blau} for more information on Geometric Quantization
in general.

\bibliographystyle{himpel_gtart}
\bibliography{../references}

\end{document}